\documentclass[reqno,11pt]{amsart}
\usepackage{amsmath,amscd,amsfonts,amssymb}
\usepackage{color}
\usepackage{mathtools}
\usepackage{mathrsfs}
\usepackage[colorlinks=true]{hyperref}
\usepackage{a4wide}
\usepackage[textwidth=22mm,textsize=scriptsize,color=green!20]{todonotes}

\usepackage{setspace}
\newtheorem{theorem}{Theorem}[section]
\newtheorem{proposition}[theorem]{Proposition}
\newtheorem{lemma}[theorem]{Lemma}
\newtheorem{question}{Question}

\newtheorem{corollary}[theorem]{Corollary}

\newtheorem{thmA}{Theorem}

\newtheorem{corA}[thmA]{Corollary}

\theoremstyle{definition}
\newtheorem{definition}[theorem]{Definition}
\newtheorem{conjecture}[theorem]{Conjecture}

\newtheorem{remark}[theorem]{Remark}

\begin{document}
\title{Limit distribution of algebraic integral points on curves}
\author{Binggang Qu}
\address{Binggang Qu: Instituto de Ciencias Matem\'aticas (ICMAT), Calle Nicol\'as Cabrera No. 13--15, Cantoblanco, 28049 Madrid, Spain}
\email{binggang.qu@icmat.es}

\author{Chengyuan Yang}
\address{Chengyuan Yang: Beijing International Center for Mathematical Research (BICMR), Peking University, Yiheyuan Road No. 5, Haidian District, 100871 Beijing, China}
\email{chengyuanyang020416@pku.edu.cn}

\date{September 30, 2026}

\begin{abstract}
  For a quasi-projective arithmetic surface $\mathcal U/\mathbb Z$ and a compact subset $E\subset \mathcal U(\mathbb C)$ under mild regularity assumptions, we study algebraic integral points on $\mathcal U$ whose Galois orbits lie in $E$.
  We characterize the collections of local probability measures that arise simultaneously as the local limit distributions of Galois orbits of such points.
  Our result also works for measures prescribed at a subset of places.
  This generalizes the results of Smith and Orloski--Sardari on $\mathcal{U}=\mathbb{A}^1$ concerning the archimedean place.
  As a consequence, we prove an integral version of Szachniewicz's theorem on curves, which states that every \emph{integral} GVF functional can be approximated by a sequence of algebraic integral points in the GVF topology.

  We also give several applications: we prove that the essential minimum of height functions on curves can be attained by algebraic integral points; we connect integer Chebyshev constants with the essential minima of certain height functions on $\mathbb{A}^1$ and prove a conjecture of Montgomery; we also answer affirmatively a question of Levenberg--Londhe by showing that the smallest limit of averaged trace of totally positive algebraic integers can be attained by a sequence of totally positive algebraic units.
\end{abstract}

\setcounter{tocdepth}{1}
\maketitle
\tableofcontents

\section{Introduction}
\label{sec:introduction}

Let $\mathcal U/\mathbb Z$ be a quasi-projective arithmetic surface.
We study the possible limit distributions of Galois orbits of algebraic integral points on $\mathcal U$, subject to the requirement that the orbits remain in a prescribed compact set of $\mathcal U(\mathbb C)$.
Regular functions on $\mathcal U$ impose logarithmic inequalities on these limit measures.
Our aim is to determine whether these necessary conditions are also sufficient.

On the affine line, Smith \cite{Smi24} and Orloski--Sardari \cite{OS23} established such a characterization under suitable assumptions on the compact set.
We extend this characterization to quasi-projective arithmetic surfaces and to simultaneous local distributions at arbitrary places.
In particular, we allow the distributions to be prescribed at any subset of places.
Our result generalizes the results of Smith and Orloski--Sardari and strengthens the Fekete--Szeg\H{o} theorem on curves of Rumely \cite[Theorem 1.2]{Rum13}.

We also prove stronger approximation results that control heights together with local distributions.
In the language of globally valued fields (GVFs), these give an integral version of Szachniewicz’s approximation theorem on curves: every normalized integral GVF functional can be approximated, in the GVF topology, by a sequence of distinct algebraic integral points.
This approximation result allows us to treat measures that are not compactly supported.
Applying this result to hermitian line bundles, we show that nef hermitian line bundles can be approximated by algebraic integral points.

We give three applications.
Under suitable curvature hypotheses, we show that the essential minimum of a height function on a curve can be approached by algebraic integral points.
We identify integer Chebyshev constants with the essential minima of certain height functions on the affine line and prove Montgomery’s conjecture concerning the integer Chebyshev constant of $[0,1]$.
Finally, we answer a question of Levenberg–Londhe by showing that the smallest possible limit of the average traces of totally positive algebraic integers can be approached by a sequence of totally nonnegative algebraic units.

\subsection{The archimedean case}
We first consider the distribution of Galois orbits at the archimedean place in a simplified setting.
This case already exhibits the central feature of our results: under a compact containment condition, the arithmetic constraints imposed by regular functions characterize the possible limit measures.
We work throughout this subsection with a normal quasi-projective arithmetic surface $\mathcal U/\mathbb Z$.
We also assume that the generic fiber of $\mathcal U$ is not projective.

For an algebraic point $\alpha\in\mathcal U(\overline{{\mathbb {Q}}})$, the Dirac measure of the Galois orbit $\alpha'_{\mathbb C}$ of $\alpha$ in $\mathcal U(\mathbb C)$ is the probability measure
\[
  \delta_{\alpha, \mathbb C}=\frac{1}{\deg(\alpha)}\sum\limits_{\beta\in \alpha'_{\mathbb C}}\delta_\beta.
\]
Here, $\alpha'$ is the closed point corresponding to $\alpha$ and $\alpha'_{\mathbb C}$ is the base change $\alpha'\times_{\mathbb Q}\mathbb C$.
When $\mathcal U=\mathbb{A}_{\mathbb Z}^1$, an algebraic point is just an algebraic number and its Galois orbit is the usual Galois orbit in $\mathbb C$.
We seek to characterize the weak limits of these orbit measures along sequences of distinct algebraic integral points.

Every such limit is invariant under complex conjugation.
Without a restriction on the location of the conjugates, however, weak convergence can lose arithmetic information.
A vanishing proportion of the conjugates may escape to infinity without affecting the weak limit, while still contributing substantially to logarithmic averages.
Thus, weak convergence alone need not preserve the inequalities arising from integrality.

We therefore require all Galois orbits in the approximating sequence to lie in a common compact set.
More precisely, given a compact subset $E\subset\mathcal U(\mathbb C)$, we ask which probability measures $\mu$ on $E$ can be obtained as weak limits:
\[
  \delta_{\alpha_n,\mathbb C}\longrightarrow\mu
  \quad\text{weakly},
\]
where $(\alpha_n)$ is a sequence of distinct points in $\mathcal U(\overline{\mathbb Z})$ and $\alpha'_{n,\mathbb C}\subset E$ for every $n$.
Under this compact containment condition, integrality imposes logarithmic inequalities on the limit measure.
For each non-zero regular function $Q$ on $\mathcal U$ and each algebraic integral point $\alpha$, $Q(\alpha)$ is an algebraic integer.
Hence, if $Q(\alpha)\neq 0$, we have
\[
  \int_{\mathcal U(\mathbb C)} \log\left|Q\right|\,\mathrm{d}\delta_{\alpha, \mathbb C}\geq 0.
\]
Since $\log|Q|$ is upper semicontinuous on $E$, we obtain
\[
  \int_{\mathcal U(\mathbb C)} \log\left|Q\right|\,\mathrm{d}\mu\geq \limsup_{n \to \infty} \int_{\mathcal U(\mathbb C)} \log\left|Q\right|\,\mathrm{d}\delta_{\alpha_n, \mathbb C}\geq 0.
\]
Such inequalities have been studied in \cite{Smy84, Ser19} on the affine line.

Before asking which distributions can occur, one may ask whether there are infinitely many algebraic integral points whose Galois orbits lie in the set.
This is the subject of Fekete–Szeg\H{o} theory, originating in work on algebraic integers \cite{Sch18, Fek23, Sie45, FS55, Rob64a}, extended to algebraic units \cite{Rob64b, Can80}, and developed on general quasi-projective arithmetic varieties by Rumely \cite{Rum89, Rum13}.
This theory defines a real number, called the capacity of the set, using potential theory; this number governs the existence problem.
If the set has capacity strictly greater than $1$, its neighborhood contains infinitely many Galois orbits of algebraic integral points.
Conversely, if the capacity is strictly smaller than $1$, there are only finitely many Galois orbits of algebraic integral points.

The distribution problem asks for more: it seeks to determine which probability measures can be realized by such sequences.
Smith \cite{Smi24} and Orloski--Sardari \cite{OS23} showed that on the affine line, the inequalities arising from testing regular functions give an equivalent condition.
For instance, they prove the following theorem.
\begin{theorem}[Smith, Orloski--Sardari]
  Let $E$ be a finite union of closed intervals in $\mathbb R$ and closed rectangles in $\mathbb{C}$.
  Assume that $E$ is invariant under complex conjugation and that the capacity of $E$ is strictly greater than $1$.
  Let $\mu$ be a probability measure supported on $E$ that is invariant under complex conjugation.
  Then, there exists a sequence of distinct algebraic integers $\{\alpha_n\}$ whose Galois orbits are contained in $E$ such that the Dirac measures of Galois orbits $\delta_{\alpha_n, \mathbb C}$ converge to $\mu$ weakly if and only if
  \[
    \int \log|Q| \,\mathrm{d}\mu \geq 0, \quad \text{for all $0 \neq Q \in \mathbb{Z}[x]$}.
  \]
\end{theorem}

The criterion in this theorem is expressed entirely in terms of the measure.
This suggests fixing $\mu$ and asking how closely the approximating orbits can be confined to its support.
Requiring the orbits to lie exactly in $\operatorname{supp}(\mu)$ may be too restrictive, since that support need not contain any algebraic points.
We therefore allow the orbits to lie in an arbitrarily prescribed compact neighborhood of the support.
In this formulation, the same logarithmic criterion extends to arithmetic surfaces.

\begin{thmA}\label{thm:archimedean-main}
  Let $\mathcal U/\mathbb Z$ be a normal quasi-projective arithmetic surface whose generic fiber is not projective.

  Let $\mu$ be a probability measure compactly supported in $\mathcal U(\mathbb C)$.
  Assume that $\mu$ is invariant under complex conjugation.
  Let $E$ be a compact neighborhood of the support $\operatorname{supp}(\mu)$ of $\mu$.
  Then, the following are equivalent:
  \begin{enumerate}
    \item there exists a sequence $\{\alpha_n\}$ of distinct algebraic integral points on $\mathcal U$ whose Galois orbits lie in $E$, such that the Dirac measures $\delta_{\alpha_n, \mathbb C}$ of the Galois orbits $\alpha_{n, \mathbb C}'$ of $\alpha_n$ in $\mathcal U(\mathbb C)$ converge weakly to $\mu$,
    \item \label{item:archimedean-logarithmic-condition} for each $0\neq Q\in H^0(\mathcal{U}, \mathcal{O}_{\mathcal{U}})$, we have
      \[
        \int_{\mathcal U(\mathbb{C})}\log\left|Q\right| \,\mathrm{d}\mu\geq 0.
      \]
  \end{enumerate}
\end{thmA}
It is also worth noting that our result has an $\mathbb R$-neighborhood version.
In such a case, we are able to treat totally real integers and units.
For a precise version, see Section \ref{sec:archimedean-main-proof}.

Theorem \ref{thm:archimedean-main} has no separate capacity hypothesis.
For a compact set $E$, if its capacity is strictly greater than $1$, the equilibrium measure satisfies condition~\ref{item:archimedean-logarithmic-condition}, and our theorem strengthens the neighborhood version of the Fekete--Szeg\H{o} theorem on curves of Rumely \cite[Theorem 1.2]{Rum13}.
If the capacity is exactly $1$, our criterion works beyond the Fekete--Szeg\H{o} theorem.

\subsection{The adelic case}
\label{sec:adelic-case}
We now consider the simultaneous distribution of Galois orbits at several places.
We turn to the setting of adelic sets developed by Rumely \cite{Rum89,Rum13}.

Let $K$ be a number field.
Let $C/K$ be a projective, smooth and geometrically integral curve.
Let $X \subseteq C$ be a non-empty finite subset of closed points and set $U=C\setminus X$.
Let $\mathcal{C}/O_K$ be a projective regular model of $C/K$, and $\mathcal{U} \subseteq \mathcal{C}$ an open subscheme with generic fiber $U$.
Write $M_K$ for the set of places of $K$, with subsets $M_{K,\infty}$ and $M_{K,f}$ of archimedean and non-archimedean places, respectively.
We normalize the absolute values so that
\[
  \sum_{v\in M_K}\log|a|_v=0
  \qquad\text{for every }a\in K^\times.
\]

We describe distributions in the setting of Berkovich analytification.
For each place $v\in M_K$, let $C_v^\mathrm{an}$ be the Berkovich analytification of $C_v=C\times_K K_v$.
For an archimedean place $v \in M_{K,\infty}$, the analytification $C_v^\mathrm{an}$ of $C$ at $v$ is just $C_v(\mathbb{C})$ when $K_v \cong \mathbb{C}$, and is $C_v(\mathbb{C})$ modulo complex conjugation when $K_v \cong \mathbb{R}$.
For a non-archimedean place $v \in M_{K,f}$, $C_v^\mathrm{an}$ is a compact Hausdorff space, containing each closed point of $C_v=C\times_{K}K_v$.

Denote by $k_v$ the residue field of $K_v$ and by $\mathcal{C}_{k_v}=\mathcal C\times_{\mathcal O_K} k_v$ the special fiber.
There is a reduction map $\operatorname{red}_v: C_v^\mathrm{an}\longrightarrow \mathcal{C}_{k_v}$ which reflects the reduction of a closed point in $C_v$ on the model $\mathcal C_v=\mathcal C\times_{\mathcal O_K}\mathcal O_{K_v}$.
Thus, an algebraic point on $\mathcal U$ is an algebraic integral point if and only if its Galois orbit lies in $E_v=\operatorname{red}_v^{-1}(\mathcal{U}_{k_v})$ in $U_v^\mathrm{an}$ for each non-archimedean place $v$.

We work with admissible adelic sets.
For a place $v\in M_K$, a compact set $F_v\subset U_v^\mathrm{an}$ is said to be \emph{locally admissible} if
\begin{enumerate}
  \item when $v$ is non-archimedean, there exist a projective model $\mathcal C$ of $C$ and an open subvariety $V\subset \mathcal C_{k_v}$ such that $F_v=\operatorname{red}_v^{-1}(V)$,
  \item when $K_v \cong \mathbb{C}$, $F_v$ is a finite union of closures of connected open subsets of $C_v^\mathrm{an}$,
  \item when $K_v \cong \mathbb{R}$, $F_v$ is a finite union of closures of connected open subsets of $C_v^\mathrm{an}$ and closed intervals in $C_v(\mathbb{R})$.
\end{enumerate}
We also allow $U=C$ in this local definition, and then say \emph{locally admissible in $C_v^{\mathrm{an}}$}.
An \emph{adelic set} on $U$ is a family $\{E_v\}_{v\in M_K}$ of non-empty compact sets $E_v\subset U_v^\mathrm{an}$.
An adelic set $(E_v)_{v\in M_K}$ is said to be \emph{admissible} if each $E_v$ is locally admissible and there exists a quasi-projective model $\mathcal U$ of $U$ such that $E_v=\operatorname{red}_v^{-1}(\mathcal U_{k_v})$ for every non-archimedean place $v$.
In this setting, we are able to develop Fekete--Szeg\H{o} theory.

For an algebraic point $\alpha\in\mathcal U(\overline K)$ and a place $v\in M_K$, the Galois orbit $\alpha'_{v}$ of $\alpha$ in $U_v^\mathrm{an}$ is the base change $\alpha'\times_{K}K_v$ where $\alpha'$ is the closed point corresponding to $\alpha$.
The Dirac measure of the Galois orbit $\alpha'_v$ is defined as
\[
  \delta_{\alpha, v}=\frac{1}{\deg(\alpha')}\sum\limits_{\beta\in \alpha'_{v}}\deg(\beta)\delta_\beta.
\]
For an adelic set $\mathbb E=\{E_v\}_{v\in M_K}$, we say that the algebraic point $\alpha$ is \emph{totally in} $\mathbb E$ if $\alpha'_v\subset E_v$ holds for every place $v\in M_K$.

\begin{theorem}[Rumely {\cite[Theorem 1.6]{Rum13}}]
  \label{thm:rumely-fekete-szego}
  Let $\mathbb E=(E_v)_{v\in M_K}$ be an admissible adelic set on $U$, with $E_v\subset U_v^\mathrm{an}$.
  Assume that $\operatorname{cap}_X(\mathbb E)>1$.
  Then there exist infinitely many algebraic points $\alpha$ on $U$ totally in $\mathbb E$.
\end{theorem}

Our first adelic result characterizes the possible simultaneous limit distributions of these points.
\begin{thmA}[Theorem {\ref{thm:adelic-sos-main-text}}]
  \label{thm:adelic-sos}
  Let $K$ be a number field.
  Let $C/K$ be a projective, smooth and geometrically integral curve.
  Let $X \subseteq C$ be a non-empty finite set of closed points and set $U=C\setminus X$.
  Let $\mathbb{E}$ be an admissible adelic set of $U$.
  Assume that $\operatorname{cap}_X(\mathbb{E})>1$.
  Let $\{\mu_v\}_{v \in M_K}$ be a collection of probability measures on $E_v$.
  The following are equivalent:
  \begin{enumerate}
    \item there exists a sequence of distinct algebraic points $(\alpha_n\in U(\overline{K}))$ totally in $\mathbb E$ such that for each $v\in M_K$, the Dirac measures $\delta_{\alpha_n, v}$ of the Galois orbits of $\alpha_n$ in $U_v^{\mathrm{an}}$ converge weakly to $\mu_v$,
    \item \label{item:adelic-logarithmic-condition} $\displaystyle \sum_{v \in M_K} \int \log|Q|_v \,\mathrm{d}\mu_v \geq 0$ for all $0 \neq Q \in H^0(U, \mathcal O_U)$.
  \end{enumerate}
\end{thmA}
Since we have no restrictions on probability measures, the sum $\sum_{v \in M_K} \int \log|Q|_v \,\mathrm{d}\mu_v$ might not be a finite sum.
The sum still makes sense since each integral takes either a finite value or $-\infty$ and at most finitely many of them are positive.

Our result can naturally be strengthened.
It works if one cares about the distribution in a prescribed subset of places, by replacing the integrals with logarithmic sup-norms at the remaining places.
For a set $E_v\subset U_v^\mathrm{an}$ and a rational function $Q\in K(C)$, define the sup-norm $\lVert Q\rVert_{E_v, \mathrm{sup}}\coloneqq \sup_{z\in E_v}\left|Q(z)\right|_v$.
Below is our main theorem.
\begin{thmA}\label{thm:prescribed-places}
  Let $K$ be a number field.
  Let $C/K$ be a projective, smooth and geometrically integral curve.
  Let $X \subseteq C$ be a non-empty finite set of closed points and set $U=C\setminus X$.
  Let $\mathbb{E}$ be an admissible adelic set of $U$.
  Let $S$ be a subset of $M_K$.
  Assume that $\operatorname{cap}_X(\mathbb{E})>1$.
  Let $\{\mu_v\}_{v \in S}$ be a collection of probability measures on $E_v$.
  The following are equivalent:
  \begin{enumerate}
    \item there exists a sequence of distinct algebraic points $(\alpha_n\in U(\overline{K}))$ totally in $\mathbb E$, such that for each $v\in S$, the Dirac measures of the Galois orbits of $\alpha_n$ in $U_v^{\mathrm{an}}$ converge weakly to $\mu_v$,
    \item $\displaystyle \sum_{v \in S} \int \log|Q|_v \,\mathrm{d}\mu_v + \sum_{v \notin S} \log\lVert Q\rVert_{E_v, \mathrm{sup}}  \geq 0$ for all $0 \neq Q \in H^0(U, \mathcal O_U)$.
  \end{enumerate}
\end{thmA}
We are also able to control the limit of the divisor class of $\alpha'_n/ \deg(\alpha'_n)$ on $C$; see Theorem \ref{thm:adelic-sos-fixed-class-main-text}.
Taking $S=M_K$ recovers Theorem \ref{thm:adelic-sos}.
Taking $S=M_{K, \infty}$, we recover a variant of the archimedean theorem.
Taking $K=\mathbb Q, \mathcal U=\mathbb A^1$ and $S=\{\infty\}$, we recover the result of Smith and Orloski--Sardari.

\subsection{GVF functionals and Szachniewicz's theorem}
\label{sec:existential-closedness}
In this section, we control heights together with local distributions using the language of globally valued fields (GVFs).
In Theorem \ref{thm:int-existential-closedness}, we prove an integral version of Szachniewicz’s approximation theorem on curves.
In Corollary \ref{cor:int-existential-closedness-archimedean}, we apply the theorem to hermitian line bundles, showing that nef hermitian bundles can be approximated by a sequence of algebraic integral points, in the sense of height convergence.

Keep the notation from Section \ref{sec:adelic-case}.
In the previous sections, all Galois orbits are assumed to lie in a prescribed set, and hence the limit measure is supported on that set.
However, many measures that appear naturally in arithmetic geometry are not compactly supported.
For example, on $\mathbb{A}^1$, the following question is natural in arithmetic geometry.

\begin{question}\label{ques:alg-int-equ-hei}
  Fix a probability measure $\mu$ on $\mathbb C$, stable under complex conjugation.
  Let $\alpha_n$ be a sequence of algebraic integers such that $\delta_{\alpha_n, \infty}$ converges to $\mu$ weakly.
  Can the sequence of Weil heights $h(\alpha_n)$ converge?
  If it can converge, what is the smallest possible limit?
\end{question}

In such a case, we still want our convergence to carry arithmetic information.
In the case $C=\mathbb{P}^1_{K}$, if the convergence is log convergence, then all arithmetic information can be transferred to the points.
For instance, the question has been studied in \cite[Theorem B]{BMQS26}.

\begin{theorem}[Burgos--Menares--Qu--Sombra]
  \label{thm:burgos-menares-qu-sombra}
  Let $\mu$ be a probability measure on $\mathbb C$, invariant under complex conjugation.
  Assume $\displaystyle \int \log^+|z| \,\mathrm{d}\mu < +\infty$. Then the following are equivalent:
  \begin{enumerate}
    \item there exists a sequence of distinct algebraic integers $(\alpha_n)$ such that
      \[
        \delta_{\alpha_n, \infty} \longrightarrow \mu \text{ log-weakly,}
      \]
      that is, for all positive real numbers $A$, $B$ and every continuous function $f$ on $\mathbb C$ with $\left|f(z)\right|<A+B\log^+\left|z\right|$, we have
      \[
        \lim_{n\to\infty}\int f \,\mathrm{d}\delta_{\alpha_n, \infty} = \int f \,\mathrm{d}\mu.
      \]
    \item $\displaystyle \int \log|Q| \,\mathrm{d}\mu \geq 0$ for all $0\neq Q \in \mathbb{Z}[x]$.
  \end{enumerate}
\end{theorem}

We now return to Question~\ref{ques:alg-int-equ-hei}.
We have
\[
  \liminf_{n \to \infty} h(\alpha_n)=\liminf_{n \to \infty} \int\log^+\left|z\right|\,\mathrm{d}\delta_{\alpha_n,\infty}\geq \int \log^+\left|z\right|\,\mathrm{d}\mu.
\]
Therefore, if $\int \log|Q| \,\mathrm{d}\mu \geq 0$ holds for all $0\neq Q \in \mathbb{Z}[x]$, the smallest height limit in Question~\ref{ques:alg-int-equ-hei} is $\int \log^+\left|z\right|\,\mathrm{d}\mu$.

Unfortunately, if the inequality does not hold, it is then not possible to characterize the sequence of points using log-weak convergence.
Moreover, on curves of positive genus, it turns out that the arithmetic information appears not only in the measure, but also in the geometric divisor class of the points.

These two problems dissolve naturally under the bigger framework of \emph{globally valued fields (GVFs)}. It is introduced by Ben Yaacov and Hrushovski to study model theory of global fields and has a deep connection with Arakelov geometry \cite{szachniewicz2023existentialclosednessoverlinemathbbqglobally,yaacov2024globallyvaluedfieldsfoundations}.
Let
\[
  \operatorname{\widehat{Div}}(K(C)/\mathbb{Z}) \coloneqq \varinjlim_{V \subseteq C} \operatorname{\widehat{Div}}(V/\mathbb{Z})
\]
be the space of adelic divisors on the essentially quasi-projective variety $\operatorname{Spec}(K(C))$ over $\mathbb{Z}$ as in \cite{YZ21}. This space is an $\mathbb R$-vector space.
A \emph{GVF functional} $\ell$ on $K(C)$ is a linear functional
\[
  \ell: \operatorname{\widehat{Div}}(K(C)/\mathbb{Z}) \longrightarrow \mathbb{R}
\]
that takes nonnegative values on effective adelic divisors and the value zero on principal adelic divisors.
A GVF functional $\ell$ is said to be \emph{normalized} if $\ell\circ\pi^*=\operatorname{\widehat{deg}}$ on $\operatorname{\widehat{Div}}(\mathbb{Z})$, where $\pi: \operatorname{Spec}(K(C)) \longrightarrow \operatorname{Spec}(\mathbb{Z})$ is the structure morphism.

Szachniewicz \cite[Theorem A]{szachniewicz2023existentialclosednessoverlinemathbbqglobally} showed that any normalized GVF functional can always be approximated by a generic sequence of algebraic points, in arbitrary dimension. Here we present Szachniewicz's theorem on curves.
\begin{theorem}[Szachniewicz] \label{thm:existential-closedness}
  Let $\ell$ be a normalized GVF functional on $K(C)$.
  Then there exists a sequence of distinct algebraic points $(\alpha_n)$ in $C(\overline{K})$ such that $\ell(\overline D)=\lim_{n\to\infty}h_{\overline D}(\alpha_n)$ for each $\overline D\in \operatorname{\widehat{Div}}(K(C)/\mathbb{Z})$.
\end{theorem}

We say that the sequence $(\alpha_n)$ converges to $\ell$ in the GVF topology if the above convergence holds for all $\overline D$.
The original version of Szachniewicz's theorem tests only adelic line bundles on $C$, while the convergence extends to all adelic line bundles on $\operatorname{Spec}(K(C))$ naturally; see the proof of Lemma~\ref{lem:gvf-point-convergence}.
Szachniewicz's theorem solves Question \ref{ques:alg-int-equ-hei} for algebraic numbers.
We will prove an integral version of Szachniewicz's theorem on curves, which solves Question \ref{ques:alg-int-equ-hei} for algebraic integers.

Let $\ell$ be a normalized GVF functional.
For a place $v$ and a continuous function $\varphi_v$ on $C_v^\mathrm{an}$, there is an adelic divisor $(0, \{g_w\})$, where $g_w=0$ for $w\neq v$ and $g_v=\varphi_v$.
By testing such adelic divisors and applying the Riesz representation theorem, we obtain a probability measure $\mu_v$, and GVF convergence implies local equidistribution with respect to $\mu_v$.
For an adelic set $\mathbb E=\{E_v\}$ on $C$, we say that $\ell$ is \emph{integral} with respect to $\mathbb{E}$ if $\mu_v$ is supported on $E_v$ for each $v \in M_K$.
We also have a linear functional $m$ on $\operatorname{Div}(C)_{\mathbb R}$ given by
\[
  m(D) = \ell (\overline{D}) - \sum_v \int g_{D,v} \,\mathrm{d}\mu_v,
\]
where $\overline{D}=(D, \{g_{D, v}\})$ is any adelic lift of $D$.
This functional takes nonnegative values on effective $\mathbb R$-divisors.
Below is a GVF interpretation of our main theorem.

\begin{thmA}[Theorem \ref{thm:int-existential-closedness-t-0-main-text}] \label{thm:int-existential-closedness-t-0}
  Let $K$ be a number field and $C/K$ a smooth, projective and geometrically integral curve.
  Let $X$ be a non-empty finite set of closed points of $C$, and put $U=C\setminus X$.
  Let $\mathbb{E}$ be an admissible adelic set of $U$ with $\operatorname{cap}_X(\mathbb{E})>1$.
  Let $\ell$ be a normalized GVF functional on $K(C)$ that is integral with respect to $\mathbb{E}$.
  Assume that $m(x)=0$ holds for every $x \in X$.
  Then there exists a sequence of distinct algebraic points $(\alpha_n)$ in $U(\overline{K})$, totally in $\mathbb{E}$, such that $\alpha_n \longrightarrow \ell$ in the GVF topology.
\end{thmA}

In fact, under the restriction that the sequence of points is totally in an admissible adelic set, GVF convergence of the sequence is the same as local equidistribution, together with certain asymptotic behavior of the generic divisor class.

\subsubsection{Loose admissible adelic sets}
If one wants the sequence of algebraic points to be totally in an admissible adelic set of $U$, then the assumption that $m(x)=0$ for all $x \in X$ is restrictive but necessary.
For weak convergence, this requirement ensures the arithmetic information of measures is carried to the sequence of points.
However, for GVF convergence, the arithmetic information is intrinsically carried, and we are therefore able to loosen the restrictions at some places.

Let $\widetilde{\mathbb{E}}=\{E_v\}_{v\in M_K}$ be an adelic set of $C$.
We say that $\widetilde{\mathbb{E}}$ is a \emph{loose admissible} adelic set of $U$ if
\begin{itemize}
  \item for every $x \in X$, there exists a place $v \in M_K$ such that $x_v \cap E_v \neq \varnothing$,
  \item for each place $v$, $E_v$ is locally admissible in $C_v^{\mathrm{an}}$,
  \item there exists an admissible adelic set $\mathbb E=\{E'_v\}_{v\in M_K}$ such that $E_v=E'_v$ for all but finitely many places $v$.
\end{itemize}
The finite set $S=\{v\in M_K|X_v\cap E_v \neq \varnothing\}$ is called the set of \emph{loose places}.

\begin{thmA}[Theorem \ref{thm:int-existential-closedness-main-text}] \label{thm:int-existential-closedness}
  Let $K$ be a number field and $C/K$ a smooth, projective, geometrically integral curve. Let $X$ be a non-empty finite set of closed points of $C$, and put $U=C\setminus X$.
  Assume that $\widetilde{\mathbb{E}}$ is a loose admissible adelic set of $U$.
  Let $\ell$ be a normalized GVF functional on $K(C)$ that is integral with respect to $\widetilde{\mathbb{E}}$. Then there exists a sequence of distinct algebraic points $(\alpha_n)$ in $U(\overline{K})$, integral with respect to $\widetilde{\mathbb{E}}$, such that $\alpha_n \longrightarrow \ell$ in the GVF topology.
\end{thmA}

We can let loose all archimedean places.
Set $E_v=C_v^{\mathrm{an}}$ for $v$ archimedean and $E_v=\operatorname{red}_v^{-1}(\mathcal U_{k_v})$ for $v$ non-archimedean.
With $\mathcal U=\mathbb{A}^1$, this choice of $E_v$ recovers Theorem \ref{thm:burgos-menares-qu-sombra}.

For a nef hermitian line bundle $\overline{\mathcal{L}}=(\mathcal{L}, \lVert\cdot\rVert)$ on $\mathcal C$ with $\deg(\mathcal{L}_K)\neq 0$, the intersection numbers
\[
  \ell_{\overline{\mathcal L}}(\overline{D})=\deg(\mathcal{L}_K)^{-1}\overline{\mathcal{L}}\cdot \overline{D}
\]
give a GVF functional.
The condition that $\ell_{\overline{\mathcal L}}$ is integral with respect to $\widetilde{\mathbb{E}}=(E_v)_{v\in M_K}$ is the same as $\deg_G\mathcal L=0$ for each irreducible component $G$ of $\mathcal C\setminus\mathcal U$.
Therefore, we deduce the following corollary.
\begin{corA} \label{cor:int-existential-closedness-archimedean}
  Let $\overline{\mathcal{L}}=(\mathcal{L}, \lVert\cdot\rVert)$ be a nef hermitian line bundle on $\mathcal C$ with $\deg(\mathcal{L}_K)\neq 0$.
  Assume that $\deg_G\mathcal L=0$ for each irreducible component $G$ of $\mathcal C\setminus\mathcal U$.
  Then there exists a sequence of distinct algebraic integral points $(\alpha_n)$ on $\mathcal{U}$ such that for each hermitian line bundle $\overline{\mathcal{M}}$, $\lim_{n\to\infty} h_{\overline{\mathcal{M}}}(\alpha_n)=\deg(\mathcal{L}_K)^{-1}\overline{\mathcal{L}}\cdot \overline{\mathcal{M}}$.
\end{corA}

\subsection{The essential minimum of height functions}
Keep the notation in Section \ref{sec:adelic-case}.
Let $\overline{L}=\left( L,\{\|\cdot\|_v\} \right)$ be an adelic line bundle on $\operatorname{Spec} K(C)$ with geometric part $\widetilde{L}$, and assume $\deg(\widetilde{L})>0$.

The \emph{essential minimum} $\operatorname{ess}(\overline{L})$ of $\overline{L}$ is defined as
\[
  \operatorname{ess}(\overline{L}) = \inf \left\{ \liminf_{n\rightarrow\infty} h_{\overline{L}}(\alpha_n): (\alpha_n) \text{ is a sequence of distinct algebraic points in $U(\overline{K})$} \right\}.
\]
It has been the protagonist of the Bogomolov conjecture along with the development of Arakelov theory.

For adelic line bundles on projective varieties, building on the arithmetic BDPP theorem of Ikoma \cite[Theorem 6.4]{Ikoma15}, Balla\"y \cite{Bal21} characterizes the essential minimum using sections, and Szachniewicz \cite[Theorem D]{szachniewicz2023existentialclosednessoverlinemathbbqglobally} characterizes it using normalized GVF functionals.

Using the argument of Szachniewicz and applying balayage, we prove the following.
\begin{thmA}[Theorem \ref{thm:ess-min-attained-by-integers-main-text}] \label{thm:ess-min-attained-by-integers}
  Let $K$ be a number field and $C/K$ a smooth, projective, geometrically integral curve. Let $X$ be a non-empty finite set of closed points of $C$, and put $U=C\setminus X$.
  Let $\overline L$ be an adelic line bundle on $\operatorname{Spec} K(C)$ whose geometric part $\widetilde L$ has positive degree.
  Let $\mathbb{E}=(E_v)_{v\in M_K}$ be a loose admissible adelic set on $U$. Assume $c_1(\overline{L})_v \leq 0$ outside $E_v$ at every place $v \in M_K$. Then there exists a sequence of distinct algebraic points $(\alpha_n)$ in $U(\overline{K})$, integral with respect to $\mathbb{E}$, such that $h_{\overline{L}}(\alpha_n) \longrightarrow \operatorname{ess}(\overline{L})$.
\end{thmA}

\subsection{The integer Chebyshev problem and a conjecture of Montgomery}
Let $E \subseteq \mathbb{C}$ be an admissible compact set invariant under complex conjugation.
The \emph{classical Chebyshev problem} asks for the asymptotic behavior of the best approximation of the constant function $0$ by monic polynomials on $E$ in the $L^\infty$ norm.
To be precise, let
\[
  t_n(E) = \inf \Big\{ \lVert f\rVert_{E,\mathrm{sup}}: f \in \mathbb{C}[x] \text{ is monic}, \deg(f)=n \Big\},
\] where $\lVert f\rVert_{E,\mathrm{sup}} = \sup_{x \in E} |f(x)|$. We want to determine the limit $\displaystyle t(E) = \lim_{n \rightarrow \infty} t_n(E)^{1/n}$, which is called the \emph{Chebyshev constant} of $E$.

It is a great achievement of potential theory, proved by Fekete and Szeg\H{o}, that $t(E)=\operatorname{cap}(E)$ \cite{KIRSCH05}.
Since we have methods to compute the capacity of many kinds of compact sets \cite[Section 5.2]{Ransford}, this also answers the Chebyshev problem for these kinds of sets.

The \emph{integral Chebyshev problem} is a variant of the classical Chebyshev problem.
We work only with polynomials with integer coefficients rather than with polynomials with complex coefficients.
Let
\[
  t_{\mathbb{Z},n}(E) = \inf \Big\{ \lVert f\rVert_{E,\mathrm{sup}}: 0\neq f \in \mathbb{Z}[x], \deg(f) \leq n \Big\}.
\]
We want to determine the \emph{integer Chebyshev constant} $\displaystyle t_\mathbb{Z}(E) \coloneqq \lim_{n \rightarrow \infty} t_{\mathbb{Z},n}(E)^{1/n}$ of $E$.

The integral Chebyshev problem is trivial if $\operatorname{cap}(E) \geq 1$: in this circumstance, nobody beats the constant integer polynomial $1$ and $t_{\mathbb{Z}}(E)=1$.
In the case $\operatorname{cap}(E) < 1$, by contrast, the integer Chebyshev problem is very difficult, and a fundamental result of Hilbert--Fekete gives the bounds
\(
\operatorname{cap}(E)\leq t_{\mathbb Z}(E)\leq \operatorname{cap}(E)^{1/2}
\).
Thus the integer Chebyshev constant is still bounded from both sides by the capacity, but it contains genuinely arithmetic information beyond the potential theory of $E$ on the complex plane.
Gelfond and Schnirelman conjectured $t_\mathbb{Z}([0,1])=e^{-1}$ and proved that $t_\mathbb{Z}([0,1])=e^{-1}$ implies the Prime Number Theorem (see \cite[Chapter 10]{Mon94}), but the conjectured equality is false.

Let $\mathcal{F}_{[0,1]}(n) \subseteq \mathbb{Z}[x]$ be the set of irreducible polynomials of degree $n$ with integer coefficients and with all their roots in $[0,1]$. Let
\[
  s \coloneqq \liminf_{n \rightarrow \infty} \inf_{f_n \in \mathcal{F}_{[0,1]}(n)} \left|c_n\right|^{1/n},
\]
where $c_n$ is the leading coefficient of $f_n$. It is not difficult to show that $t_\mathbb{Z}([0,1]) \geq s^{-1}$ by \cite[Proposition 1.7]{Prit05}.
Using this method, Gorshkov showed that $t_\mathbb{Z}([0,1])>e^{-1}$, disproving the conjecture of Gelfond--Schnirelman \cite{Gorshkov1959}.

Montgomery conjectured that one can in fact reach $t_\mathbb{Z}([0,1])$ through this method {\cite[Chapter 10, after Theorem 2]{Mon94}}.
\begin{conjecture}[Montgomery] \label{conj:montgomery}
  $t_\mathbb{Z}([0,1]) = s^{-1}$.
\end{conjecture}

To attack Montgomery's conjecture, we first present Theorem \ref{thm:ess-min-and-int-chebyshev-constant}, which says that integer Chebyshev constants are nothing but the essential minima of particular height functions on $\mathbb{A}^1$ coming from potential theory.
It is a direct consequence of Ballaÿ's theorem \cite[Theorem 1.2]{Bal21} and Bernstein's lemma \cite[Theorem 5.5.7]{Ransford}.

Let $g_E:=g_{E,\infty}$ be the Green function of $E$ with pole at $\infty$ (normalized so that $g_E=0$ on $E$). Consider the height function
\begin{align}
  h_E: \overline{\mathbb{Q}} &\longrightarrow \mathbb{R}, \notag\\
  \alpha &\longmapsto \frac1{\deg(\alpha)} \left(\log|a_\alpha| + \sum_{\beta \in \alpha'_{\infty}} g_E(\beta) \right). \label{eq:height-e}
\end{align} where $P_\alpha\in\mathbb{Z}[x]$ is the primitive minimal polynomial of $\alpha$ with positive leading coefficient $a_\alpha$.

\begin{theorem} \label{thm:ess-min-and-int-chebyshev-constant}
  Let $E\subseteq\mathbb C$ be an admissible compact set invariant under complex conjugation, and let $h_E$ be the height defined in \eqref{eq:height-e}. Then
  $\operatorname{ess}(h_E)=-\log t_\mathbb{Z}(E)$.
\end{theorem}

From this Arakelov-theoretic point of view, the inequality of Hilbert--Fekete becomes
\[
  -\frac12 \log \operatorname{cap}(E) \leq \operatorname{ess}(h_E) \leq -\log \operatorname{cap}(E).
\]
It is straightforward to compute $h_E(\mathbb{P}^1)=-\frac12 \log \operatorname{cap}(E)$. Thus the first inequality is exactly Zhang's fundamental inequality bounding the essential minimum from below by the height of the variety.

Let $\alpha$ be an algebraic number whose Galois orbit is totally contained in $E$. Since $g_E=0$ on $E$, we have
\[
  h_E(\alpha) = \frac1{\deg(P_\alpha)} \log|a_\alpha|,
\] where $a_\alpha$ is the leading coefficient of the minimal polynomial $P_\alpha \in \mathbb{Z}[x]$ of $\alpha$. Applying Theorem \ref{thm:ess-min-and-int-chebyshev-constant} to $E=[0,1]$, we see that Conjecture \ref{conj:montgomery} says precisely that the essential minimum of $h_{[0,1]}$ can be attained by algebraic numbers totally in $[0,1]$, which follows immediately from Theorem \ref{thm:ess-min-attained-by-integers}.

\begin{thmA}[Theorem \ref{thm:montgomery-conj-is-true-main-text}] \label{thm:montgomery-conj-is-true}
  Montgomery's conjecture (Conjecture~\ref{conj:montgomery}) is true: $t_{\mathbb Z}([0,1])=s^{-1}$.
\end{thmA}

\subsection{The Schur--Siegel--Smyth trace problem and a question of Levenberg--Londhe}
The \emph{Schur--Siegel--Smyth trace problem} asks for the smallest possible limit of the average traces of totally positive algebraic integers.
To be precise, let $\overline{\mathbb{Z}}_{\mathrm{tp}} = \big\{ \alpha \in \overline{\mathbb{Z}}: \alpha'_{\infty} \subseteq \mathbb{R}_{\geq 0} \big\}$ be the set of totally nonnegative algebraic integers. In particular, this set includes $0$.
We want to determine
\[
  \lambda_{\mathrm{SSS}} = \inf \left\{ \liminf_{n\rightarrow\infty} \frac{\operatorname{tr}(\alpha_n)}{\deg(\alpha_n)}: (\alpha_n) \text{ is a sequence of distinct elements in $\overline{\mathbb{Z}}_{\mathrm{tp}}$ } \right\}.
\]
It was widely believed that $\lambda_{\mathrm{SSS}}=2$, until the recent breakthrough of Smith proved that $\lambda_{\mathrm{SSS}}<1.89831$ \cite[Theorem 1.1]{Smi24}.

Recently in \cite[Question 3.7(1)]{LL25}, Levenberg and Londhe ask whether $\lambda_{\mathrm{SSS}}$ can be achieved by the averaged traces of a sequence of algebraic units.
We answer their question affirmatively, and in fact we prove a much more general result as follows.
Let $X_{\mathrm{SSS}}$ be the exceptional set for the Schur--Siegel--Smyth trace problem, i.e.
\[
  X_{\mathrm{SSS}} := \left\{ x \in \overline{\mathbb{Z}}_{\mathrm{tp}}: \frac{\operatorname{tr}(x)}{\deg(x)} < \lambda_{\mathrm{SSS}} \right\}.
\]
We must remark that we do not know whether $X_{\mathrm{SSS}}$ is finite or infinite.

\begin{thmA}[Theorem \ref{thm:levenberg-question-main-text}] \label{thm:levenberg-question}
  Let $X \subseteq X_{\mathrm{SSS}}$ be a finite subset.
  Let $X'$ be the set of closed points in $\mathbb A^1_{\mathbb Q}$ corresponding to algebraic integers in $X$.
  Let $\overline{X'}_{\mathrm{zar}}$ be the Zariski closure of $X'$ in $\mathbb A^1$ corresponding to the elements of $X$, and set $\mathcal U_X=\mathbb A^1\setminus\overline{X'}_{\mathrm{zar}}$.
  Then there exists a sequence of distinct points $(\alpha_n)$ in $\mathcal{U}_X(\overline{\mathbb{Z}})\cap\overline{\mathbb{Z}}_{\mathrm{tp}}$ such that $ \frac{\operatorname{tr}(\alpha_n)}{\deg(\alpha_n)} \longrightarrow \lambda_{\mathrm{SSS}}$.
\end{thmA}

Taking $X=\{0\}$, we answer \cite[Question 3.7(1)]{LL25}.
Unfortunately, this does not seem to help determine the exact value of $\lambda_{\mathrm{SSS}}$.

\subsection{Notation and terminology}

We will use the terminology of Yuan--Zhang \cite{YZ21}, and we will review the relevant terminology before using it.
The following is some additional notation used in this paper.

\begin{enumerate}
  \item For a real number $t$, define $\log^+t = \log\max\{t,1\}.$

  \item By a \emph{closed interval}, we mean a closed subset homeomorphic to $[0, 1]$.
    By an \emph{open interval}, we mean the interior of a closed interval.

  \item
    By a \emph{variety}, we mean an integral scheme separated and of finite type over the base field.
    By a \emph{curve}, we mean a variety of dimension 1.

  \item
    For a number field $K$, denote by $M_K$, $M_{K, f}$ and $M_{K, \infty}$ the sets of all places, finite places and infinite places of $K$, respectively.
    We always normalize the absolute values so that
    \[
      \sum_{v\in M_K}\log|a|_v=0
      \qquad\text{for every }a\in K^\times.
    \]

  \item
    For a field $K$, denote by $\overline K$ an algebraic closure of $K$.

  \item
    By a \emph{projective arithmetic surface}, we mean a projective, integral and flat scheme over $\operatorname{Spec}(\mathbb Z)$ of absolute dimension $2$.
    By a \emph{quasi-projective arithmetic surface}, we mean a non-empty open subscheme of a projective arithmetic surface.

  \item
    For topological spaces $V\subset W$, denote by $\operatorname{int}_WV$ the interior of $V$ in $W$.

\end{enumerate}

\subsection*{AI disclosure statement}

This project began in September 2024, when Xinyi Yuan introduced the results in \cite{Smi24,OS23} to Binggang Qu and suggested that they might be connected to Arakelov theory. Binggang started to work on it and had the idea that it might be possible to generalize the results from $\mathbb{A}^1$ to arbitrary open curves.
Chengyuan Yang joined the project during the period.

In August 2025, Chengyuan found a way to generalize Smith's result to algebraic units by changing only the global construction step, and told Binggang about the idea.
In September 2025, Binggang moved to Madrid and stopped working on this project.
During the period from September to December 2025, Chengyuan found out that the condition in Theorem \ref{thm:integral-point-construction} is an equivalent condition, and completed all the ideas needed to prove the theorem.
In this process, Chengyuan used GPTs (here and below, generative pre-trained decoder-only autoregressive Transformer language models, not limited to ChatGPT), especially DeepSeek, for studying well-known facts in potential theory and capacity theory.
Then, Chengyuan wrote a manuscript, containing the proof of Theorem \ref{thm:integral-point-construction} (Sections 3, 4, 5, and 6 of this article) and the applications to the essential minimum, the integer Chebyshev problem and the Schur--Siegel--Smyth trace problem.
In this process, Chengyuan used GPTs, especially DeepSeek, for improving the language, checking for typos and checking the correctness of proofs.

In April 2026, Chengyuan sent his manuscript to Binggang (and many experts in the field). Binggang had the following ideas:
\begin{itemize}
  \item Theorem \ref{thm:integral-point-construction} is in fact an integral version of \cite[Theorem A]{szachniewicz2023existentialclosednessoverlinemathbbqglobally}.
    At that time, Binggang and José Ignacio Burgos Gil were trying to formulate a higher-dimensional generalization of the Smith--Orloski--Sardari results, in which one can approximate an integral ``movable curve'' by a sequence of integral points eventually avoiding every proper closed subset. The problem is what a ``movable curve'' is in the context of Arakelov theory. The Gillet--Soulé theory of arithmetic Chow groups only allows smooth currents, and that is indeed quite disappointing.
    Then the theory of GVF functionals came to mind: a ``movable curve'' in the context of Arakelov theory is a GVF functional.
  \item One needs only to test principal divisors (rather than flat divisors as in Theorem \ref{thm:integral-point-construction}) to generalize the Smith--Orloski--Sardari results to general open curves. The business of approximating measures by points can be viewed as the ``dual'' of Ballaÿ's theorem \cite[Theorem 1.2]{Bal21} in the sense of linear programming (see \cite{BMQS26} in the case of $\mathbb{A}^1$). And one needs only to test rational functions in Ballaÿ's theorem.
\end{itemize}
After several rounds of discussions with ChatGPT, Binggang worked out the details of the two ideas and he started to rewrite Chengyuan's manuscript under the new framework of GVF functionals in June 2026.
During this process, ChatGPT provided huge assistance with proof strategies and computations.
It is also worth noticing that the main technical difficulty in proving the GVF theorem had already been resolved in Chengyuan's manuscript in the application to the essential minimum.
Binggang used ChatGPT mainly for translation into the language of GVFs.

After finishing the draft, the authors also used ChatGPT to fix typos, check grammar and audit the proofs.

\subsection*{Acknowledgments}
We thank Jos\'e Ignacio Burgos Gil, Nuno Hultberg, Norm Levenberg, Mayuresh Londhe, Wenbin Luo, Ricardo Menares, and Mart\'in Sombra for helpful discussions.

The connection between the results in \cite{Smi24,OS23} and Arakelov theory was pointed out to the authors by Xinyi Yuan in his Arakelov theory seminar at BICMR, Peking University. We thank him especially.

The connection between the results in \cite{Smi24,OS23} and GVF functionals was pointed out to the first author by Micha{\l} Szachniewicz during the conference ``Diophantine Days 2025'' at ITS, Westlake University. We thank him especially, as well as Paolo Dolce for organizing the conference.
\section{Preliminaries}
\label{sec:preliminaries}
In this section, we introduce the Berkovich spaces, potential theory, capacity theory and Arakelov theory.
In Section \ref{sec:potential-theory}, we also prove some technical lemmas for Cantor's matrix.

\subsection{Berkovich spaces}
Let $K$ be a valued field and $C$ a projective, smooth and geometrically integral curve over $K$.
Berkovich constructed an arcwise connected, compact and Hausdorff space $C^{\mathrm{an}}$, defined as follows.
Let $U=\operatorname{Spec}(A)$ be an affine open subscheme of $C$.
Define $U^{\mathrm{an}}$ to be the set of multiplicative seminorms $\lvert\cdot\rvert_x : A\longrightarrow \mathbb{R}_{\geq 0}$ extending the absolute value on $K$, endowed with the coarsest topology for which $x\longmapsto\lvert f\rvert_x$ is continuous for every $f\in A$.
The space $C^{\mathrm{an}}$ is defined by taking an affine cover $\{U_i\}_{i\in I}$ of $C$ and gluing $\{U_i^{\mathrm{an}}\}_{i\in I}$ naturally.
For $K={\mathbb {C}}$, $C^{\mathrm{an}}$ is the usual complex analytification of $C$.
For $K={\mathbb {R}}$, $C^{\mathrm{an}}$ is the complex analytification modulo complex conjugation.

Assume the valuation of $K$ is non-trivial, non-archimedean and discrete.
Let $R$ be the valuation ring of $K$ and $k_s$ the residue field.
By a \emph{projective model} of $C$, we mean an integral, flat, regular and projective scheme over $\operatorname{Spec}(R)$ whose generic fiber is $C$.
Let ${\mathcal {C}}$ be a projective model of $C$.
Let ${\mathcal {C}}_{k_s}$ be the special fiber of ${\mathcal {C}}$.
There is a reduction map ${\mathrm{red}}: C^{\mathrm{an}}\longrightarrow{\mathcal {C}}_{k_s}$ defined as follows. For any $x\in C^{\mathrm{an}}$, let $H_x$ be the residue field of $x$ and let $R_x$ be the
valuation ring of $H_x$. By the valuative criterion of properness, the $K$-morphism $\operatorname{Spec}(H_x)\longrightarrow C$ extends uniquely to a $\operatorname{Spec}(R)$-morphism $\operatorname{Spec}(R_x)\longrightarrow {\mathcal {C}}$.
Define ${\mathrm{red}}(x)$ to be the image of the closed point of $\operatorname{Spec}(R_x)$ under this morphism, which is a point in ${\mathcal {C}}_{k_s}$.
Let $\eta$ be a generic point of an irreducible component of ${\mathcal {C}}_{k_s}$.
Then the inverse image ${\mathrm{red}}^{-1}(\eta)$ is a unique point $\xi_{\eta}$, called the \emph{Shilov point} of $\eta$, or the \emph{divisorial point} of the divisor $\overline{\eta}$.

Assume that $({\mathcal C}_{k_s})_{\mathrm{red}}$ is a divisor with strict
normal crossings on ${\mathcal C}$.
We call ${\mathcal C}$ an \emph{sncd-model}, as in \cite[\S 2.2.5]{MN15}, if the special fiber $({\mathcal C}_{k_s})_{\mathrm{red}}$ is a divisor with strict
normal crossings on ${\mathcal C}$.
Note that the special fiber ${\mathcal C}_{k_s}$ is not required to be reduced.
We only require that the reduced special fiber contain finitely many regular curves with normal crossings.
The \emph{dual graph} $\Gamma_{{\mathcal C}}$ is the graph whose
vertices are the irreducible components of $({\mathcal C}_{k_s})_{\mathrm{red}}$, and the edges connecting two irreducible components correspond to their intersection points.
The dual graph can be naturally embedded in the Berkovich space $C^{\mathrm{an}}$ as the skeleton of ${\mathcal C}$.
In this embedding, an irreducible component of ${\mathcal C}_{k_s}$ maps to the Shilov point of its generic point.
Also, there is a continuous retraction $\tau :C^{\mathrm{an}}\longrightarrow \Gamma_{{\mathcal C}}$; see \cite[Proposition 3.1.4 and \S 3.1.5]{MN15}.
The map $\tau$ is compatible with the reduction map in the sense that if ${\mathrm{red}}(x)$ is a regular point of the reduced special fiber, $\tau(x)$ corresponds to the irreducible component containing $\operatorname{red}(x)$; if ${\mathrm{red}}(x)$ is a double point of the reduced special fiber, $\tau(x)$ lies in the edge corresponding to this double point.

\subsection{Potential theory on algebraic curves}\label{sec:potential-theory}
\subsubsection{Local theory}
In this section, we briefly introduce both the potential theory and the capacity theory on $C^{\mathrm{an}}$.
For details, we refer to Rumely \cite{Rum89} and Thuillier \cite{Thu05}.

Let $K$ be a local field and $C$ a projective, smooth and geometrically integral curve over $K$.
For a divisor $D$ on $C$, a Green function $g_D$ of $D$ is a continuous function on $C^{\mathrm{an}}\setminus {\mathrm{supp}}(D)$ such that for each $x\in {\mathrm{supp}}(D)$, and each rational function $f$ such that ${\mathrm{div}}(f)$ coincides with $D$ in an open neighborhood of $x$, the function $g_D+\log\lvert f\rvert$ has a removable discontinuity at $x$.
We write $c_1(g_D):=\mathrm{dd}^c g_D+\delta_D$ for its curvature.
For $D=\sum_z a_z z$, put $\delta_D=\sum_z a_z\deg(z)\delta_z$, where $\delta_z$ is the unit Dirac measure at $z$.
Let $x$ be a closed point in $C$ and $H$ a compact subset of $C^{\mathrm{an}}\setminus \{x\}$.
A Green function $g_x$ at $x$ is a Green function of $x$ with $x$ viewed as a prime divisor.
Fix a Green function $g_x$ at $x$.
For any probability measure $\mu$ with compact support on $C^{\mathrm{an}}\setminus \{x\}$, there is a unique subharmonic function on $C^{\mathrm{an}}\setminus \{x\}$, called the \emph{potential} of $\mu$ with respect to $x$, such that $\mathrm{dd}^cU_\mu=\mu$ on $C^{\mathrm{an}}\setminus \{x\}$ and $U_{\mu}-g_x/\deg(x)$ extends continuously to $x$ with value $0$.
The function $G_x(x_1, x_2)$ is defined as $U_{\delta_{x_1}}(x_2)$.
This function is symmetric, and for the archimedean place, its exponential $d_x(x_1, x_2)=e^{G_x(x_1, x_2)}$ is a distance function on $C^{\mathrm{an}}\setminus \{x\}$, satisfying a weak triangle inequality.

Let $\mathscr{P}(H)$ be the set of all probability measures on $H$.
The \emph{energy integral} of $\mu$ with respect to $x$ is defined as
\begin{align*}
  I_x(\mu):&=\int_{H}U_{\mu}\,\mathrm{d}\mu\\
  &=\int_{H}\int_{H}\log d_x(x_1, x_2)\,\mathrm{d}\mu(x_1)\,\mathrm{d}\mu(x_2).
\end{align*}
The \emph{logarithmic capacity} of $H$ with respect to $x$ is defined as
$$V_x(H):=-\sup\limits_{\mu\in \mathscr{P}(H)}I_x(\mu), $$
and the \emph{capacity} is defined as
$$\operatorname{cap}_x(H):=e^{-V_x(H)}.$$
A \emph{polar set} is a compact set of capacity $0$.
Note that this does not depend on the choice of $x$.
\begin{proposition}
  Assume that $H$ is not polar.
  Then the supremum is attained by a unique probability measure $\mu_{H, x}$.
  Moreover, its potential $U_{H, x}$ attains its infimum $-V_x(H)$ on all of $H$ except on a polar set.
\end{proposition}
Note that $\mu_{H, x}$ is supported on $\partial H$.
For the archimedean case, this was proved by Rumely \cite[\S 3.1]{Rum89}.
For the non-archimedean case, this was proved by Thuillier \cite[\S 3.6]{Thu05}.
Note that Thuillier's original proof treated the case when $x$ is a point of type II or III.
By contrast, one may easily deduce our case by separating $C^{\mathrm{an}}$ by a type II point near $x$, which makes all of $H$ lie in a connected component not containing $x$.
This does not change the equilibrium measure.

Assume that $H$ is not polar.
The measure $\mu_{H, x}$ is called the \emph{equilibrium measure} of $H$ with respect to the point $x$.
The function $g_{H, x}=U_{H, x}+V_x(H)$ is called the \emph{equilibrium potential} of $H$ with respect to the point $x$.
It is known that $g_{H, x}\geq 0$, and
\[
  g_{H, x}=0 \text{ on } H \text{ except for a polar set}.
\]
Moreover, $g_{H, x}$ is continuous outside this polar set.
For the archimedean case, if $H$ has finitely many connected components, then $g_{H, x}$ is known to be continuous.
For the non-archimedean case, $U_{H, x}$ is continuous as long as $H$ is locally admissible in $C^{\mathrm{an}}$.
For each $y\notin H, y\neq x$, the Green function of $H$ is defined as $G(x,y;H)=g_{H, x}(y)$.
Note that this does not depend on the choice of the Green function at $x$.

For a probability measure $\mu$ supported on $C^{\mathrm{an}}\setminus H$, let
$$
  \mu^H=\int_{y\in C^{\mathrm{an}}\setminus H}\mu_{H, y}\,\mathrm{d}\mu
$$
be the balayage of $\mu$ onto $\partial H$.
For simplicity, we assume that $\partial H$ does not contain any classical points if $K$ is non-archimedean.
Then we have the following inequality.
\begin{proposition} \label{prop:non-sub-avg}
  Let $\Omega=C^{\mathrm{an}}\setminus H$.
  For every function $f:\overline{\Omega}\longrightarrow{\mathbb {R}}\bigcup\{-\infty\}$ subharmonic on $\Omega$ and continuous on $\partial \Omega$, the inequality
  $$
    \int_{\Omega}f\,\mathrm{d} \mu\leq \int_{\partial \Omega}f\,\mathrm{d} \mu^H
  $$
  holds.
\end{proposition}
This is a classical result of balayage when $K$ is archimedean.
For example, see \cite[Satz 4.5]{CC63}.
For non-archimedean $K$, this is a consequence of the Dirichlet problem developed in \cite[\S3.1]{Thu05}, as follows.
\begin{proof}
  We prove the case when $K$ is non-archimedean.
  By \cite[Propositions 3.1.16, 3.1.19, 3.1.20]{Thu05}, there exists a function $P_f:\overline{\Omega}\longrightarrow{\mathbb {R}}$ harmonic on $\Omega$ and continuous on $\overline{\Omega}$, such that $P_f$ and $f$ coincide on $\partial \Omega$.
  By the maximum principle, $f\leq P_f$ on $\Omega$.
  Therefore, we have
  \begin{align*}
    \int_{\Omega}f\,\mathrm{d} \mu&\leq \int_{\Omega}P_f\,\mathrm{d} \mu
    = \int_{\partial \Omega}f\,\mathrm{d} \mu^H.
  \end{align*}
\end{proof}
For a probability measure $\mu$ on $C^{\mathrm{an}}$, we extend the balayage by changing only the measure outside $H$, that is, for $\mu(H)= 1$, $\mu^H=\mu$, and for $\mu(H)\neq 1$,
\[
  \mu^{H}=\mu|_{H}+(1-\mu(H))\left(\frac{\mu-\mu|_H}{1-\mu(H)}\right)^H.
\]
Let the \emph{balayage defect} be the function $u_\mu^H$ such that $\mathrm{dd}^c u_\mu^H=\mu^H-\mu$, normalized such that $u_\mu^H=0$ on $H$ except for a polar set.
Note that $u_\mu^H\geq 0$.

The following generalization of Rouch\'e's theorem over non-archimedean fields is crucial for our proof.
\begin{proposition}\label{prop:rou-thm}
  Assume that $K$ is non-archimedean.
  Let $f_1$ and $f_2$ be rational functions on $C$, with $f_1$ non-zero.
  Let $H$ be a strictly analytic domain containing no poles of $f_1$ and $f_2$.
  Assume that $\left|f_1\right|>\left|f_2\right|$ on $\partial H$.
  Then the number of zeros of $f_1$ in $H$, counted with multiplicity, is the same as that of $f_1+f_2$.
\end{proposition}
\begin{proof}
  By the Poincar\'e--Lelong formula \cite[Proposition 3.3.15]{Thu05}, $\mathrm{dd}^c \log\left|f_1\right|=\delta_{{\mathrm{div}}(f_1)}$.
  Therefore, its mass on $H$ is the number of zeros of $f_1$ in $H$.
  Moreover, we have $\log\left|f_1\right|=\log\left|f_1+f_2\right|$ in a neighborhood of $\partial H$.
  Hence, the masses of $\mathrm{dd}^c \log\left|f_1\right|$ and $\mathrm{dd}^c \log\left|f_1+f_2\right|$ are the same on $H$.
  Applying the Poincar\'e--Lelong formula again to $f_1+f_2$, we get the proposition.
\end{proof}
\subsubsection{Global theory}
In this section, we introduce the global capacity theory on curves, including some useful technical lemmas.

Let $K$ be a number field and $C$ a projective, smooth and geometrically integral curve over $K$.
For a point $x\in C$, a place $v$ and a non-polar compact set $E_v$ in $C_v^{\mathrm{an}}\setminus x_v^{\mathrm{an}}$, let the \emph{equilibrium measure} of $E_v$ with respect to $x$ be the probability measure
$$
  \mu_{E_v, x_v}=\sum\limits_{x'\in x_v}\frac{\deg(x')}{\deg(x)}\mu_{E_v, x'}
$$
and the \emph{equilibrium potential}
$$
  g_{E_v, x_v}=\sum\limits_{x'\in x_v}\frac{\deg(x')}{\deg(x)}g_{E_v, x'}.
$$

Let ${X}$ be a non-empty finite set of closed points of $C$.
Set $U=C\setminus{X}$.
Let $\mathbb{E}=(E_v)_{v\in M_K}$ be an admissible adelic set on $U$.
For each $x\in {X}$, we take a rational function $f$ on $C$ with a simple pole at $x$, and for each $v\in M_K$ and $x'\in x_v$ we take a Green function $g_{x'}$ of $x'$ on $C_v^{\mathrm{an}}$ which equals $\log\left|f\right|_v$ in a neighborhood of $x'$.
For $x\neq y$, set
$$
  c_{x, y, v}=\sum\limits_{y'\in y_v}\frac{\deg(y')}{\deg(y)}g_{E_v, x_v}(y').
$$
Set
$$
  c_{x, x, v}=\sum\limits_{x'\in x_v}\frac1{\deg(x)}\lim\limits_{x''\to x'}\left(\deg(x')g_{E_v, x_v}(x'')-\frac{\deg(x')}{\deg(x)}g_{x'}(x'')\right).
$$
The \emph{Green's matrix} $\Gamma_X(\mathbb{E})$ is the symmetric matrix $(c_{x,y})_{x,y\in{X}}$ where $c_{x, y}=\sum_{v\in M_K}c_{x, y, v}$.
Note that this does not depend on the choice of the rational function.
The \emph{Cantor capacity} is defined as $\operatorname{cap}_X(\mathbb{E})=e^{-\mathrm{val} (\Gamma_X(\mathbb{E}))}$, where
$$
  \mathrm{val}(\Gamma)=\sup\limits_{\mathbf{s}\in\mathscr{P}^m}\inf\limits_{\mathbf{r}\in\mathscr{P}^m}\mathbf{s}^{\mathrm{T}}\Gamma\mathbf{r}
$$
for every $m\times m$ matrix $\Gamma$.
In the above, $\mathscr{P}^m$ denotes the set of $m$-dimensional probability vectors.
Note that our definition of Green's matrix is slightly different from that of Rumely \cite{Rum13}, whose rows and columns correspond to geometric points in ${X}$, while the rows and columns of our matrix correspond to closed points in ${X}$, obtained by averaging entries of Rumely's matrix over Galois orbits.
Since the Green's matrix defined in \cite{Rum13} is Galois stable, our definition of Cantor capacity gives the same number.
When the Cantor capacity $\operatorname{cap}_X(\mathbb{E})>1$, the matrix $\Gamma_X(\mathbb{E})$ is negative definite.

Let $\mathbb{E}=(E_v)_{v\in M_K}$ be an admissible adelic set on $U$. Let $D$ be an $\mathbb{R}$-divisor on $C$ of degree $1$.
Let $\{\mu_v\}_{v\in M_K}$ be a collection of compactly supported
probability measures on $E_v$ and assume that, for every
flat adelic divisor
$\overline{F}=(F,(g_{F,v})_{v\in M_K})$ with $F$ effective on $U$,
\[
  \sum_{v\in M_K}\int_{C_v^{\mathrm{an}}}
  g_{F,v}\,\mathrm{d}\mu_v
  \leq \overline{F}\cdot D.
\] This is the setting of Theorem \ref{thm:integral-point-construction} in Section \ref{sec:integral-point-construction}.
\begin{lemma}\label{lem:cap-mod-nbhd}
  The capacity $\operatorname{cap}_X(\mathbb{E})\geq1$.
\end{lemma}
\begin{proof}
  We prove by contradiction.
  Assume that $\operatorname{cap}_X(\mathbb{E})<1$.
  By \cite[Theorem 6.2.1]{Rum89}, there exists a rational function $f$ on $C$ with poles in ${X}$, such that $|f|_v\leq1$ on $E_v$ for all $v\in M_{K,f}$ and that $|f|_v<1$ for all $v\in M_{K,\infty}$.
  Then $\int_{C_v^{\mathrm{an}}}-\log|f|_v\,\mathrm{d}\mu_v\geq 0$, and the inequality is strict for archimedean $v$.
  Thus,
  $$
    \sum\limits_{v\in M_K}\int_{C_v^{\mathrm{an}}}-\log|f|_v\,\mathrm{d}\mu_v> 0.
  $$
  Consider the flat arithmetic divisor $\widehat{\mathrm{div}}(f)$.
  Since this divisor is principal, the right-hand side of the displayed inequality above is $0$.
  On the other hand, the left-hand side is greater than $0$.
  This leads to a contradiction.
\end{proof}

\begin{lemma}\label{lem:semi-def}
  Let $\Gamma$ be a symmetric matrix with all off-diagonal elements nonnegative. Assume that each proper principal submatrix of $\Gamma$ is negative semidefinite.
  If $\mathbf{x}$ is an eigenvector corresponding to a positive eigenvalue and has at least one positive coordinate, then all coordinates of $\mathbf{x}$ are positive.
\end{lemma}

\begin{proof}
  Note that $\mathbf{x}^{\mathrm{T}}\Gamma\mathbf{x}>0$.
  Let $\mathbf{x}_+=(\max\{x_i,0\})_i$ and $\mathbf{x}_-=(\max\{-x_i,0\})_i$, so that $\mathbf{x}=\mathbf{x}_+-\mathbf{x}_-$.
  Since all off-diagonal elements of $\Gamma$ are nonnegative, $\mathbf{x}_+^{\mathrm{T}}\Gamma\mathbf{x}_-\geq0$.
  Since each proper principal submatrix of $\Gamma$ is negative semidefinite, $\mathbf{y}^{\mathrm{T}}\Gamma\mathbf{y}\leq0$ for every vector $\mathbf{y}$ with at least one zero coordinate.
  Thus, $\mathbf{x}_+^{\mathrm{T}}\Gamma\mathbf{x}_+=\mathbf{x}^{\mathrm{T}}\Gamma\mathbf{x}+2\mathbf{x}_+^{\mathrm{T}}\Gamma\mathbf{x}_--\mathbf{x}_-^{\mathrm{T}}\Gamma\mathbf{x}_->0$.
  Therefore, all coordinates of $\mathbf{x}_+$ are non-zero, and we obtain that all coordinates of $\mathbf{x}$ are positive.
\end{proof}

\begin{proposition}\label{prop:gre-semi}
  The Green's matrix $\Gamma_X(\mathbb{E})$ is negative semidefinite.
\end{proposition}
\begin{proof}
  After taking a field extension of $K$, we may assume that ${X}$ consists of $K$-points.
  We prove by induction on $\#{X}$.
  We assume that the statement is true for all sets of $K$-points whose cardinality is less than $\#{X}$.
  Then each proper principal submatrix of $\Gamma_X(\mathbb{E})$ is negative semidefinite.
  By Lemma \ref{lem:semi-def}, if any eigenvalue of $\Gamma_X(\mathbb{E})$ is positive, its eigenvector can be taken to have all coordinates positive.
  This contradicts the fact that $\operatorname{cap}_X(\mathbb{E})\geq1$ guaranteed by Lemma \ref{lem:cap-mod-nbhd}.
\end{proof}
\begin{lemma}\label{lem:mat-inv}
  Let $\mathbf{x}\in{\mathbb {R}}^{{X}}$.
  If every coordinate of $\Gamma_X(\mathbb{E})\mathbf{x}$ is negative, then every coordinate of $\mathbf{x}$ is positive.
\end{lemma}
\begin{proof}
  Let $\mathbf{x}_+$ and $\mathbf{x}_-$ be defined by the same convention as above. Since every coordinate of $\Gamma_X(\mathbb{E})\mathbf{x}$ is negative,
  \[
    0\leq (-\mathbf{x}_-)^{\mathrm{T}}\Gamma_X(\mathbb{E})\mathbf{x}
    =-\mathbf{x}_-^{\mathrm{T}}\Gamma_X(\mathbb{E})\mathbf{x}_+
    +\mathbf{x}_-^{\mathrm{T}}\Gamma_X(\mathbb{E})\mathbf{x}_-
    \leq \mathbf{x}_-^{\mathrm{T}}\Gamma_X(\mathbb{E})\mathbf{x}_-
    \leq 0.
  \]
  Thus, $\mathbf{x}_{-}=\mathbf{0}$.
  Therefore, every coordinate of $\mathbf{x}$ is nonnegative.
  On the other hand, for $\varepsilon$ small enough, every coordinate of $\Gamma_X(\mathbb{E})(\mathbf{x}-\varepsilon\mathbf{1})$ is also negative.
  This forces the coordinates of $\mathbf{x}$ to be positive.
\end{proof}

\begin{lemma}\label{lem:can-cap-sol}Assume that $\operatorname{cap}_X(\mathbb{E})>1$.
  Let $\mathbf{x}\in{\mathbb {R}}^{{X}}$, and let $c\geq 0$ be a number greater than or equal to every coordinate of $\mathbf{x}$.
  Then every coordinate of $\Gamma_X(\mathbb{E})^{-1}\mathbf{x}$ is at least $-c/(\log\operatorname{cap}_X(\mathbb{E}))$.
\end{lemma}
\begin{proof}
  Let $\mathbf{x}'=\mathbf{x}-c\mathbf{1}$.
  By Lemma \ref{lem:mat-inv}, every coordinate of $\Gamma_X(\mathbb{E})^{-1}\mathbf{x}'$ is nonnegative.
  On the other hand, $-(\log\operatorname{cap}_X(\mathbb{E}))\Gamma_X(\mathbb{E})^{-1}\mathbf{1}$ is a probability vector.
  Therefore, every coordinate of $c\Gamma_X(\mathbb{E})^{-1}\mathbf{1}$ is at least $-c/(\log\operatorname{cap}_X(\mathbb{E}))$.
\end{proof}
\begin{lemma}\label{lem:gre-def}
  Fix an archimedean place $v_0$.
  Assume that each connected component of $C_{v_0}^{\mathrm{an}}\setminus {\mathrm{supp}}(\mu_{v_0})$ that contains points in ${X}$ intersects $E_{v_0}$.
  Then $\operatorname{cap}_X(\mathbb{E})>1$.
\end{lemma}
\begin{proof}
  For each pair of open sets $V_2\subseteq V_1$ in $C^{\mathrm{an}}_{v_0}$ with $V_1$ arcwise connected, we can take a compact set $H$ in $V_2$ such that the interior of $H$ is non-empty and that $V_1\setminus H$ is arcwise connected, as follows.
  By replacing $V_2$ by a subset of itself, we may assume that $V_2$ is homeomorphic to the open unit disk.
  Take $H$ to be the image of the closed ball centered at the origin with radius $1/2$ under the homeomorphism.
  For each open interval $I$ in $C_{v_0}({\mathbb {R}})$ and each point $y\in I$, it is easy to see that removing a small neighborhood of $y$ cuts $I$ into two open intervals.
  Therefore, we are able to construct a locally admissible compact set $E_{v_0}'\subset E_{v_0}$ such that for each connected component $W$ of $C_{v_0}^{\mathrm{an}}\setminus {\mathrm{supp}}(\mu_{v_0})$ that contains points in ${X}$, $W\bigcap(E_{v_0}\setminus E_{v_0}')$ is non-empty.
  Note that non-emptiness is equivalent to non-polarity, since the set contains an open interval or an open subset as long as it is not empty.

  Let $\mathbb{E}'=(E_v')_{v\in M_K}$ where $E_v'=E_v$ for each $v\neq v_0$.
  Now we compare $\Gamma_X(\mathbb{E})$ and $\Gamma_X(\mathbb{E}')$.
  Since $E_{v_0}'\subset E_{v_0}$, every entry of the former matrix is less than or equal to the corresponding entry of the latter.
  On the diagonal, the inequality is strict since for each connected component $W$ of $C_{v_0}^{\mathrm{an}}\setminus {\mathrm{supp}}(\mu_{v_0})$ that contains points in ${X}$, $W\bigcap(E_{v_0}\setminus E_{v_0}')$ is not polar.
  By Proposition \ref{prop:gre-semi}, $\Gamma_X(\mathbb{E}')$ is negative semidefinite.
  Hence, $\Gamma_X(\mathbb{E})$ is negative definite, and $\operatorname{cap}_X(\mathbb{E})>1$.
\end{proof}

\subsection{Arakelov theory}
\subsubsection{Flat arithmetic divisors}
In this section, we introduce flat arithmetic divisors and the arithmetic Hodge index theorem on arithmetic surfaces developed by Faltings \cite{Fal84} and Hriljac \cite{Hri85}.

Let $C$ be a projective, smooth and geometrically integral curve over a number field $K$. Let ${\mathcal {C}}$ be a projective regular model of $C$ over $O_K$.
An arithmetic divisor is called \emph{vertical} if its geometric part, i.e. its restriction to $C$, is trivial.
For each $D\in \mathrm{Pic}^0(C)$, up to pull-back of ${\mathbb {Q}}$-arithmetic divisors on $\operatorname{Spec}(K)$, there exists a unique ${\mathbb {Q}}$-arithmetic divisor $\overline{\mathcal D}$ on ${\mathcal {C}}$ such that $\overline{\mathcal D}\cdot\overline{\mathcal E}=0$ for all vertical arithmetic divisors $\overline{\mathcal E}$.
Such arithmetic divisors are called \emph{flat}.
Recall that there is a N\'eron--Tate height pairing
$$
  \langle\cdot, \cdot\rangle_{NT}: \mathrm{Pic}^0(C_{{\overline{K}}})\times\mathrm{Pic}^0(C_{{\overline{K}}})\longrightarrow {\mathbb {R}}
$$
which is a non-degenerate symmetric bilinear form.
The arithmetic Hodge index theorem asserts the following equality.
\begin{theorem}[Faltings, Hriljac]
  For each $D, E\in \mathrm{Pic}^0(C)$, let $\overline{\mathcal D}$, $\overline{\mathcal E}$ be flat arithmetic divisors with geometric parts $D$ and $E$, respectively.
  Then we have
  $$
    \overline{\mathcal D}\cdot\overline{\mathcal E}=-2[K:{\mathbb {Q}}]\langle D, E\rangle_{NT}.
  $$
\end{theorem}
\subsubsection{Arithmetic Nakai--Moishezon Theorem}
In this section, we introduce the arithmetic Nakai--Moishezon theorem proved by Zhang \cite[Theorem 1.5]{Zha92}.

Let $C$ be a projective, smooth and geometrically integral curve over a number field $K$.
Let ${\mathcal {C}}$ be a projective regular model of $C$ over $O_K$.
A hermitian line bundle is said to be \emph{nef} if its intersection with any effective hermitian line bundle is nonnegative.
Let $\overline{{\mathcal {L}}}$ be a hermitian line bundle on ${\mathcal {C}}$.
A section $s\in H^0({\mathcal {C}}, {\mathcal {L}})$ is said to be \emph{small} if for each $v\in M_{K,\infty}$, $\lVert s(x)\rVert_{\overline{{\mathcal {L}}}, v}\leq 1$ for each $x\in C_v^{\mathrm{an}}$.
\begin{theorem}[Zhang]
  If $\overline{{\mathcal {L}}}$ is nef, then for any hermitian line bundle $\overline{{\mathcal {M}}}$ and any $\varepsilon>0$, there exists $N>0$ such that for each $n>N$, the $K$-vector space $H^0(C, M\otimes L^{\otimes n})$ has a basis consisting of small sections of $\overline{{\mathcal {M}}}\otimes \overline{{\mathcal {L}}}(\varepsilon)^{\otimes n}$.
\end{theorem}

\subsubsection{Adelic divisors}
In this section, we briefly introduce adelic divisors on curves.
For details, we refer to \cite{Zha95}.
In the proof of our main theorem, we only use model adelic divisors, i.e. those for which some multiple corresponds to an arithmetic divisor.
We consider them as adelic divisors since the Green function at non-archimedean places is required.

Let $K$ be a number field and $C$ a projective smooth curve over $K$.
Let ${\mathcal {C}}$ be a projective regular model and $\overline{\mathcal D}=({\mathcal {D}},(g_{{\mathcal {D}},v})_{v\in M_{K,\infty}})$ an arithmetic divisor on ${\mathcal {C}}$.
Let $D$ be a divisor on $C$ and $(g_{D,v})_{v\in M_K}$ a collection of Green functions of $D_v$ on $C_v^{\mathrm{an}}$.
The collection ${\overline{D}}=(D,(g_{D,v})_{v\in M_K})$ is said to be the adelic divisor corresponding to $\overline{\mathcal D}$ if the following holds.
Firstly, $D={\mathcal {D}}_K$.
For each archimedean $v$, $g_{D, v}=g_{{\mathcal {D}}, v}$.
For each non-archimedean $v$, recall that there is a reduction map ${\mathrm{red}}_v: C_v^{\mathrm{an}}\longrightarrow {\mathcal {C}}_{k_v}$.
Let $g_{D,v}(x)=-\log\left|f(x)\right|_v$ for each $x\in C_v^{\mathrm{an}}$ where $f$ is a rational function on $C$ such that ${\mathrm{div}}(f)$ coincides with ${\mathcal {D}}$ in an open neighborhood of ${\mathrm{red}}_v(x)$ in ${\mathcal {C}}$.
Note that $g_{D,v}$ is a Green function of $D_v$ for each $v$.

An arithmetic divisor is said to be \emph{integrable} if it is the difference of two nef arithmetic divisors.
If for some integer $n>0$, $n{\overline{D}}=(nD,(ng_{D,v})_{v\in M_K})$ corresponds to some (integrable) arithmetic divisor on some projective model, we say that ${\overline{D}}$ is a \emph{model adelic divisor} (respectively, an \emph{integrable model adelic divisor}).
If there exists a sequence of model adelic divisors ${\overline{D}}_n=(D,(g_{D,n,v})_{v\in M_K})$ such that there exists a finite set $S\subset M_K$ with $g_{D, n, v}=g_{D, v}$ for each $n$ and $v\notin S$ and that the difference $g_{D,n,v}-g_{D,v}$ extends continuously to
$C_v^{\mathrm{an}}$ and converges uniformly to zero there
as $n\to+\infty$ for each $v\in S$, we say that ${\overline{D}}=(D,(g_{D,v})_{v\in M_K})$ is an \emph{adelic divisor}.

For an integrable model adelic divisor ${\overline{D}}=(D,(g_{D,v})_{v\in M_K})$, the \emph{Chambert-Loir measure} of ${\overline{D}}$ is defined as $c_1({\overline{D}})_v=c_1(g_{D,v})$ for each $v\in M_K$, which is a signed measure of total mass $\deg(D)$.
Denote by ${\mathrm{Irr}}({\mathcal {C}}_{k_v})$ the set of irreducible components of ${\mathcal {C}}_{k_v}$.
Alternatively, the measure is given by
$$
  c_1({\overline{D}})_v= \sum\limits_{F\in{\mathrm{Irr}}({\mathcal {C}}_{k_v})}l_F\deg_F({\mathcal {D}})\delta_F,
$$
where $\delta_F$ is the Dirac measure of the divisorial point corresponding to $F$, and $l_F$ is the order of a generator of the maximal ideal of $O_{K_v}$ at $F$.
Let ${\overline{D}}=(D,(g_{D,v})_{v\in M_K})$ and ${\overline{E}}=(E,(g_{E,v})_{v\in M_K})$ be integrable model adelic divisors such that $D$ and $E$ have disjoint supports. Then we have the induction formula
$$
  {\overline{D}}\cdot {\overline{E}}=\sum\limits_{v\in M_K}\int_{C_v^{\mathrm{an}}}g_{E, v}\delta_{D, v}+\sum\limits_{v\in M_K}\int_{C_v^{\mathrm{an}}}g_{D, v}c_1({\overline{E}})_v.
$$
Note that the sums on the right-hand side are finite.
\subsection{Adelic divisors on quasi-projective varieties}
We briefly review the theory of adelic line bundles on quasi-projective varieties, developed by Yuan--Zhang \cite{YZ21}.
For details, see \cite[\S 2]{YZ21}.
We also prove that adelic line bundles on the spectrum of a field form an $\mathbb R$-vector space.

Let $\mathcal X$ be a projective arithmetic variety.
Let $\mathcal U$ be an open subscheme of $\mathcal X$.
The set of arithmetic $(\mathbb Q, \mathbb Z)$-divisors on $(\mathcal X, \mathcal U)$ is defined as
\[
  \left\{(\overline{\mathcal D}, \mathcal D')\big |\overline{\mathcal D}\in \operatorname{\widehat{Div}}(\mathcal X)_{\mathbb Q}, \mathcal D'\in\operatorname{Div}(\mathcal U), \mathcal D|_{\mathcal U}=\mathcal D'{\ \rm\  in} \operatorname{Div}(\mathcal U)_{\mathbb Q}\right\}.
\]

By a \emph{projective model} of $\mathcal U$, we mean a projective arithmetic variety $\mathcal Y$ with an open immersion $\mathcal U \longrightarrow \mathcal Y$.
Note that projective models of $\mathcal U$ form an inverse system.
The set of all model adelic divisors is defined as
\[
  \operatorname{\widehat{Div}}(\mathcal U/\mathbb Z)_{\rm mod}\coloneqq \varinjlim_{\mathcal Y {\ \rm is\ a\ projective\ model\ of\ } \mathcal U } \operatorname{\widehat{Div}}(\mathcal Y, \mathcal U).
\]

By a \emph{boundary divisor} $(\mathcal X_0, \overline{\mathcal E})$ of $\mathcal U$, we mean a projective model $\mathcal X_0$ of $\mathcal U$, together with a strictly effective arithmetic divisor $\overline{\mathcal E}$ on $\mathcal X_0$ whose support equals $\mathcal X_0\setminus \mathcal U$.
The \emph{boundary norm} on $\operatorname{\widehat{Div}}(\mathcal U/\mathbb Z)_{\rm mod}$ for the boundary divisor $(\mathcal X_0, \overline{\mathcal E})$ is defined as
\[
  \lVert\overline{\mathcal D}\rVert_{(\mathcal X_0, \overline{\mathcal E})}=\inf\{\epsilon\in\mathbb Q_{>0}|-\epsilon\overline{\mathcal E}\leq \overline{\mathcal D}\leq \epsilon \overline{\mathcal E}\}.
\]
Here, we view $\inf\emptyset=+\infty$.
The set of \emph{adelic divisors} $\operatorname{\widehat{Div}}(\mathcal U/\mathbb Z)$ on $\mathcal U$ is the completion of $\operatorname{\widehat{Div}}(\mathcal U/\mathbb Z)_{\rm mod}$ under the boundary norm.
Note that the definition does not depend on the choice of the boundary divisor.

Adelic divisors can be defined not only on quasi-projective varieties but also on essentially quasi-projective varieties.
For the definition of essentially quasi-projective varieties, see \cite[\S 2.3]{YZ21}.
In this section, we only focus on the crucial case when $\mathcal U$ is the generic point of an arithmetic variety, or equivalently, the spectrum of a finitely generated field.
In such a case, adelic divisors have a natural $\mathbb R$-multiplication structure.
Let $\mathcal{K}$ be a finitely generated field.
A quasi-projective model of $\mathcal{K}$ is a quasi-projective arithmetic variety with function field $\mathcal{K}$.
Define
\[
  \operatorname{\widehat{Div}}(\operatorname{Spec}(\mathcal{K})/\mathbb{Z})= \varinjlim_{\emptyset \neq \mathcal U {\rm\  is\ a\ quasi-projective\ model\ of\ } \mathcal{K}}\operatorname{\widehat{Div}}(\mathcal U/\mathbb{Z}).
\]

Adelic divisors are defined as an abelian group.
In the case of a field, however, it turns out that $\operatorname{\widehat{Div}}(\operatorname{Spec}(\mathcal{K})/\mathbb{Z})$ forms an $\mathbb R$-vector space.

\subsection{GVF functionals on curves} \label{sec:gvf}

Let $K$ be a number field, and let $C/K$ be a smooth, projective and
geometrically integral curve.
We shall use without further comment the space
\[
  \widehat{\operatorname{Div}}(K(C)/\mathbb Z)
  =\varinjlim_{\varnothing\ne V\subseteq C}\widehat{\operatorname{Div}}(V/\mathbb Z)
\]
of adelic $\mathbb{R}$-divisors on $\operatorname{Spec}(K(C))$.  When a class is represented on $C$,
we write it as
\(
\overline D=(D,(g_{D,v})_{v\in M_K})
\),
where $D$ is an $\mathbb{R}$-divisor and $g_{D,v}$ is a Green function for $D$
on $C_v^{\mathrm{an}}$.  Our convention for a principal adelic divisor is
\begin{equation}\label{eq:principal-divisor}
  \widehat{\operatorname{div}}(f)
  =\bigl(\operatorname{div}(f),(-\log|f|_v)_{v\in M_K}\bigr),
  \qquad f\in K(C)^\times.
\end{equation}

\begin{definition}\label{def:gvf-functional}
  A \emph{GVF functional} on $K(C)$ is an $\mathbb{R}$-linear map
  \[
    \ell:\widehat{\operatorname{Div}}(K(C)/\mathbb Z)\longrightarrow\mathbb{R}
  \]
  which is nonnegative on effective adelic divisors and vanishes on principal
  adelic divisors.  It is \emph{normalized} if
  \begin{equation}\label{eq:normalization}
    \ell(\pi^*\overline A)=\widehat{\deg}(\overline A)
  \end{equation}
  for every $\overline A\in\widehat{\operatorname{Div}}(\operatorname{Spec}\mathbb Z)$, where
  $\pi:\operatorname{Spec}(K(C))\longrightarrow\operatorname{Spec}(\mathbb Z)$ is the structure morphism.
\end{definition}

Thus a GVF functional is nothing but a positive functional on the space of adelic line bundles. The vanishing on \eqref{eq:principal-divisor} is called the product formula.

The representation theorem for GVF structures
\cite[Corollary 9.4]{yaacov2024globallyvaluedfieldsfoundations} produces measures on the Berkovich analytification $C_v^\mathrm{an}$ for $v\in M_K$.
For every $v\in M_K$, the theorem gives a probability measure $\mu_v$ on $C_{v}^{\mathrm{an}}$.
For the trivial valuation, note that every nontrivial valuation of $K(C)$ trivial on
$K$ is of the form $t\operatorname{ord}_x$, for an $x\in|C|$ and some $t>0$.  Put
\[
  \mathcal R_x:=\{t\operatorname{ord}_x:t>0\},
  \quad
  \lambda_x(t\operatorname{ord}_x):=t,
  \quad
  m(x):=\int_{\mathcal R_x}\lambda_x\,\mathrm d\mu_0.
\]
Here $\mathcal R_x$ is the ray generated by $x \in |C|$ and $\mu_0$ is the restriction to the
trivially valued locus of an admissible representative $\mu$ on the global
valuation space $\Omega_{K(C)}$. Extend $m$ linearly to $\operatorname{Div}(C)_{\mathbb R}$.
The functional $m$ is canonical, but the measure $\mu_0$ is not canonical \cite[Remark~9.5]{yaacov2024globallyvaluedfieldsfoundations}.
\begin{theorem}
  A normalized GVF functional on $K(C)$ is equivalent to data $\bigl(m,\{\mu_v\}_{v\in M_K}\bigr)$, where $m:\operatorname{Div}(C)_{\mathbb R}\to\mathbb R$ is linear and each $\mu_v$ is a probability measure on $C_v^{\mathrm{an}}$, satisfying
  \begin{enumerate}
    \item $m(D)\geq0$ for every effective $\mathbb R$-divisor $D$,
    \item ($L^1$ integrable) $\displaystyle  \sum_{v\in M_K}\int \bigl|\log|f|_v\bigr|\,\mathrm d\mu_v<+\infty$, for every $f \in K(C)^\times$,
    \item (product formula) $\displaystyle m(\operatorname{div}(f))=\sum_{v\in M_K}\int \log|f|_v\,\mathrm d\mu_v$, for every $f\in K(C)^\times$.
  \end{enumerate}
\end{theorem}

For an adelic divisor $\overline{D}=(D,(g_{D,v})_{v\in M_K})$, the
evaluation formula is
\begin{equation}\label{eq:evaluation}
  \ell(\overline D)
  =m(D) + \sum_{v\in M_K}
  \int_{C_v^{\mathrm{an}}}g_{D,v}\,\mathrm d\mu_v.
\end{equation}
The infinite sum is absolutely convergent because of the $L^1$-integrability condition above.
Thus $m(D)=\ell(\overline D)-\sum_v\int g_{D,v}\,\mathrm d\mu_v$ for any adelic lift $\overline D$ of $D$.
For boundary points $x\in X$, we use the notation
\[
  m_x:=\frac{m(x)}{\deg(x)},
\]
and write $m_x(\ell)$ when the dependence on $\ell$ needs to be specified.

\begin{lemma}\label{lem:gvf-d-l}
  For every normalized GVF functional $\ell$ there is a unique class
  \(
  D_\ell\in\operatorname{Pic}^1(C)_\mathbb{R}
  \)
  such that
  \begin{equation}\label{eq:d-ell}
    \ell(\overline E)=D_\ell\cdot\overline E, \quad \text{for every flat adelic divisor $\overline E$}.
  \end{equation}
\end{lemma}

\begin{proof}
  Choose $\widetilde D\in\operatorname{Pic}^1(C)_\mathbb{R}$.  If $E\in\operatorname{Pic}^0(C)_\mathbb{R}$ and
  $\overline E$ is a flat adelic divisor, define
  \[
    \Lambda(E)=\ell(\overline E)-\widetilde D\cdot\overline E.
  \]
  $\Lambda(E)$ is a linear functional on
  $\operatorname{Pic}^0(C)_\mathbb{R}$.  The Faltings--Hriljac theorem and the nondegeneracy of the
  N\'eron--Tate pairing give a unique
  $D'\in\operatorname{Pic}^0(C)_\mathbb{R}$ such that
  $\Lambda(E)=D'\cdot\overline E$ for all $E$.  Then
  $D_\ell=\widetilde D+D'$ satisfies \eqref{eq:d-ell}, and the same
  nondegeneracy proves uniqueness.
\end{proof}

\begin{lemma} \label{lem:gvf-on-curves}
  Fix a closed point $x_0\in|C|$. A normalized GVF functional $\ell=\bigl(m,\{\mu_v\}_{v\in M_K}\bigr)$ is the same as the data $\bigl(D_\ell,m(x_0),\{\mu_v\}_{v\in M_K}\bigr)$ where
  \begin{enumerate}
    \item $D_\ell \in \operatorname{Pic}^1(C)_\mathbb{R}$, $m(x_0) \in \mathbb{R}_{\geq 0}$,
    \item ($L^1$ integrable) $\displaystyle  \sum_{v\in M_K}\int \bigl|\log|f|_v\bigr|\,\mathrm d\mu_v<+\infty$, for every $f \in K(C)^\times$,
    \item (positivity) $\displaystyle \frac{m(x_0)}{\deg(x_0)}+D_\ell \cdot \overline E -\sum_{v\in M_K}\int g_{E,v}\,\mathrm d\mu_v \geq 0$, for every $E=x/\deg(x)-x_0/\deg(x_0)$ and flat adelic divisor $\overline{E}=(E,(g_{E,v})_{v\in M_K})$.
  \end{enumerate}
\end{lemma}

\begin{proof}
  For $x \in |C|$, let $E = x/\deg(x)-x_0/\deg(x_0)$ and $\overline{E}=(E,(g_{E,v})_{v\in M_K})$ be a flat adelic divisor, and then set
  \[
    \frac{m(x)}{\deg(x)}=\frac{m(x_0)}{\deg(x_0)}+D_\ell \cdot \overline E -\sum_{v\in M_K}\int g_{E,v}\,\mathrm d\mu_v.
  \]

\end{proof}

The \emph{GVF topology} is just pointwise
convergence: we say $\ell_n\longrightarrow\ell$ in the GVF topology if
$\ell_n(\overline D)\longrightarrow\ell(\overline D)$ for every adelic
divisor $\overline D$. For a sequence $(\alpha_n)$ of distinct points in $C(\overline K)$, we say
$\alpha_n\longrightarrow\ell$ in the GVF topology if
$h_{\overline D}(\alpha_n)\longrightarrow\ell(\overline D)$ for every
adelic divisor $\overline D$.

\begin{lemma} \label{lem:gvf-point-convergence}
  Let $(E_v)_{v\in M_K}$ be an admissible adelic set on $U$.
  Fix $x_0\in X$.
  Let $(\alpha_n)$ be a sequence of distinct points in $C(\overline K)$ and $\ell=(m,\{\mu_v\})$ be a normalized GVF functional with $m_{x_0}=0$.
  Assume that both $\alpha'_{n, v}$ and $\operatorname{supp}\mu_v$ are contained in $E_v$ for each $v$.
  Then $\alpha_n \longrightarrow \ell$ in the GVF topology if and only if
  \[
    \alpha_n'/\deg(\alpha_n') \longrightarrow D_\ell \text{ in $\operatorname{Pic}^1(C)_\mathbb{R}$}\text{ and } \quad \delta_{\alpha_n,v} \longrightarrow \mu_v \text{ for all $v \in M_K$}.
  \]
\end{lemma}
\begin{proof}
  Suppose first that $\alpha_n\to\ell$ in the GVF topology.
  Testing against adelic divisors with trivial underlying divisor and a continuous Green function at a single place gives
  \[
    \delta_{\alpha_n,v}\longrightarrow\mu_v
  \]
  weakly for every $v\in M_K$.
  Moreover, for every flat adelic divisor $\overline F$,
  \[
    \alpha_n'/\deg(\alpha_n')\cdot\overline F
    =h_{\overline F}(\alpha_n)
    \longrightarrow\ell(\overline F)
    =D_\ell\cdot\overline F.
  \]
  We get the convergence by the non-degeneracy of the N\'eron--Tate height pairing.

  Conversely, we choose a model adelic divisor
  \[
    \overline D_0=(x_0,(g_{0,v})_{v\in M_K}).
  \]
  By admissibility, there is a finite set $S_0\subset M_K$ such that $g_{0,v}=0$ on $E_v$ for $v\notin S_0$.
  For $v\in S_0$, the restriction of $g_{0,v}$ to $E_v$ is bounded and continuous.
  Therefore, we obtain
  \[
    h_{\overline D_0}(\alpha_n)=\sum_{v\in S_0}\int g_{0,v}\,\mathrm d\delta_{\alpha_n,v}\longrightarrow \sum_{v\in S_0}\int g_{0,v}\,\mathrm d\mu_v =\ell(\overline D_0).
  \]
  Also, for each flat adelic divisor $\overline F$, we have $h_{\overline F}(\alpha_n)\to \ell(\overline F)$.
  Therefore, we obtain the convergence for every model adelic divisor.

  For each effective adelic divisor $\overline D_1\in\operatorname{\widehat{Div}}(K(C)/\mathbb Z)$, since the points $\alpha_n$ are distinct, $\alpha_n\notin\operatorname{supp}(D_1)$ for $n$ large enough.
  Therefore, $h_{\overline D_1}(\alpha_n)\geq 0$ for $n$ large enough.
  Thus, for any adelic divisor $\overline D_1$ and any model
  adelic $\mathbb R$-divisor $\overline D_0\geq\overline D_1$,
  we have
  \[
    \limsup_{n \to \infty}h_{\overline D_1}(\alpha_n)\leq \lim_{n \to \infty} h_{\overline D_0}(\alpha_n)=\ell(\overline D_0).
  \]
  Therefore, we have
  \[
    \limsup_{n \to \infty}h_{\overline D_1}(\alpha_n)\leq\inf_{\overline D_2\text{ model and }\overline D_2\geq \overline D_1}\ell(\overline D_2)=\ell(\overline D_1).
  \]
  For the same reason,
  \[
    \liminf_{n \to \infty}h_{\overline D_1}(\alpha_n)\geq \ell(\overline D_1).
  \]
  We thus get the convergence for all adelic divisors.
\end{proof}

\subsubsection{Balayage of GVF}
\label{sec:balayage}

Let $\ell=(m,\{\mu_v\})$ be a normalized GVF functional.
Let $\mathbb E=\{E_v\subset C_v^{\mathrm{an}}\}_{v\in M_K}$ be an adelic set.
Assume that each $E_v$ is locally admissible in $C_v^{\mathrm{an}}$, and there exists an open subvariety $U\subset C$ together with its regular quasi-projective model $\mathcal U$, such that $E_v=\operatorname{red}_v^{-1}(\mathcal U_{k_v})$ for all but finitely many $v\in M_{K, f}$.
Write $u_v=u_{\mu_v}^{E_v}$ for the balayage defect of $\mu_v$ to $E_v$.
For a divisor $D_v=\sum_t a_t t$, we write $u_v(D_v)=\sum_t a_t\deg(t)u_v(t)$.
We construct the balayage GVF as follows.
\begin{lemma}\label{lem:balayage-gvf}
  For $D\in\operatorname{Div}(C)_{\mathbb R}$, let
  \[
    m^{\mathbb E}(D)=m(D)+\sum_{v\in M_K}u_v(D_v).
  \]
  Then the ``balayage GVF'' $\ell^{\mathbb E}=\bigl(m^{\mathbb E},\{\mu_v^{E_v}\}_{v\in M_K}\bigr)$ is again a normalized GVF functional.
\end{lemma}

\begin{proof}
  First, we check the $L^1$ condition.
  For each non-zero rational function $f$ on $C$, we consider its absolute value on $\partial E_v$ for $v\in M_K$.
  Except for finitely many places, $v$ is non-archimedean, $\mathcal U_{k_v}$ is integral, and $f$ has order zero along $\mathcal U_{k_v}$.
  Then, for these places, $\partial E_v$ is either empty or consists of a unique point where $|f|_v=1$.
  Let $S$ be the finite set of exceptional places.
  For each $v\notin S$, we have
  \[
    \int_{C_v^\mathrm{an}}\left|\log\left|f\right|_v\right|\,\mathrm d\mu_v\geq \int_{C_v^\mathrm{an}}\left|\log\left|f\right|_v\right|\,\mathrm d\mu_v^{E_v},
  \]
  since the balayage only moves the measure outside $E_v$ to $\partial E_v$, and $\log|f|_v=0$ on $\partial E_v$.

  Set $D=(f)_0+(f)_{\infty}$.
  For each $v\in S$, let $g_{D,v}$ be a nonnegative continuous Green function of $D_v$ such that $c_1(g_{D,v})$ is supported on $E_v$.
  Then, $g_{D,v}$ is superharmonic outside $E_v$.
  Applying Proposition \ref{prop:non-sub-avg} to $-\min(g_{D,v},N)$ and letting $N\to+\infty$, we obtain
  \[
    \int_{C_v^\mathrm{an}}g_{D,v}\,\mathrm d\mu_v\geq \int_{C_v^\mathrm{an}}g_{D,v}\,\mathrm d\mu_v^{E_v}.
  \]
  We also see that
  \[
    \sup_{C_v^\mathrm{an}\setminus\operatorname{Supp}(D_v)}
    \left|\left|\log\left|f\right|_v\right|-g_{D,v}\right|<\infty,
  \]
  since $|\log|f|_v|$ is also a continuous Green function of $D_v$.
  Therefore,
  \[
    \int_{C_v^\mathrm{an}}\left|\log\left|f\right|_v\right|\,\mathrm d\mu_v^{E_v}
    -\int_{C_v^\mathrm{an}}\left|\log\left|f\right|_v\right|\,\mathrm d\mu_v<+\infty.
  \]
  We hence obtain the adelic $L^1$ condition.

  Now, for each closed point $x\in U$, $x_v\subset E_v$ for all but finitely many places.
  By the $L^1$ condition of $\mu_v$, we see that $u_v(t)<+\infty$ for each $t \in x_v$.
  Therefore, $m^{\mathbb E}(x)<+\infty$ for each $x\in U$.

  For each non-zero regular function $f$ on $U$, by the Green reciprocity law and the Lelong--Poincar\'e formula, we have
  \[
    m^\mathbb{E}(\operatorname{div}(f)) - m(\operatorname{div}(f)) = \sum_{v\in M_K}\int \log|f|_v \,\mathrm{d}(\mu_v^{E_v}-\mu_v) = \sum_{v\in M_K}u_v(\operatorname{div}(f)_v).
  \]
  Hence, we see that $m^{\mathbb E}(x)<+\infty$ for each $x\in C$.
  Now, the same equality holds for every non-zero rational function $f$; hence we have $\ell^{\mathbb E}(\widehat{\operatorname{div}}(f))=\ell(\widehat{\operatorname{div}}(f))=0$.
\end{proof}

\begin{lemma}\label{lem:balayage-inequality}
  Take $\overline L\in \operatorname{\widehat{Div}}(K(C)/\mathbb Z)$.
  Assume that the geometric part $\widetilde{L}$ has positive degree and $c_1(\overline{L})_v\leq 0$ outside $E_v$ for each $v$.
  Then, we have $\ell^{\mathbb E}(\overline{L})\leq\ell(\overline{L})$.
\end{lemma}

\begin{proof}
  By taking powers of $\overline L$, we may assume that $\widetilde{L}$ is effective.
  Take a section $s$ of $\widetilde L$.
  Write $\widehat{\operatorname{div}}(s)=(D, \{g_{D, v}\})$.
  For a place $v\in M_K$, take a Green function $g'_{D, v}$ of $D_v$ such that $c_1(g'_{D, v})$ is supported on $E_v$.
  Then, $c_1(g_{D, v}-g'_{D, v})\leq0$ outside $E_v$.
  Thus,
  \[
    \int g_{D, v}\,\mathrm{d}\mu_v^{E_v}-\int g_{D, v}\,\mathrm{d}\mu_v\leq \int g'_{D, v}\,\mathrm{d}\mu_v^{E_v}-\int g'_{D, v}\,\mathrm{d}\mu_v=-u_v(D_v)
  \]
  \begin{flalign*}
    \ell^{\mathbb E}(\overline{L}) &= m(D) + \sum_{v\in M_K}u_v(D_v) + \sum_{v\in M_K}\int g_{D, v}\,\mathrm{d}\mu_v^{E_v}  \\
    &= \ell(\overline{L}) + \sum_{v\in M_K}u_v(D_v) + \sum_{v\in M_K}\int g_{D, v}\,\mathrm{d}\mu_v^{E_v}-\sum_{v\in M_K}\int g_{D, v}\,\mathrm{d}\mu_v \\
    &\leq \ell(\overline{L}).
  \end{flalign*}
\end{proof}

In particular, the preceding two lemmas apply to loose admissible adelic sets.
\section{Regularization of sets and measures}
\label{sec:regularization-of-measures}

Let $K$ be a number field and $C/K$ be a projective curve that is smooth and geometrically connected.
Let $X \subseteq C$ be a finite set of points and $U=C-X$.
Let $\mathbb{E}=(E_v)_{v\in M_K}$ be an admissible adelic set on $U$ with $\operatorname{cap}_X(\mathbb{E})>1$.
Let $D$ be an $\mathbb{R}$-divisor on $C$ of degree $1$, and assume that we have, for every flat adelic divisor $\overline{F}=(F,(g_{F,v})_{v\in M_K})$ with $F$ effective on $U$,
\begin{equation}
  \sum \int g_{F,v} \,\,\mathrm{d}\mu_v \leq \overline{F} \cdot D. \label{eq:regularization-hypothesis}
\end{equation}
These are the assumptions of Theorem \ref{thm:integral-point-construction}.

In this section, we prove technical lemmas and propositions concerning regularization of the sets $E_v$ and the measures $\mu_v$. They will be used in the proof of Theorem \ref{thm:integral-point-construction} in Section \ref{sec:integral-point-construction}.

\subsection{Archimedean case}
First, we study local theory.
Let $C/\mathbb{C}$ be a projective smooth curve that is geometrically connected.
Take a point $x\in C(\mathbb{C})$ and a Green function $g_x$ at $x$.
Set $G_x(y_1, y_2)=U_{y_1}(y_2)$.
Then we have a distance function $d(y_1,y_2) \coloneqq e^{G_x(y_1,y_2)}$.
Note that this distance function does not satisfy the triangle inequality.
However, there is a metric $\rho$ on $C(\mathbb{C})\setminus\{x\}$, such that $d(y_1, y_2)=\rho(y_1, y_2)e^{\eta(y_1, y_2)}$ where $\eta$ is a continuous function on $C(\mathbb{C})\times C(\mathbb{C})$.
See \cite[p. 75]{Rum13} and \cite[p.129]{Rum89} for the existence of such a metric.
For simplicity, we assume that $\eta(y_1, y_2)<0$ for each $y_1, y_2$.

\begin{lemma}\label{lem:cap-low-bd}
  Let $H\subset C(\mathbb{C})\setminus\{x\}$ be a path-connected compact set.
  Set
  \[
    A=\inf_{y_1, y_2\in H}\eta(y_1, y_2).
  \]
  Then the capacity $\operatorname{cap}_x(H)\geq e^{A}\operatorname{diam}_\rho(H)/4$, where $\operatorname{diam}_\rho(H)$ is the diameter of $H$ under the metric $\rho$.
\end{lemma}
\begin{proof}
  Let $y_1$ and $y_2$ be two points at distance $\operatorname{diam}_\rho(H)$ in $H$.
  Let $\Lambda$ be an arc connecting $y_1$ and $y_2$.
  Then $\operatorname{cap}_x(H)\geq \operatorname{cap}_x(\Lambda)$.
  Let $h: \Lambda \longrightarrow [0,\operatorname{diam}_\rho(H)]$ be the surjective continuous map $y\longmapsto \rho(y_1,y)$.
  Let $\mu_{\Lambda}$ be the equilibrium measure of $\Lambda$.
  Then for all $y\in \Lambda$,
  $$
    \int_\Lambda \log d(y,y')\,\mathrm{d}\mu_{\Lambda}(y')=\log \operatorname{cap}_x(\Lambda).
  $$
  On the other hand, since the capacity of $[0,\operatorname{diam}_\rho(H)]$ is $\operatorname{diam}_\rho(H)/4$, there exists $z'\in[0,\operatorname{diam}_\rho(H)]$, such that
  $$
    \int_{\Lambda}\log\left|h(y)-z'\right|\,\mathrm{d}\mu_{\Lambda}(y)=\int_0^{\operatorname{diam}_\rho(H)} \log \left|z-z'\right|\,\mathrm{d} h_*\mu_{\Lambda}(z)\geq \log (\operatorname{diam}_\rho(H)/4).
  $$
  Take $y'$ to be an inverse image of $z'$ under $h$, and we obtain
  $$
    \int_\Lambda \log d(y,y')\,\mathrm{d}\mu_{\Lambda}(y)\geq \int_\Lambda (A+\log \rho(y,y'))\,\mathrm{d}\mu_{\Lambda}(y)\geq A+\log (\operatorname{diam}_\rho(H)/4).
  $$
\end{proof}
For any two compactly supported measures $\nu_1, \nu_2$ of the same mass on $C(\mathbb{C})\setminus\{x\}$, the \emph{Wasserstein $1$-distance} is
$$
  W_1(\nu_1, \nu_2)\coloneqq\sup\limits_{f\in \mathrm{Lip}_1}\left|\int f\,\mathrm{d}\nu_1-\int f\,\mathrm{d}\nu_2\right|
$$
where $\mathrm{Lip}_1$ is the set of all functions $f$ on $C(\mathbb{C})\setminus\{x\}$ such that $\left|f(y_1)-f(y_2)\right|\leq \rho(y_1, y_2)$ for all $y_1, y_2\in C(\mathbb{C})\setminus\{x\}$.
Recall that a \emph{compact exhaustion} of a topological space $V$ is a sequence of non-empty compact subsets $\{H_i\}_{i\geq 1}$ with the following properties.
First, the union $\bigcup_{i=1}^{\infty}H_i$ equals $V$.
Moreover, for each $i\geq 1$, $H_i$ is contained in the interior of $H_{i+1}$ (relative to $V$).
\begin{proposition}\label{prop:mea-com-exh}
  Let $V$ be a path-connected set in $C(\mathbb{C})$ which is not a point.
  Assume that the closure $\overline V$ does not contain $x$.
  Let $\mu$ be a probability measure supported on $\overline V$, such that $\mu(y)=0$ for all $y\in \overline V$.
  Assume that $V$ admits a compact exhaustion $\{H_i\}_{i\in \mathbb{Z}_+}$.
  Then for every $\delta>0$, there exists a positive integer $m$ together with a probability measure $\nu_\delta$ supported on $H_m$, such that $U_{\nu_\delta}\geq U_{\mu}-\delta$, $W_1(\mu, \nu_\delta)\leq \delta$, and the potential $U_{\nu_\delta}$ is continuous.
\end{proposition}
\begin{proof}

  Since $\eta$ is continuous on $C(\mathbb{C})\times C(\mathbb{C})$, we can take $\varepsilon'>0$ such that
  \[
    \left|\eta(y_1, y_2)-\eta(y'_1, y'_2)\right|<\delta/3
  \]
  as long as $\rho(y_1, y'_1), \rho(y_2, y'_2)<2\varepsilon'$.
  Take $0<\varepsilon<\varepsilon'$.
  Take a finite cover of $\overline{V}$ consisting of open balls $B_1\coloneqq B(y_1, \varepsilon), B_2\coloneqq B(y_2, \varepsilon), \ldots, B_n\coloneqq B(y_n, \varepsilon)$ with $y_i\in V$ for all $1\leq i \leq n$.
  Here, all balls are taken with respect to the metric $\rho$.
  By shrinking $\varepsilon$, we may assume that $V$ is not contained in any of the balls.
  Set $A_{i, min}=\inf_{y_1, y_2\in B_i}\eta(y_1, y_2)$.
  For the ball $B_i$, find an arc $\Lambda_i$ in $V\bigcap\overline{B(y_i, \varepsilon)}$ connecting $y_i$ and a point in $V\bigcap\partial B(y_i, \varepsilon)$.
  Note that this can be achieved by taking the connected component of $\Lambda_i'\bigcap\overline{B_i}$ containing $y_i$ where $\Lambda_i'$ is an arc in $V$ connecting $y_i$ and a point not in the ball $B_i$.
  Let $\mu_{\Lambda_i}$ be the equilibrium measure of $\Lambda_i$.
  Then since the diameter of $\Lambda_i$ is at least $\varepsilon$, the potential $U_{\Lambda_i}\geq A_{i, min}+\log(\varepsilon/4)$ by Lemma \ref{lem:cap-low-bd}.

  By a partition of unity, we can decompose $\mu$ into measures $\mu_1, \mu_2, \ldots, \mu_n$ with $\mu_i$ supported on $B_i$ for all $i$.
  Then $W_1(\mu_i(C(\mathbb{C}))\mu_{\Lambda_i}, \mu_i)\leq 2\varepsilon\mu_i(C(\mathbb{C}))$ where $\mu_{\Lambda_i}$ is the equilibrium measure of $\Lambda_i$.
  Consider the probability measure $\mu^\varepsilon\coloneqq\sum_{i=1}^{n}\mu_i(C(\mathbb{C}))\mu_{\Lambda_i}$.
  Since $\mu^\varepsilon$ is a linear combination of equilibrium measures of arcs, its potential is continuous.
  For any $y\in C(\mathbb{C})$ with $y\neq x$, we have
  $$
    \int_{C(\mathbb{C})}\log d(y,z)\,\mathrm{d}(\mu^\varepsilon-\mu)(z) =\sum\limits_{i=1}^n\int_{C(\mathbb{C})}\log d(y,z)\,\mathrm{d}(\mu_i(C(\mathbb{C}))\mu_{\Lambda_i}-\mu_i)(z).
  $$
  Fix $\delta'>1$ such that $\log(\delta'-1)> \log(\delta'+1)-\delta/6$.
  For any $i$ with $\rho(y,y_i)\geq\delta'\varepsilon$, we have
  \begin{align*}
    &\int_{C(\mathbb{C})}\log d(y,z)\,\mathrm{d}(\mu_i(C(\mathbb{C}))\mu_{\Lambda_i}-\mu_i)(z)\\
    &\quad\geq \mu_i(C(\mathbb{C}))(\log(\rho(y,y_i)-\varepsilon)-\log(\rho(y,y_i)+\varepsilon)-2\delta/3)\\
    &\quad\geq \mu_i(C(\mathbb{C}))(\log\frac{\delta'-1}{\delta'+1}-2\delta/3).
  \end{align*}
  For any $i$ with $\rho(y,y_i)<\delta'\varepsilon$, by further shrinking $\varepsilon$ so that $\varepsilon(\delta'+1)<\varepsilon'$, we have
  \begin{align*}
    \int_{C(\mathbb{C})}\log d(y,z)\,\mathrm{d}(\mu_i(C(\mathbb{C}))\mu_{\Lambda_i}-\mu_i)(z)&\geq
    \mu_i(C(\mathbb{C}))(\log \frac\varepsilon 4-\log (\delta'\varepsilon+\varepsilon)-2\delta/3)\\
    &=-\mu_i(C(\mathbb{C}))(\log (4\delta'+4)+2\delta/3).
  \end{align*}
  Also, the ball $B_i$ is contained in the ball $B(y,\varepsilon+\delta'\varepsilon)$.
  Thus,
  $$
    \sum\limits_{\rho(y,y_i)<\delta'\varepsilon}\mu_i(C(\mathbb{C}))\leq \mu(B(y,\varepsilon+\delta'\varepsilon)).
  $$

  Suppose that there is a sequence of balls $\{B(z_i, \varepsilon_i)\}$ with $\lim\limits_{i\to\infty}\varepsilon_i=0$, such that there is a $\delta_1>0$ with $\mu(B(z_i, \varepsilon_i))>\delta_1$ for all $i$.
  Let $z$ be an accumulation point of $z_i$. Then $\mu(V)>\delta_1$ for each open set $V$ containing $z$.
  This contradicts the fact that $\mu(z)=0$.
  Thus, $\mu(B(z_i, \varepsilon_i))$ tends to $0$.

  In conclusion, we have
  \begin{align*}
    \int_{C(\mathbb{C})}\log d(y,z)\,\mathrm{d}(\mu^\varepsilon-\mu)(z)&\geq
    \log\frac{\delta'-1}{\delta'+1}-\mu(B(y,\varepsilon+\delta'\varepsilon))\log (4\delta'+4)-2\delta/3\\
    &\geq -5\delta/6-\mu(B(y,\varepsilon+\delta'\varepsilon))\log (4\delta'+4).
  \end{align*}
  As $\varepsilon$ tends to $0$, $\mu(B(y,\varepsilon+\delta'\varepsilon))$ tends to $0$ uniformly in $y$.
  Then $U_{\mu^\varepsilon}\geq U_{\mu}-\delta$.
  Also, $W_1(\mu^\varepsilon,\mu)\leq 2\varepsilon$.
  Thus, by taking $\varepsilon\leq\delta/2$, we get $W_1(\mu^\varepsilon,\mu)\leq \delta$.
  Now, since $\Lambda_i$ is a compact set in $V$, $H_m$ contains $\Lambda_i$ for $m$ large enough.
  Take $m$ such that $H_m$ contains $\Lambda_i$ for all $i$, and take $\nu_\delta=\mu^{\varepsilon}$. The proposition follows.
\end{proof}
\begin{corollary}\label{cor:gre-com-exh}
  Let $V=\bigcup_{i=1}^t V_i$ be a finite union of path-connected sets in $C(\mathbb{C})$.
  Assume that $\overline V$ does not contain $x$, that $V$ is not a finite set of points, and that $V_i$ admits a compact exhaustion for all $i$.
  Let $\{H_j\}_{j\in \mathbb{Z}_+}$ be a sequence of compact subsets of $V$ such that, for each $i$, there is a compact exhaustion $\{K_{i,j}\}_{j\in\mathbb{Z}_+}$ of $V_i$ with $K_{i,j}\subseteq H_j$ for all $j$.
  We also assume that $H_j$ is a finite union of path-connected sets.
  Then the local Green function $G(x,y;H_j)$ converges to $G(x,y;\overline{V})$ uniformly (all viewed as functions of $y$).
\end{corollary}
\begin{proof}
  We may assume that $V_i$ is not a point for all $i$.
  We may also assume that each connected component of $H_j$ is not a point.
  Now, we have $U_{\mu_{\overline V}}=-V_x(\overline V)$ on $\overline{V}$ and the same for $H_j$.
  Let $\mu_{\overline{V}}$ be the equilibrium measure of $\overline V$.
  For any $\delta_1>0$, applying Proposition \ref{prop:mea-com-exh} to $V_i$ for all $i$, we get an integer $n_{\delta_1}$ and a measure $\nu_{\delta_1}$ supported on $H_{n_{\delta_1}}$ with $U_{\nu_{\delta_1}}\geq U_{\mu_{\overline V}}-{\delta_1}$ and $W_1(\mu_{\overline V}, \nu_{\delta_1})\leq {\delta_1}$.
  For any real number $r$, set $\log_r(t)\coloneqq\max\{\log t, r\}$ for $t>0$.
  For any probability measure $\nu$, let
  $$
    U_{\nu,r}(y)\coloneqq\int_{C(\mathbb{C})} \log_r d(y,z) \,\mathrm{d}\nu.
  $$
  Note that $\log_r d(y,z)$ is a continuous function on the compact set $\overline{V}\times \overline{V}$, which can be viewed as an equicontinuous family of functions $\{f_{y}(z)=\log_r d(y,z)\}$.
  Therefore, since $\nu_{\delta_1}$ converges to $\mu_{\overline{V}}$ weakly as $\delta_1$ tends to $0$, by the Arzel\`a--Ascoli theorem, $U_{\nu_{\delta_1},r}(y)$ converges to $U_{\mu_{\overline V}, r}$ uniformly on $\overline{V}$ as $\delta_1$ tends to $0$.
  Since $\{U_{\mu_{\overline V},{-m}}\}_{m\in\mathbb{Z}_+}$ gives a decreasing sequence of continuous functions that pointwise converge to the continuous function $U_{\mu_{\overline V}}$, the convergence is uniform.
  Thus, for every $\delta_2>0$, there is a positive integer $m$ and a probability measure $\nu$ supported on $H_m$ such that $\left|U_{\mu_{\overline V}}(y)-U_{\nu}(y)\right|<\delta_2$ for all $y\in\overline V$.

  On $H_{m}$, $\left|U_{\nu}+V_x(\overline V)\right|\leq \delta_2$.
  Hence, $\left|U_{\nu}+V_x(\overline V)-U_{\mu_{H_m}}-V_x(H_m)\right|\leq \delta_2$ on $H_m$ and hence on all of $C(\mathbb{C})$ by the maximum principle.
  The corollary follows since $G(x,y; H)=U_{\mu_H}(y)+V_x(H)$.
\end{proof}
We now return to the global setting.
Let $K$ be a number field, $C$ a projective, smooth and geometrically integral curve over $K$.
Let $X$ be a set of closed points in $C$.
Let $\{\mu_{v}\}_{v\in M_K}$ be a collection of probability measures of compact support  on $C_v^\mathrm{an}$, satisfying (\ref{eq:regularization-hypothesis}) for some $D$.
Let $\mathbb{E}=(E_v)_{v\in M_K}$ be an admissible adelic set on $U$ such that $E_v\supset \mathrm{supp}(\mu_v)$.
Let $v$ be an archimedean place of $K$.
\begin{lemma}\label{lem:arch-mea-no-atoms}
  Let $x\in C_{v}^\mathrm{an}$. Then $\mu_{v}(x)=0$.
\end{lemma}
\begin{proof}
  We prove by contradiction.
  Assume that $\mu_{v}(x)>0$.
  Let $f$ be a non-constant rational function on $C$ with poles in $X$.
  For any positive integers $N$ and $M$, and any sequence of integers $\mathbf{a}\coloneqq\{a_n\}_{0\leq n\leq N}$ with $\max_{0\leq n\leq N}\{\left| a_n\right|\}<M$, consider the function $f_{\mathbf{a}}=\sum_{n=0}^{N}a_nf^n$.
  Let $S_0$ be the finite set of places consisting of all archimedean places and all non-archimedean places $v'$ such that $\left|f(y)\right|_{v'}>1$ for some $y\in\mathrm{supp}(\mu_{v'})$.
  Set $T=\max\{1, \sup_{v'\in S_0, y\in \mathrm{supp}(\mu_{v'})}|f(y)|_{v'}\}$.
  At each archimedean $v'$, write $|\cdot|_{v'}=|\cdot|_{\mathrm{usual}}^{\kappa_{v'}}$ with $\kappa_{v'}>0$; set $\kappa_{v'}=0$ at non-archimedean places.
  Put $A=\sum_{v'\in S_0}\kappa_{v'}$.
  Then we have $|f_{\mathbf{a}}(y)|_{v'}\leq (M(N+1))^{\kappa_{v'}}T^N$ for all $v'\in S_0, y\in \mathrm{supp}(\mu_{v'})$.
  Also, we have $|f_{\mathbf{a}}(y)|_{v'}\leq 1$ for all $v'\notin S_0, y\in \mathrm{supp}(\mu_{v'})$.
  Thus, we obtain
  $$
    \sum\limits_{v'\in M_K}\int_{C_{v'}^\mathrm{an}}-\log|f_{\mathbf{a}}|_{v'}\,\mathrm{d}\mu_{v'} + \mu_{v}(x)\log|f_{\mathbf{a}}(x)|_{v}\geq -A(\log M+\log(N+1))-\#S_0N\log T.
  $$

  On the other hand, the multiset $\{f_{\mathbf{a}}(x)|0\leq a_n<M \ \forall 0\leq n\leq N\}$ has $M^{N+1}$ elements and lies in a disk of radius $M(N+1)T^{N/\kappa_v}$ for the usual absolute value.
  By subdividing a square containing this disk, for $M$ sufficiently large there are two distinct coefficient vectors whose values differ by at most $M^{\kappa_v(1-N/2)}(N+1)^{\kappa_v}T^N$ in the $v$-absolute value.
  Taking their difference, we get a non-zero sequence $\mathbf{a}'$ with
  \[
    \left|f_{\mathbf{a}'}(x)\right|_v\leq M^{\kappa_v(1-N/2)}(N+1)^{\kappa_v}T^N.
  \]
  Thus, we get
  \begin{align*}
    &\sum\limits_{v'\in M_K}\int_{C_{v'}^\mathrm{an}}-\log|f_{\mathbf{a}'}|_{v'}\,\mathrm{d}\mu_{v'}\\
    &\quad\geq \mu_{v}(x)(\kappa_v(\frac N2-1)\log M-\kappa_v\log(N+1)-N\log T)\\
    &\qquad{}-A(\log M+\log(N+1))-\#S_0N\log T.
  \end{align*}
  First take $N$ large enough that $\mu_v(x)\kappa_v(N/2-1)>A$, and then take $M$ sufficiently large. The right-hand side is greater than $0$.
  This contradicts (\ref{eq:regularization-hypothesis}).
\end{proof}

\begin{proposition}\label{prop:com-exh-nor}
  Let $V=\bigcup_{i=1}^t V_i$ be a finite union of path-connected sets in $E_v$.
  Assume that $\overline{V}=E_v$, and that $V_i$ admits a compact exhaustion for all $i$.
  Let $\{H_n\}_{n\in\mathbb{Z}_+}$ be a sequence of compact subsets of $V$ such that, for each $i$, there is a compact exhaustion $\{K_{i,n}\}_{n\in\mathbb{Z}_+}$ of $V_i$ with $K_{i,n}\subseteq H_n$ for all $n$.
  Assume that each $H_n$ has finitely many path-connected components.
  Then for every $\delta>0$, there exists $n\in\mathbb{Z}_+$, with the following properties.
  Let $\mathbb{E}_n$ be the adelic set which is the same as $\mathbb{E}$ at all places except $v$, and the set at $v$ is $H_n$.
  Then $\operatorname{cap}_X(\mathbb{E}_n)>\operatorname{cap}_X(\mathbb{E})-\delta$, and there is a collection of probability measures $\{\mu_{v'}'\}_{v'\in M_K}$ supported on $\mathbb{E}_n$ with $\mu_{v'}'=\mu_{v'}$ for every $v'\neq v$, and with the following properties.
  First, $W_1(\mu_{v}, \mu_{v}')\leq \delta$.
  Moreover, for each $F=F_+-F_-$, where $F_+$ and $F_-$ are effective $\mathbb{Z}$-divisors with $\deg(F_+)=\deg(F_-)$ and $\mathrm{supp}{F_-}\subset X$, we have
  $$
    \sum_{v'\in M_K}\int_{C_{v'}^{\mathrm{an}}}g_{F,v'}\,\mathrm{d}\mu'_{v'}\leq \overline{F}\cdot D+\delta\deg(F_-)
  $$
  where $\overline{F}=(F,(g_{F,v'})_{v'\in M_K})$ is the flat arithmetic divisor.
\end{proposition}
\begin{proof}
  Note that though Corollary \ref{cor:gre-com-exh} and Proposition \ref{prop:mea-com-exh} are only proved for $\mathbb{C}$, the real case follows automatically by push-forward through the map $C_v(\mathbb{C})\longrightarrow C_v^{\mathrm{an}}$.

  By Corollary \ref{cor:gre-com-exh}, the Green's matrix $\Gamma_X(\mathbb{E}_n)$ converges to $\Gamma_X(\mathbb{E})$ as $n\longrightarrow\infty$.
  Thus, there exists $n_1>0$ such that for every $n\geq n_1$, $\operatorname{cap}_X(\mathbb{E}_n)> \operatorname{cap}_X(\mathbb{E})-\delta$.
  By Proposition \ref{prop:mea-com-exh}, applied to $V_i$ for all $i$,  there is a sequence of probability measures $\nu_n$ supported on $H_n$ such that for each $\varepsilon>0$, when $n$ is large enough, we have $U_{\nu_n}\geq U_{\mu_v}-\varepsilon$, $U_{\nu_n}$ is continuous, and $W_1(\mu_v, \nu_n)\leq \varepsilon$.
  Then $\nu_n$ converges to $\mu_v$.

  Take $n\geq n_1$.
  We also take $n$ large enough that $W_1(\mu_v, \nu_n)\leq \delta$.
  Let $y, z$ be closed points on $C$ with $y\in X$.
  Set $F=z/\deg(z)-y/\deg(y)$.
  Let $g_{F,v}$ be a flat Green function of $F$ on $C_{v}^\mathrm{an}$.
  By the Poincar\'e--Lelong formula,
  $$
    \int_{C_{v}^\mathrm{an}}g_{F,v}\,\mathrm{d} (\mu_v-\nu_n) =-\int_{C_{v}^\mathrm{an}}(U_{\mu_v}-U_{\nu_n})\,\mathrm{d}\delta_{F,v}.
  $$
  By taking $n$ large enough, $U_{\nu_n}-U_{\mu_v}\geq -\delta/2$, and we have
  $$
    -\int_{C_{v}^\mathrm{an}}(U_{\mu_v}-U_{\nu_n})\,\mathrm{d}\delta_{z, v}\geq -\delta/2.
  $$
  Since $\nu_n$ converges to $\mu_v$, as we take $n$ large enough,
  $$
    -\int_{C_{v}^\mathrm{an}}(U_{\mu_v}-U_{\nu_n})\,\mathrm{d}\delta_{y, v}\leq \delta/2
  $$
  for each $y\in X$.
  Thus,
  $$
    \int_{C_{v}^\mathrm{an}}g_{F,v}\,\mathrm{d} (\mu_v-\nu_n)\geq -\delta.
  $$
  For $F=F_+-F_-$, we can write $F$ as the sum of $\deg(F_-)$ divisors of the form $z/\deg(z)-y/\deg(y)$.
  Thus,
  $$
    \sum_{v'\in M_K}\int_{C_{v'}^{\mathrm{an}}}g_{F,v'}\,\mathrm{d}\mu'_{v'}\leq \overline{F}\cdot D+\delta\deg(F_-).
  $$

\end{proof}
For an archimedean place $v$, we say that the compact set $E_v$ is \emph{simple} if $E_v$ is a finite disjoint union of closures of path-connected open subsets of $C_v^\mathrm{an}$ and, when $K_v=\mathbb{R}$, closed intervals in $C_v(\mathbb{R})$.
\begin{corollary}\label{cor:arc-mea}
  Let $\rho$ be a metric on $C_v^\mathrm{an}$.
  In Proposition \ref{prop:com-exh-nor}, take $V=E_v$.
  For any $\delta>0$, we can take $H_n$ to be a simple set satisfying all the requirements such that the diameter of each connected component of $H_n$ with respect to $\rho$ is less than $\delta$.
\end{corollary}
\begin{proof}
  For any ${\varepsilon_1}>0$, let $\mathcal{T}_{\varepsilon_1}$ be a triangulation of $C_{v}^\mathrm{an}$, with each edge a smooth arc and the diameter of each triangle less than ${\varepsilon_1}$.
  Let $X_{\varepsilon_1}$ be the set of all triangles contained in the interior of $E_v$.
  If $K_{v}=\mathbb{R}$, let $Y_{\varepsilon_1}$ be the set of edges that are contained in the interior of $C_{v}(\mathbb{R})\cap E_v$ in $C_v(\mathbb{R})$, and let $Y_{\varepsilon_1}'$ be the subset of $Y_{\varepsilon_1}$ containing the edges that are not edges of any triangle in $X_{\varepsilon_1}$.
  If $K_{v}=\mathbb{C}$, let $Y_{\varepsilon_1}$ be the empty set.
  Set $U_{\varepsilon_1}=\bigcup_{\sigma \in X_{\varepsilon_1}}\sigma$ and $V_{\varepsilon_1}=\bigcup_{\lambda \in Y_{\varepsilon_1}}\lambda$.
  Then as $\varepsilon_1$ tends to $0$, $U_{\varepsilon_1}$ gives a compact exhaustion of the interior of $E_v$, and $V_{\varepsilon_1}$ gives a compact exhaustion of the interior of $C_{v}(\mathbb{R})\cap E_v$ in $C_{v}(\mathbb{R})$.

  For any $\varepsilon_2>0$, $\sigma\in X_{\varepsilon_1}$, let $\sigma_{\varepsilon_2}$ be a triangle with each edge a smooth arc in the interior of $\sigma$, such that $\rho(y, \sigma_{\varepsilon_2})<\varepsilon_2$ for all $y\in\sigma$.
  For any $\lambda \in Y_{\varepsilon_1}'$, let $\lambda_{\varepsilon_2}$ be a closed interval in the interior of $\lambda$ in $C_{v}(\mathbb{R})$ that contains all points whose distance from the endpoints of $\lambda$ is greater than $\varepsilon_2$.
  We apply Proposition \ref{prop:com-exh-nor} twice: first to $U_{\varepsilon_1}$, and then to $\bigcup_{\sigma \in X_{\varepsilon_1}} \sigma_{\varepsilon_2}\bigcup_{\lambda \in Y_{\varepsilon_1}'}\lambda_{\varepsilon_2}$.
  Note that the union is disjoint, with the diameter of each set less than $\varepsilon_1$.
  The corollary follows.
\end{proof} \subsection{Non-archimedean case}
Now we focus on non-archimedean places.
Let $v$ be a non-archimedean place of $K$.
\begin{proposition}\label{prop:non-mea}
  There is a sequence of measures $\{\nu_n\}_{n\in \mathbb{Z}_+}$ on $C_v^\mathrm{an}$, each of which is a finite linear combination of Dirac measures at divisorial points on $E_v$, such that
  \begin{enumerate}
    \item $\lim_{n\to\infty}\nu_n=\mu_v$,
    \item for every subharmonic function $g$ on $E_v$,
      $$
        \int_{E_v} g\,\mathrm{d}\mu_v\leq \int_{E_v} g\,\mathrm{d}\nu_n
      $$
      for every $n$.
  \end{enumerate}
\end{proposition}
\begin{proof}
  First, we reduce to the semistable case to ensure that there are enough sncd-models.
  Let $\mathcal{C}$ be a projective model of $C$ over $\operatorname{Spec}(O_{K_v})$, and let $U$ be an open subscheme of the special fiber that defines $E_v$.
  Let $K'$ be a field extension of $K_v$ such that $C_{K'}$ is semistable.
  Let $\mathcal{C}'$ be a semistable model of $C_{K'}$ over $\operatorname{Spec}(O_{K'})$.
  After blowing up, we may assume that the rational map $\mathcal{C}'\longrightarrow\mathcal{C}$ is a morphism.
  Let $U'$ be the base change $U\times_{\mathcal{C}}\mathcal{C}'$.
  Let $E'$ be the locally admissible compact set defined by $U'$.
  Then $E'$ is the inverse image of $E_v$.
  Let $\mu'$ be the pull-back of the measure $\mu_v$ under the map $C_{K'}^\mathrm{an}\longrightarrow C_v^\mathrm{an}$.
  Assume that we have proved the proposition for $E'$ and $\mu'$ on $C_{K'}^\mathrm{an}$. After taking the push-forward of the sequence of measures, we obtain a sequence of measures for the original proposition.

  Now, we assume that $C_v$ is semistable.
  Note that both semistable models and blow-ups of sncd-models are sncd-models.
  Let $\mathcal{C}_1$ be a semistable model of $C$ over $\operatorname{Spec}(O_{K_v})$, $U$ the open subscheme of the special fiber that defines $E_v$.
  We may blow up regular points on the special fiber, so we may assume that the special fiber is not irreducible.
  Since there are only countably many projective models of $C$, we may take $\{\mathcal{C}_n'\}$ to be a sequence containing all projective models of $C$.
  Let $\mathcal{C}_{n+1}$ be a model obtained by blowing up $\mathcal{C}_n$, such that $\mathcal{C}_{n+1}$ admits a regular map to $\mathcal{C}_n'$.
  Let $\Gamma_n$ be the dual graph of the special fiber of $\mathcal{C}_n$.
  Recall that the metrized graph $\Gamma_n$ embeds canonically in $C_v^\mathrm{an}$, and admits a deformation retraction $\tau_n: C_v^\mathrm{an}\longrightarrow\Gamma_n$.
  Therefore, we get a sequence of sncd-models $\{\mathcal{C}_n\}$ of $C$.

  Fix $n$.
  Let $U_n$ be the open subscheme of the special fiber of $\mathcal{C}_n$ that defines $E_v$.
  Let $W_1, W_2, \ldots, W_m$ be connected affine open subschemes of the special fiber of $\mathcal{C}_n$ containing exactly one double point, such that $\bigcup_{i=1}^mW_i=U_n$.
  Then $\bigcup_{i=1}^{m}\mathrm{red}^{-1}(W_i)=E_v$.
  We may decompose $\mu_v$ into a sum of measures $\sum_{i=1}^m\mu_i$ where $\mu_i$ is supported on $\mathrm{red}^{-1}(W_i)$.
  By Proposition \ref{prop:non-sub-avg}, for each $i$, there exists a measure $\mu'_i$ supported on $\partial \mathrm{red}^{-1}(W_i)$ with the same total mass as $\mu_i$ such that
  $$
    \int_{\mathrm{red}^{-1}(W_i)}g\,\mathrm{d}\mu'_i\geq\int_{\mathrm{red}^{-1}(W_i)}g\,\mathrm{d}\mu_i
  $$
  for each subharmonic function $g$ on $\mathrm{red}^{-1}(W_i)$.
  It is easy to see that if $g$ is harmonic, equality holds in the above inequality.
  Let $\nu_n=\sum_{i=1}^{m}\mu'_i$.
  Note that each compact set $\mathrm{red}^{-1}(W_i)$ contains exactly one edge in $\Gamma_n$, and its boundary consists of the endpoints of the edge.

  It remains to prove that $\nu_n$ converges weakly to $\mu_v$.
  A model function is a function that is the pull-back of a piecewise-affine function on a graph $\Gamma$ under the deformation retraction, where $\Gamma$ is the dual graph of an sncd-model.
  By \cite[Theorem 7.12]{Gub98}, model functions are dense in the set of all continuous functions on $C_v^\mathrm{an}$.
  Thus, we only need to prove that for every model function $f$, $\lim_{n\to\infty}\int_{E_v} f\,\mathrm{d} \nu_n=\int_{E_v}f \,\mathrm{d}\mu_v$.
  By our choice of $\mathcal{C}_n$, we may assume that $f$ is defined by a piecewise-affine function on $\Gamma_n$ for every $n\geq n_0$.
  Then we have that $\int_{E_v} f\,\mathrm{d} \nu_n=\int_{E_v}f \,\mathrm{d}\mu_v$ for all $n\geq n_0$ since $f$ is harmonic outside the vertices of $\Gamma_n$.
\end{proof}
\begin{proposition}\label{prop:non-exh}
  Let $x$ be a divisorial point on $C_v^\mathrm{an}$. Then there exists a decreasing sequence of compact sets $\{E_n\}$, locally admissible in $C_v^{\mathrm{an}}$ and containing $x$, such that every sequence of probability measures $\{\mu_n\}$ with $\mu_n$ supported on $E_n$ converges to $\delta_x$ weakly.
\end{proposition}
\begin{proof}
  Take a projective model $\mathcal{C}_1$ of $C$ over $O_{K_v}$ such that $x$ is a vertex in the dual graph of $\mathcal{C}_1$.
  Take a sequence of projective models $\{\mathcal{C}_n\}$ such that $\mathcal{C}_{n+1}$ is obtained by blowing up $\mathcal{C}_n$ and that each projective model $\mathcal{C}$ admits a regular map $\mathcal{C}_n\longrightarrow \mathcal{C}$ for $n$ sufficiently large, as in the proof of the previous proposition.
  Let $W_n$ be the intersection of the component corresponding to $x$ with the regular locus of the reduced special fiber of $\mathcal C_n$.
  Let $E_n$ be the compact set corresponding to $W_n$, locally admissible in $C_v^{\mathrm{an}}$.
  Then each model function is constant on $E_n$ for $n$ large enough.
  Since model functions are dense, $\mu_n$ converges to $\delta_x$ weakly.
\end{proof}
\section{Local construction}
\label{sec:local-construction}

In this section we prove technical lemmas and propositions concerning the construction of local functions. They will be used in the proof of Theorem \ref{thm:integral-point-construction} in Section \ref{sec:integral-point-construction}.

The main result is Proposition \ref{prop:loc-sec}.
We also refer the reader to \cite[Chapters 7--10]{Rum13} since some basic ideas are similar.

\subsection{Statements}
Let $K$ be a local field and $C$ a projective, smooth and geometrically integral curve of genus $g$ over $K$.
Let ${X}$ be a finite set of closed points on $C$.
Let $E$ be a compact set in $C^{\mathrm{an}}\setminus {X}$.
If $K$ is archimedean, assume that $E$ is simple.
If $K$ is non-archimedean, assume that $E$ is locally admissible.

For the case $K={\mathbb {R}}$, let $H$ be a closed interval in $C({\mathbb {R}})$.
For a continuous real-valued function $g$ on $H$, a non-zero rational function $f$ on $C$ is said to be \emph{$g$-oscillating} on $H$ if $H$ contains no poles and only simple zeros of $f$, and each connected component of the set $\{z\in H|\log\left|f\right|<g\}$ contains at most one point which is either a zero of $f$ or an endpoint of $H$.
The union of those connected components that contain a zero of $f$ is called the \emph{oscillating interval}.
For non-zero rational functions $f_1, f_2, \ldots, f_m$, we say that $(f_1, g_1), \ldots, (f_m, g_m)$ \emph{oscillate compatibly} if each $f_j$ is $g_j$-oscillating on $H$ and on the oscillating interval, $\log\left|f_i\right|\geq g_i$ for each $i\neq j$.
Let $f_1, f_2, \ldots, f_m$ be rational functions on $C$ whose pole divisors are supported outside $E$.
For continuous functions $g_1, g_2, \ldots, g_m$ on $E$, we say that $(f_1, g_1), (f_2, g_2), \ldots, (f_m, g_m)$ satisfy the \emph{boundary dominance condition} if the following conditions hold:
\begin{enumerate}
  \item For each $i$, all zeros of $f_i$ lie in $E$;
  \item for each $i$, $\log\left|f_i\right|\geq g_i$ on $\partial E'$ where $E'$ is a connected component of $E$ which is not a real interval;
  \item if $K={\mathbb {R}}$, for each connected component $E'$ of $E$ which is a real interval, $(f_1, g_1),\allowbreak \dots,\allowbreak (f_m, g_m)$ oscillate compatibly on $E'$.
\end{enumerate}

If $K\neq {\mathbb {R}}$, the boundary dominance condition can be verified independently.
Our definition ensures that, even for the case $K={\mathbb {R}}$, if $(f_1, g_1), (f_2, g_2), \ldots, (f_m, g_m)$ satisfy the boundary dominance condition, then so do $(f_1f_2, g_1+g_2), (f_3, g_3), \ldots, (f_m, g_m)$.

Let $\mathbf{s}=(s_x)_{x\in {X}}$ be a vector of positive rational numbers.
Let $N_1$ be the smallest positive integer such that $N_1s_x$ is an integer at least $2g+2$ for every $x\in X$.
Let $N$ be a multiple of $N_1$.
A \emph{band filtration} for $\mathbf{s}, N$ is a decreasing sequence $\{D_N^{(j)}\}$ of divisors, as follows.
Firstly, we have $D_N^{(0)}=\sum_{x\in{X}}Ns_xx$ and $D_N^{(1)}=\sum_{x\in{X}}(Ns_x-1)x$.
Moreover, for each $0\leq N'<N$ which is a multiple of $N_1$, $\sum_{x\in{X}}(N's_x+2g)x$ appears in the sequence.
The sequence ends with $D_N^{(j_N)}=\sum_{x\in{X}}2gx$.
Lastly, $D_N^{(j)}-D_N^{(j+1)}$ is a sum of distinct prime divisors.

For simplicity, we write $H^0(D)$ for $H^0(C, {\mathcal {O}}(D))$.
Set $V_N^{(j)}=H^0(D_N^{(j)})/H^0(D_N^{(j+1)})$.
For each point $x\in {X}$, take a Green function $g_x$ of $x$.
For $\phi_N^{(j)}\in V_N^{(j)}$, $x\in{\mathrm{supp}}(D_N^{(j)}-D_N^{(j+1)})$ and $d={\mathrm{ord}}_x(D_N^{(j)})$, set
$$
  \left|\phi_N^{(j)}\right|_{x, d}=\lim_{x'\to x}\left(\left|\phi_N^{(j)}(x')\right|e^{-dg_{x}(x')}\right).
$$
By the Riemann--Roch theorem, $V_N^{(j)}=\oplus_{x\in{\mathrm{supp}}(D_N^{(j)}-D_N^{(j+1)})}K(x)$, and the norm $|\cdot|_{x, {\mathrm{ord}}_x(D_N^{(j)})}$ factors through the projection to $K(x)$.
Set $V_N=\oplus_{j=0}^{j_N-1} V_N^{(j)}$.

Let $\phi_N^{\mathrm{init}}$ be an element in $V_N^{(0)}$.
Set $b_x=\log\left|\phi_N^{\mathrm{init}}\right|_{x, Ns_x}$, and ${\mathbf{b}}=(b_x)_{x\in {X}}$.
Take ${\mathbf{c}}=(c_x)\in{\mathbb {R}}^{X}$ and $B\in {\mathbb {R}}$.
The element $\phi_N^{\mathrm{init}}$ is said to be \emph{represented by an $n$-th power} if there exists a rational function $f$, such that $f^n$ represents $\phi_N^{\mathrm{init}}$.
A \emph{patching vector} with initial element $\phi_N^{\mathrm{init}}$ bounded by $B, {\mathbf{c}}$ is an element $\boldsymbol{\phi}_N=(\phi_N^{(j)})\in V_N$ with the following property.
For $j\geq1$, $\left|\phi_N^{(j)}\right|_{x, {\mathrm{ord}}_x(D_N^{(j)})}\leq Be^{b_x-(Ns_x-d)c_x}$ where $d={\mathrm{ord}}_x (D_N^{(j)})$.
For $j=0$, if $K$ is non-archimedean, $\left|\phi_N^{(0)}-\phi_N^{\mathrm{init}}\right|_{x, Ns_x}\leq Be^{b_x}$.
If $K(x)={\mathbb {C}}$, $\left|\phi_N^{(0)}\right|_{x, Ns_x}= e^{b_x}$.
If $K(x)={\mathbb {R}}$, $\phi_N^{(0)}$ is the same as $\phi_N^{\mathrm{init}}$ after the projection to $K(x)$.

A \emph{patching map} is a map $\psi_N: V_N\longrightarrow H^0(D_N^{(0)})/H^0(D_N^{(j_N)})$ with the following properties.
First, $\psi_N(\mathbf{0})=0$.
Moreover, let $\boldsymbol{\phi}_N=(\phi^{(j)}_N)$ and $\boldsymbol{\phi}'_N=({\phi'}^{(j)}_N)$ be patching vectors.
If $\phi^{(j)}_N={\phi'}^{(j)}_N$ for each $j<j_0$, then $\psi_N (\boldsymbol{\phi}_N) -\psi_N(\boldsymbol{\phi}'_N)-\phi^{(j_0)}_N +{\phi'}^{(j_0)}_N=0$ in $H^0(D_N^{(0)})/H^0(D_N^{(j_0+1)})$.

Let $g_{E, x}$ be the equilibrium potential of $E$ at $x$.
Let $D$ be an effective ${\mathbb {Q}}$-divisor of degree $1$ on $C$ supported outside $E$ and ${X}$.
Let $g_D$ be a Green function of $D$.
For $K=\mathbb R$, the $\mathbb R$-interior of $E$ is $\operatorname{int}_{C^{\mathrm{an}}}(E)\cup\operatorname{int}_{C(\mathbb R)}(E\cap C(\mathbb R))$; for $K=\mathbb C$, it is the ordinary interior.
Assume that $c_1(g_D)=\mu$ is a probability measure supported on the $\mathbb R$-interior of $E$ if $K$ is archimedean, on finitely many type~II points otherwise.
For each $x, y\in {X}$ with $x\neq y$, set $c_{x,y}=g_{E, x}(y)$.
Set $c_{x,x}=\lim_{x'\to x}(g_{E, x}(x')-g_{x}(x')/\deg(x))$.
Set $a_x=g_D(x)$.

\begin{proposition}\label{prop:loc-sec}
  For each $\varepsilon>0$, there exist real numbers $c'_{x, y}$ such that $\left|c'_{x, y}-c_{x, y}\right|\leq \varepsilon$.
  If $K$ is archimedean, each $c'_{x, x}$ can be taken independently in a subinterval of $[c_{x, x}-\varepsilon, c_{x, x}+\varepsilon]$, and its choice will not affect the choice of any further constants.
  Still, if $K$ is archimedean, let $\{\kappa_i\}$ be a subsequence of $\{i!\}$. Then there exists a subsequence $\{\kappa'_i\}$ of $\{\kappa_i\}$.
  There exist real numbers $a'_x$ such that $\left|a_{x}-a'_x\right|\leq \varepsilon$.
  Let $s_x$ be positive rational numbers.
  Set $c'_x=\sum_{y\in{X}}s_y\deg(y)c'_{x, y}/s_x$ and ${\mathbf{c}}'=(c'_x)_{x\in {X}}$.
  For the non-archimedean case, there exists $B>0$.
  For the archimedean case, $B>0$ can be taken arbitrarily large.

  For each integer $n>0$, there exists $N_1>0$.
  Let $N$ be a multiple of $N_1$.
  If $K$ is archimedean, further assume that $N$ appears in the sequence $\{\kappa'_i\}$.
  Then for each band filtration for ${\mathbf{s}}=(s_x)_x$ and $N$, there exists a patching map $\psi_N$ and an initial function $\phi_N^{\mathrm{init}}\in V_N^{(0)}$ with the following properties.
  First, $\log \left|\phi_N^{\mathrm{init}}\right|_{x, Ns_x}=Na'_x+Ns_xc'_x$.
  Second, $\phi_N^{\mathrm{init}}$ is represented by an $n$-th power.
  Lastly, for every patching vector $\boldsymbol{\phi}_N$ with initial element $\phi_N^{\mathrm{init}}$ bounded by $B, {\mathbf{c}}'$, there is a representative $G_N$ of $\psi_N(\boldsymbol{\phi}_N)$ with pole divisor $D_N^{(0)}+ND$ such that $(G_N, Ng_D)$ satisfies the boundary dominance condition on $E$ and that the number of zeros of $G_N$ in $E'$ is at least $N(\mu(E')-\varepsilon)$ for each connected component $E'$ of $E$.
\end{proposition}

\subsection{Archimedean case}
Let $K$ be ${\mathbb {R}}$ or ${\mathbb {C}}$, $C$ a projective, smooth and geometrically integral curve of genus $g$ over $K$, $x$ a point in $C^{\mathrm{an}}$ and $E$ a simple set in $C^{\mathrm{an}}\setminus\{x\}$.
\subsubsection{Rational functions satisfying the boundary dominance condition}
First, we study the generating property of the Jacobian.
\begin{lemma}\label{lem:uni-apr-com}
  Let $\{\mu_i\}_{i\in {\mathbb {Z}}_+}$ be a sequence of probability measures supported on $E$ which converges to a probability measure $\mu$ weakly.
  Let $\Gamma$ be a compact set in $C^{\mathrm{an}}$ which is disjoint from $E$ and $x$.
  Then $U_{\mu_i}$ converges to $U_\mu$ uniformly on $\Gamma$.
\end{lemma}
\begin{proof}
  Consider the continuous function
  \begin{align*}
    f:E\times \Gamma&\longrightarrow{\mathbb {R}}\\
    (x_1,x_2)&\longmapsto G_x(x_1,x_2).
  \end{align*}
  We have
  $$
    U_{\mu_i}(x_2)=\int_{E}f(x_1,x_2)\,\mathrm{d}\mu_i(x_1).
  $$
  Since $f$ is continuous, we get that $U_{\mu_i}$ converges pointwise to $U_\mu$ on $\Gamma$.
  Also, since $E$ is compact, the family of functions $\{f(x_1, \cdot)\}_{x_1\in E}$ is equicontinuous.
  Thus, the family of functions $\{U_{\mu_i}\}_{i\in {\mathbb {Z}}_+}$ is also equicontinuous.
  By the Arzel\`a--Ascoli theorem, since $\Gamma$ is compact, $U_{\mu_i}$ converges to $U_\mu$ uniformly on $\Gamma$.
\end{proof}
\begin{lemma}\label{lem:gen-jac-c}
  Let $U$ be a non-empty open set in $C({\mathbb {C}})$. Then there exists an integer $n$ such that for every divisor $D$ on $C_{\mathbb {C}}$ of degree $n$, there exists an effective divisor $D'$ supported on $U$ rationally equivalent to $D$.
\end{lemma}
\begin{proof}
  For simplicity of notation, we assume that $K={\mathbb {C}}$.
  If $g=0$, there is nothing to prove.
  Assume that $g>0$.

  Let $J_C$ be the Jacobian of $C$.
  Then $J_C={\mathbb {R}}^{2g}/\Lambda$ as a Lie group where $\Lambda$ is a full rank lattice in ${\mathbb {R}}^{2g}$.
  Let $\alpha: C^g\longrightarrow J_C$ be the map $(x_1, x_2, \ldots, x_g)\longmapsto (x_1+x_2+\cdots+x_g-gx)$.
  Let $U'$ be the smooth locus of $\alpha$ in $U^g$.
  Then $\alpha(U')$ is a non-empty open set in $J_C$.
  For $m$ sufficiently large, $m\alpha(U')$ covers a fundamental domain of $\Lambda$, and is thus all of $J_C$.
  The lemma follows by taking $n=gm$.
\end{proof}

For the next lemma, assume $K=\mathbb R$ and let $H$ be a closed interval in $C(\mathbb R)\setminus\{x\}$.
By \cite[Proposition 3.11(A)(2)]{Rum13}, there is a continuous bijection $h:I=[0,1]\longrightarrow H$ and a continuous function $\eta: I^2\longrightarrow {\mathbb {R}}$ such that $G_x(h(a), h(b))=\log\left|a-b\right|+\eta(a,b)$ for all distinct $a,b\in I$.
Let $B$ be the supremum of $\left|\eta\right|$.
For a continuous function $g$ on $H$, an upper semicontinuous function $g': H\longrightarrow {\mathbb {R}}\bigcup\{-\infty\}$ is said to be \emph{$g$-oscillating} if every connected component of the open set defined by $g'-g<0$ in $H$ contains at most $1$ point $z\in H$ such that either $g'(z)=-\infty$ or $z$ is an endpoint of $H$.
\begin{lemma}\label{lem:gen-jac-r}
  There exist an integer $n$ and a real number $M$, such that for every divisor $D$ of degree $n$, there is an effective divisor $D'$ which is the sum of distinct points on $H$ and is rationally equivalent to $2D$ such that $U_{\delta_{D'}}$ is $M$-oscillating on $H$.
  Here, $M$ is viewed as the constant function on $H$.
\end{lemma}
\begin{proof}
  Fix a point $y\in H$. Let $W$ be an open set in $C^{\mathrm{an}}$ whose closure is disjoint from $x$ and $H$.
  By Lemma \ref{lem:gen-jac-c}, there is a number $m$, such that for every divisor $D$ of degree $m$, $2D$ is rationally equivalent to an effective divisor supported on $W$.
  Then $m\geq g$.

  Consider the space
  $$
    V:=\prod\limits_{i=1}^{g} \left(\frac{4i-3}{8m},\frac{4i-1}{8m}\right).
  $$
  Let $\alpha: h(V)\longrightarrow J_C$ be the map $\alpha(x_1, x_2, \ldots, x_{g})=(x_1+x_2+\cdots+x_{2m}-2my)$ where $x_i=h((2i-1)/4m)$ for $i>g$.
  Let $V'\subset V$ be the inverse image under $h$ of the smooth locus of $\alpha$ in $h(V)$.
  Then $\alpha(h(V'))$ is open in $J_C({\mathbb {R}})$.
  Note that the identity component of $J_C({\mathbb {R}})$ is a real torus of dimension $g$, and the connected components of $J_C({\mathbb {R}})$ form a finite abelian group of exponent $2$.
  Let $m'$ be an even integer such that $m'\alpha(h(V'))$ is the whole identity component of $J_C({\mathbb {R}})$.
  Let $n=mm'$.

  On $H\times \overline{W}$, $G_x$ is continuous, and hence bounded.
  Thus, $U_\mu$ is bounded on $H$ for all probability measures $\mu$ supported on $\overline{W}$ by a constant $M'$ independent of the choice of $\mu$.
  For any divisor $D$ of degree $n$, $2D-2ny$ lies in the identity component of $J_C$.
  Thus, $2D-2ny=m'D_1$ for some $D_1\in \alpha(h(V'))$.
  Since $D_1\in J_C({\mathbb R})^0$ and multiplication by $2$ on this real torus is surjective, choose a real divisor $A$ of degree $m$ with $2A\sim D_1+2my$. By the choice of $m$, there is an effective divisor $D_2$ supported on $W$ with $D_2\sim 2A\sim D_1+2my$.
  Let $(x_1, x_2, \ldots, x_g)\in \alpha^{-1}(D_1)$ be an element in $h(V')$.
  Set $D_3=x_1+x_2+\cdots+x_{2m}$ where $x_i=h((2i-1)/4m)$ for $i>g$.
  Since $D_3$ is linearly equivalent to $D_2$, let $f$ be a rational function with ${\mathrm{div}}(f)=D_3-D_2$, normalized such that $\left|f(x)\right|=1$.
  Then $\log\left|f\right|=U_{\delta_{D_3}}-U_{\delta_{D_2}}$.

  Let $r=i/(2m)$ for some integer $0\leq i\leq 2m$.
  Then $G_x(h(r),z)\geq -\log8-\log m-B$ for all $z\in{\mathrm{supp}}(D_3)$.
  Thus, $U_{\delta_{D_3}}(h(r))\geq 2m(-\log8-\log m-B)$.
  Also, $U_{\delta_{D_2}}(h(r))\leq 2mM'$.
  Therefore, $\log\left|f(h(r))\right|\geq 2m(-\log8-\log m-B)-2mM'$.
  Let $M_1=\min\{-1, 2m(-\log8-\log m-B)-2mM'\}$.

  Let $T_i$ be the $i$th Chebyshev polynomial of the first kind, that is, $\cos\circ(i \arccos)$.
  Let $f_1(z)=T_{m'}(e^{-M_1}f(z))/\left|T_{m'}(e^{-M_1}f(x))\right|$.
  Then $\left|f_1(x)\right|=1$.
  Note that $\left|T_{m'}(e^{-M_1}f(x))\right|\leq (2e^{-M_1}+2)^{m'}$.
  Since $f$ is $M_1$-oscillating on $H$, $f_1$ is $-m'\log (2e^{-M_1}+2)$-oscillating.
  Let $D'$ be the zero divisor of $f_1$.
  Then $U_{D'}=m'U_{D_2}+\log\left|f_1\right|$.
  Since $m'U_{D_2}(z)\geq -2nM'$ for all $z\in H$, $U_{D'}$ is $-m'\log (2e^{-M_1}+2)-2nM'$-oscillating.
\end{proof}

Let $\mu$ be a probability measure supported in the interior of $H$ relative to $C({\mathbb {R}})$.
Set $\nu=h^{-1}_*(\mu)$.
Fix $c<1$, such that $\nu$ is supported on $[0,c]$.
Let $r_i=\inf_{\nu([0,r])\geq i/n}r$ for every $1\leq i\leq n$.
Let $\{s_i\}_{1\leq i\leq n}$ be the smallest strictly increasing sequence of integers such that $s_i\geq m r_i$.
For positive integers $m, n$ and real numbers $0<\alpha, \alpha_1, \alpha_2<1/2$ with $\alpha_1<\alpha_2$, set $\delta'_{m, n, \alpha}=\sum_{i=1}^{n}\delta_{(s_i+\alpha)/m}/n$, and set $\nu_{m,n,\alpha_1,\alpha_2}=\sum_{i=1}^{n}m\mu^{Leb}_{[(s_i+\alpha_1)/m, (s_i+\alpha_2)/m]}/(n(\alpha_2-\alpha_1))$ where $\mu^{Leb}_{[a, b]}$ is the Lebesgue measure on $[a, b]$.
If $n\leq (1-c)m$, then both $\delta'_{m, n, \alpha}$ and $\nu_{m,n,\alpha_1,\alpha_2}$ are supported on $[0,1]$.  Set $\delta_{m, n, \alpha}=h_*\delta'_{m, n, \alpha}$ and $\mu_{m,n,\alpha_1,\alpha_2}=h_{*}\nu_{m,n,\alpha_1,\alpha_2}$.

For every $d\in (0,1]$, set $B_d=\sup_{i_1, i_2, i_3\in I, \left|i_1-i_2\right|\leq d}\left(\eta(i_1,i_3)-\eta(i_2,i_3)\right)$.
Let $B$ be the supremum of $\left|\eta\right|$.
As $\eta$ is continuous on $I^2$, $B_d$ tends to zero as $d$ tends to $0$.
\begin{lemma}\label{lem:apr-dir-int}
  Assume that $U_\mu$ is bounded below by a real number $L$, and that $n\leq (1-c)m$.
  Fix $\alpha<1/2$.
  Then for every real number $0<r<1/2$ and every $0\leq s\leq 1$ such that $\left|s-(s_i+\alpha)/m\right|\geq r/m$ for every $1\leq i\leq n$, the following inequality holds:
  \begin{align*}
    U_{\delta_{m, n, \alpha}}(h(s))&\geq U_\mu(h(s))-\frac{B_1}{\sqrt{n}}-B_{1/\sqrt{n}}-B_t+\frac5n\left(\log r-\log m\right)+\log\left(1-\sqrt{t}\right)\\
    &-\frac{2(L-B-\frac1n\log m)\log (n+\tfrac32) }{\log t}
  \end{align*}
  where $t=(n+\tfrac12)/m$.
\end{lemma}
\begin{proof}
  Let $n_1$ be the greatest integer such that $r_{n_1}\leq s-1/m$.
  If such integers do not exist, set $n_1=0$.
  Then
  \begin{align*}
    \sum\limits_{i=1}^{n_1}\left(\frac1n\log\left|r_i-s\right|-\int_{r_{i-1}}^{ r_i}\log\left|\beta-s\right|d\nu(\beta)\right)
    &\geq \sum\limits_{i=1}^{n_1}\frac1n \left(\log\left|r_i-s\right|-\log\left|r_{i-1}-s\right|\right)\\
    &\geq -\frac1n\log m
  \end{align*}
  where $r_0$ is viewed as $0$.
  Let $n_2$ be the greatest integer such that $r_{n_2}\leq s-\sqrt{t}$.
  If such integers do not exist, set $n_2=0$.
  Then
  \begin{align*}
    \sum\limits_{i=n_2+1}^{n_1}\frac1n \log \left|r_i-s\right|&\geq \sum\limits_{i=1}^{n_1} \frac1n\log \left|r_i-s\right|\\
    &\geq \sum\limits_{i=1}^{n_1}\int_{r_{i-1}}^{ r_i}\log\left|r-s\right|d\nu(r) -\frac1n\log m\\
    &\geq \int_0^1 \log\left|r-s\right|d\nu(r) -\frac1n\log m\\
    &\geq \int_0^1 G_x(h(r) ,h(s))d\nu(r) - B -\frac1n\log m\\
    &\geq L-B-\frac1n\log m.
  \end{align*}

  Let $X$ be the set of all integers $i$ such that $n_2+1\leq i\leq n_1$ and that $\left|(s_i+\alpha)/m-s\right|\geq 1/m$.
  Take $i\in X$.
  Since $0\leq (s_i+\alpha)/m-r_i<t$ and $\left|(s_i+\alpha)/m-s\right|\geq 1/m$, we obtain $\log\left|(s_i+\alpha)/m-s\right|\geq \log\left|r_i-s\right|-\log (n+\tfrac32)$.
  On the other hand, for every $s-\sqrt{t}< r_i\leq s -1/m$, $\log\left|r_i-s\right|\leq \log t/2$.
  Thus, we have
  $$
    \log\left|(s_i+\alpha)/m-s\right|/\log\left|r_i-s\right|\leq 1-2\log (n+\tfrac32)/\log t.
  $$
  And we have
  \begin{align*}
    \sum\limits_{i\in X}\frac1n \log \left|\frac{s_i+\alpha}m -s\right|&\geq \left(1-2\frac{\log (n+\tfrac32)}{\log t}\right) \sum\limits_{i\in X}\frac1n \log \left|r_i-s\right|\\
    &\geq \sum\limits_{i\in X}\frac1n \log \left|r_i-s\right| -\frac{2 (L-B-\frac1n\log m)\log(n+\tfrac32)}{\log t}.
  \end{align*}
  Also, we have
  $$
    \sum\limits_{i=1}^{n_2}\frac1n \log \left|\frac{s_i+\alpha}m -s\right|\geq \sum\limits_{i=1}^{n_2}\frac1n \log \left|r_i -s\right|+\log\left(1-\sqrt{t}\right).
  $$

  For $i\geq n_1+3$,
  $$
    \frac1n\log\left|\frac{s_i+\alpha}m-s\right|-\int_{r_{i-1}}^{ r_i}\log\left|r-s\right|d\nu(r)\geq 0.
  $$
  For $n_2+1\leq i\leq n_1+2$ with $i\notin X$, which has at most $4$ choices,
  $$
    \frac1n\log\left|(s_i+\alpha)/m-s\right|-\int_{r_{i-1}}^{ r_i}\log\left|r-s\right|d\nu(r)\geq \frac1n\left(\log r-\log m\right).
  $$
  Thus, we have
  \begin{align*}
    \sum\limits_{i=1}^{n} \left(\frac1n\log\left|\frac{s_i+\alpha}m-s\right|-\int_{r_{i-1}}^{ r_i}\log\left|r-s\right|d\nu(r)\right)&\geq \frac5n\left(\log r-\log m\right)+\log\left(1-\sqrt{t}\right)\\
    &-\frac{2(L-B -\frac1n\log m)\log (n+\tfrac32)}{\log t}.
  \end{align*}

  Finally, since $G_x(h(r), h(s))-\log\left|r-s\right|=\eta(r,s)$, we get
  \begin{align*}
    \frac1n\left(\log|r_i-s|-G_x(h(r_i),h(s))\right)-\int_{r_{i-1}}^{ r_i}\log\left|r-s\right|d\nu(r)+\int_{r_{i-1}}^{r_i}G_x(h(r), h(s))d\nu(r)\\
    =-\frac1n\eta(r_i,s)+\int_{r_{i-1}}^{r_i}\eta(r, s)d\nu(r)\leq \frac {B_{r_i-r_{i-1}}}n,
  \end{align*}
  and
  \begin{align*}
    \log\left|(s_i+\alpha)/m-s\right|-G_x(h\left((s_i+\alpha)/m\right),h(s))-\log|r_i-s|+G_x(h(r_i),h(s))\\
    =-\eta((s_i+\alpha)/m,s)+\eta(r_i, s)\leq B_t.
  \end{align*}
  Since the number of $i$ such that $r_{i-1}\leq r_i-1/\sqrt{n}$ is less than or equal to $\sqrt{n}$, we obtain
  \begin{align*}
    U_{\delta_{m,n, \alpha}}\left(h(s)\right)- U_\mu\left(h(s)\right)&
    \geq \sum\limits_{i=1}^{n}
    \left(\frac1n\log\left|\frac{s_i+\alpha}m-s\right|-\int_{r_{i-1}}^{ r_i}\log\left|r-s\right|d\nu(r)\right)\\&\quad-\frac{B_1}{\sqrt{n}}-B_{1/\sqrt{n}}-B_t\\
    &\geq -\frac{B_1}{\sqrt{n}}-B_{1/\sqrt{n}}-B_t+\frac5n\left(\log r-\log m\right)+\log\left(1-\sqrt{t}\right)\\
    &-\frac{2(L-B-\frac1n\log m)\log (n+\tfrac32)}{\log t}.
  \end{align*}
\end{proof}

\begin{lemma}\label{lem:apr-mea-sup}
  Assume that $U_{\mu}$ is bounded below.
  Take $m=\lfloor n^{(\log n)^2}\rfloor$ where $\lfloor a\rfloor$ is the greatest integer smaller than or equal to $a$.
  Fix $0<\alpha_1<\alpha_2<1/2$. Then the following properties hold.
  First, the sequence of measures $\mu_{m,n,\alpha_1,\alpha_2}$ converges to $\mu$ weakly.
  Moreover, for each $\varepsilon>0$, there exists $n_0$ such that for each $n>n_0$, the potential $U_{\mu_{m,n,\alpha_1,\alpha_2}}\geq U_{\mu}-\varepsilon$.
\end{lemma}

\begin{proof}
  Set $t=(n+\tfrac12)/m$. Since $\nu_{m,n,\alpha_1,\alpha_2}([a, b])\leq \nu([a-t, b]) +1/n$, and since $\nu([a, b])\leq \nu_{m,n,\alpha_1,\alpha_2}([a, b+t]) +1/n$, the sequence of measures $\nu_{m,n,\alpha_1,\alpha_2}$ converges weakly to $\nu$.
  Taking the push-forward by $h$, we get that the sequence of measures $\mu_{m,n,\alpha_1,\alpha_2}$ converges to $\mu$ weakly.

  For each $0\leq\beta<1$, and each $0<\alpha<1/2$, let $f(\beta, \alpha)=\min\{\left|\beta-\alpha\right|, \left|\beta-\alpha-1\right|\}$.
  Let $m'<m$ be a nonnegative integer.
  By Lemma \ref{lem:apr-dir-int},
  $$
    U_{\delta_{m, n, \alpha}}(h((m'+\beta)/m))\geq U_{\mu}(h((m'+\beta)/m))-A_n+\frac5n\log f(\beta, \alpha)
  $$
  where $A_n$ tends to $0$ as $n$ tends to $\infty$.
  Therefore, we have
  \begin{align*}
    U_{\mu_{m,n,\alpha_1,\alpha_2}}(h((m'+\beta)/m))&=\frac1{\alpha_2-\alpha_1}\int_{\alpha_1}^{\alpha_2} U_{\delta_{m, n, \alpha}}(h((m'+\beta)/m))\,\mathrm{d}\alpha\\
    &\geq U_{\mu}(h((m'+\beta)/m))-A_n+ \frac1{\alpha_2-\alpha_1}\int_{\alpha_1}^{\alpha_2}\frac5n\log f(\beta, \alpha)\,\mathrm{d}\alpha\\
    &\geq U_{\mu}(h((m'+\beta)/m))-A_n+\frac{10(\log \left(\alpha_2-\alpha_1\right)-\log 2-1)}n.
  \end{align*}
  As $n$ tends to $\infty$, the constant tends to $0$.
  This shows that for $n$ large enough, $U_{\mu_{m,n,\alpha_1,\alpha_2}}\geq U_{\mu}-\varepsilon$ on $H$.
  By the maximum principle, for such $n$, $U_{\mu_{m,n,\alpha_1,\alpha_2}}\geq U_{\mu}-\varepsilon$ on the whole space.
\end{proof}
Now back to the general case.
Let $E$ be a simple set, with a point $x$ outside $E$.
Let $G$ be the union of all connected components in $E$ which contain an interior point, and $H$ the union of all other components.

Let $\mu_1, \mu_2, \ldots, \mu_m$ be probability measures supported on the ${\mathbb {R}}$-interior of $E$, with continuous potentials.
Let ${X}=\{x_1, x_2, \ldots, x_n\}$ be a finite set of points on $C^{\mathrm{an}}$, disjoint from $E$ and $x$.
Let $D_1, D_2, \ldots, D_m$ be effective ${\mathbb {Q}}$-divisors of degree $1$ supported on ${X}$.
Let $g_j$ be the Green function of $D_j$ such that $c_1(g_j)=\mu_j$.
Note that $g_j$ is continuous on $E$.
We will approximate $g_j$ by the logarithms of rational functions satisfying a suitable boundary dominance condition.
\begin{proposition}\label{prop:arc-apr-fun}
  For any $\varepsilon> 0$, there exists $d_0>0$ such that for all positive integers $d_1, d_2, \ldots, d_m$ which are divisible by $d_0$, there exist rational functions $f_1, f_2, \ldots, f_m$ with the following properties.
  \begin{enumerate}
    \item The pole divisor of $f_j$ is $d_jD_j$;
    \item $(f_1, d_1g_1+\log 8), \ldots, (f_m, d_mg_m+\log 8)$ satisfy the boundary dominance condition on $E$;
      \item\label{item:rational-function-approximation} $\log\left|f_j\right|/d_j-g_j$ extends continuously to ${X}$, and its value is between $\varepsilon$ and $2\varepsilon$ on ${X}$.
  \end{enumerate}
  Moreover, the function $f_j$ does not depend on the choice of $d_{j'}$ for every $j'\neq j$.
  Lastly, for each connected component $E'$ of $E$, as $d_j$ tends to $\infty$, the proportion of zeros of $f_j$ lying in $E'$ tends to $\mu_j(E')$.
\end{proposition}
\begin{proof}
  Firstly, we may assume that each probability measure is supported on a connected component of $E$, for the following reason.
  Assume that we decompose $\mu_1$ as $p'\mu_1'+p''\mu_1''$ for probability measures $\mu_1'$ and $\mu_1''$ with continuous potentials and positive real numbers $p', p''$ with $p'+p''=1$.
  Assume we have proved the proposition for $\mu_1', \mu_1'', \mu_2, \ldots, \mu_m$.
  Then for $d_1$ a multiple of $d_0$, we take $d_1'=\lfloor p'd_1/d_0\rfloor d_0$ and $d_1''=d_1-d_1'$.
  Solve the proposition for $d_1', d_1'', d_2, \ldots, d_m$ and get $f_1', f_1'', f_2, \ldots, f_m$.
  By taking $f_1=f_1'f_1''$ and $d_1$ large enough, we satisfy all the properties.
  Therefore, we obtain the original proposition.
  Also, by Lemma \ref{lem:apr-mea-sup}, we may assume that each connected component of $H$ has at most one measure on it.
  In what follows, all functions $f_j$ will be taken such that as $d_j$ tends to infinity, the distribution of zeros of $f_j$ tends to $\mu_j$.
  Thus, by Lemma \ref{lem:uni-apr-com}, we reduce to the case when $E$ is connected and $m=1$.
  Write $\mu=\mu_1$, $g_D=g_1$ and $D=D_1$.
  First, we consider the case when $E=G$.
  Take $G'$ to be a compact set contained in the interior of $G$ such that $G'\supset {\mathrm{supp}}(\mu)$.
  For each even integer $e$, let $E'_{e}$ be a divisor supported in $G'$ of degree $e$, such that $\delta_{E'_{e}}/e$ converges to $\mu$ weakly.
  Let $e_G$ be an even integer as in Lemma \ref{lem:gen-jac-c} for the set $G$.
  Let $d$ be an even integer such that $dD$ is a ${\mathbb {Z}}$-divisor and that $e_G\leq d$.
  Set $E_{d}=E'_{d-e_G}+E''$ where $E''$ is an effective divisor as in Lemma \ref{lem:gen-jac-c} making $E_{d}$ rationally equivalent to $dD$.
  Then $\delta_{E_{d}}/d$ converges to $\mu$ weakly.
  Take $f'$ to be the rational function with ${\mathrm{div}}(f')=E_{d}-dD$, normalized such that $\log\left|f'(x)\right|=dg_D(x)$.
  Then by Lemma \ref{lem:uni-apr-com}, for each $\varepsilon>0$, there exists $d'$, such that for each $d>d'$, we have $\log\left|f'\right|-dg_D>-d\varepsilon/2$ on $\partial E$ and ${X}$.
  We are done by taking $f=e^{3d\varepsilon/2}f'$ and $d>\varepsilon^{-1}\log 8$.

  Now, we consider the case when $E=H$.
  Let $H'$ be a closed interval contained in the interior of $H$ such that ${\mathrm{supp}}(\mu)$ is contained in the interior of $H'$.
  Let $H''$ be a closed interval contained in $\operatorname{int}(H)\setminus H'$.
  For each integer $e$, set $e'=\lfloor e^{(\log e)^2}\rfloor$.
  Let $E'_{e}$ be the divisor corresponding to the measure $e\delta_{e', e, 1/3}$ as in Lemma \ref{lem:apr-dir-int}.
  Here, we take the closed interval $H'$ and the probability measure $\mu$ in the construction of the measure.
  This is defined for $e$ large enough.
  Let $n_0$ be the integer given by Lemma \ref{lem:gen-jac-r} for the interval $H''$, let $M$ be its oscillating constant, and put $e=2n_0$.
  Let $d>e$ be an even integer such that $dD$ is an even ${\mathbb Z}$-divisor; this can be arranged by increasing $d_0$. Put $N=d-e$.
  Fix a point $y$ in the open interval between $H'$ and $H''$. Since
  \[
    [dD-E'_N-ey]=[dD-dy]-[E'_N-Ny]\in J_C({\mathbb R})^0,
  \]
  there is a real divisor $A_d$ of degree $n_0$ such that $2A_d\sim dD-E'_N$. Here $dD-dy$ is twice a real divisor, and all points of $E'_N$ lie in the same real interval as $y$. Apply Lemma \ref{lem:gen-jac-r} to $A_d$ to obtain an effective divisor $E''$ of degree $e$ supported on $H''$, with $E''\sim 2A_d$ and $U_{\delta_{E''}}$ $M$-oscillating. Set $E_d=E'_N+E''$, so that $E_d\sim dD$.
  Then $\delta_{E_{d}}/d$ converges to $\mu$ weakly.
  Take $f'$ to be the rational function with ${\mathrm{div}}(f')=E_{d}-dD$, normalized such that $\log\left|f'(x)\right|=dg_D(x)$.
  Put $F_d=\log\left|f'\right|-dg_D$. The normalization at $x$ gives
  \[
    F_d=U_{\delta_{E'_N}}+U_{\delta_{E''}}-dU_\mu.
  \]
  By Lemma \ref{lem:uni-apr-com}, $F_d/d$ tends uniformly to zero on $X\cup\partial H\cup\{y\}$. Thus $|F_d|<d\varepsilon/2$ there for all sufficiently large $d$, with the values on $X$ understood by continuous extension.
  On $H'$, the potentials $U_{\delta_{E''}}$ are uniformly bounded, since $\deg(E'')=e$ and $H'$ and $H''$ are disjoint. Applying Lemma \ref{lem:apr-dir-int} with $r=1/3$ gives $F_d\geq-d\varepsilon/2$ at points separating the consecutive zeros of $E'_N$, for all sufficiently large $d$.

  On $H''$, Lemma \ref{lem:uni-apr-com} gives $(U_{\delta_{E'_N}}-dU_\mu)/d\to0$ uniformly. Since $U_{\delta_{E''}}$ is $M$-oscillating, its zeros are separated by points where $F_d\geq M-o(d)\geq-d\varepsilon/2$ for all sufficiently large $d$.

  Together with the estimates at $y$ and the endpoints of $H$, these points separate all the zeros of $f'$ in $H$. All these zeros are simple. Hence $f'$ is $dg_D-d\varepsilon/2$-oscillating on $H$.
  We are done by taking $f=e^{3d\varepsilon/2}f'$ and $d>\varepsilon^{-1}\log 8$.
\end{proof}

If the properties in Proposition \ref{prop:arc-apr-fun} hold for $\{f_j\}$ and $\{g_j\}$, we say that they are \emph{$-\varepsilon$-oscillating compatibly}.
Let $f$ be a rational function on $C$.
If ${\mathrm{ord}}_y(f)=0$, we say that $f$ is positive at $y$ if there exists a rational function $f'$ such that $f(y)=f'(y)^2$.
If ${\mathrm{ord}}_y(f)$ is an even number, we say that $f$ is positive at $y$ if there exists a non-zero rational function $f'$ with ${\mathrm{ord}}_y(f)+2{\mathrm{ord}}_y(f')=0$ such that $ff'^2$ is positive at $y$.
In Proposition \ref{prop:loc-sec}, we want the initial element to be represented by an $n$-th power.
For the case $K={\mathbb {C}}$, this condition is trivial as long as the order at each point in ${X}$ is a multiple of $n$ greater than $2gn$.
For the case $K={\mathbb {R}}$, we additionally require that the function representing the element is positive at each ${\mathbb {R}}$-point of ${X}$, which is ensured by the following corollary.
\begin{corollary}\label{cor:arc-pos-r}
  In the previous proposition, the rational functions $f_1, f_2, \ldots, f_m$ can be assumed to be positive at all ${\mathbb {R}}$-points in ${X}$.
\end{corollary}
\begin{proof}
  Let $n_1$ be the number of ${\mathbb {R}}$-points in ${X}$.
  Consider the measures $\{\mu_{i,j}\}_{1\leq i\leq m, 1\leq j\leq 2^{n_1}+1}$ where $\mu_{i,j}=\mu_i$.
  Apply the previous proposition to these measures, and we get rational functions $f_{i,j}$.
  By the pigeonhole principle, for each $i$, there exist $j_1, j_2$ such that $f_{i,j_1}f_{i,j_2}$ is positive at each ${\mathbb {R}}$-point in ${X}$.
  Thus, the corollary follows by multiplying $d_0$ by $2$.
\end{proof}
\subsubsection{Construction of a patching map}
Let $T_i$ be the $i$th Chebyshev polynomial of the first kind.
\begin{lemma}\label{lem:che-ind}
  For every complex number $z$, if $\left|z\right|\geq 1$, $\left|zT_i(z)\right|\leq \left|T_{i+1}(z)\right|\leq \left|4zT_i(z)\right|$.
\end{lemma}
\begin{proof}
  The sequence of polynomials $T_i$ follows the inductive relation $T_{i+1}(z)=2zT_{i}(z)-T_{i-1}(z)$ for all $i\geq 1$.
  Here, $T_0(z)$ is viewed as $1$.
  We prove by induction that $\left|zT_i(z)\right|\leq \left|T_{i+1}(z)\right|$.
  For $i=1$, $\left|T_1(z)\right|=\left|z\right|$.

  Assume that $\left|T_i(z)\right|\geq \left|zT_{i-1}(z)\right|$.
  Then $\left|T_{i+1}(z)\right|\geq 2\left|zT_{i}(z)\right|- \left|T_{i-1}(z)\right|\geq \left|zT_i(z)\right|$, and
  $\left|T_{i+1}(z)\right|\leq 2\left|zT_i(z)\right|+ \left|T_{i-1}(z)\right|<4\left|zT_{i}(z)\right|$.
\end{proof}
\begin{lemma}\label{lem:che-low}
  For every complex number $z$, and all positive integers $j>i$, $\left|2T_i(z)\right| \leq \max\{4, 2\left|2T_j(z)\right|^{i/j}\}$.
\end{lemma}
\begin{proof}
  Assume that $2z=y+1/y$, with $\left|y\right|\geq 1$.
  Then $2T_i(z)=y^i+y^{-i}$, and $\left|2T_i(z)\right|\leq \left|y\right|^i+\left|y\right|^{-i}$.
  If $\left|2T_i(z)\right|\leq 4$, the inequality is obvious.
  Otherwise, $\left|y\right|^i> 3$.
  Therefore, $\left|y\right|^i+\left|y\right|^{-i}\leq 10\left|y\right|^i/9$.
  Since $j>i$, $\left|y\right|^j>3$, and we have $\left|2T_j(z)\right|\geq \left|y\right|^j-\left|y\right|^{-j}\geq 8\left|y\right|^j/9$.
  Thus,
  $$
    2\left|2T_j(z)\right|^{i/j}\geq 2\left|y\right|^i \left(\frac89\right)^{i/j}\geq \frac{10}9 \left|y\right|^i\geq \left|2T_i(z)\right|.
  $$
\end{proof}

For each positive integer $i$ and positive real number $r$, set $S_{i, r}(z)=2^{1-i}r^iT_i(r^{-1}z)$.
Then $S_{i, r}$ is monic, and $S_{ij,r}=S_{j, 2^{1-i}r^{i}}\circ S_{i, r}$.
By Lemma \ref{lem:che-ind}, for each $\left|z\right|\geq r$, we have
$$
  \left|zS_{i,r}(z)\right|/2\leq \left|S_{i+1,r}(z)\right|\leq 2\left|zS_{i,r}(z)\right|.
$$
By Lemma \ref{lem:che-low}, for each $j>i$, it is not hard to see that
$$
  \left|S_{i,r}(z)\right|\leq \max\{2^{2-i}r^i, 2\left|S_{j,r}(z)\right|^{i/j}\}.
$$

\begin{lemma}
  Assume that $(f_1, g_1), \ldots, (f_n, g_n)$ satisfy the boundary dominance condition on $E$.
  Then we have the following.
  \begin{enumerate}
    \item Let $f'_n$ be a rational function whose pole divisor is smaller than that of $f_n$ and such that $\log\left|f'_n\right|\leq \max\{\log\left|f_n\right|, g_n\}-\log 2$.
      Then $(f_1, g_1), \ldots, (f_{n-1}, g_{n-1}), (f_n+f'_n, g_n-\log 2)$ satisfy the boundary dominance condition. \label{item:boundary-dominance-perturbation}
    \item If $g_n=\log M$ is a constant function, then for each positive integer $k$, we have that $(f_1, g_1), \ldots,$ $(f_{n-1}, g_{n-1}), (S_{k, M}(f_n), (1-k)\log 2+k\log M)$ satisfy the boundary dominance condition. \label{item:boundary-dominance-chebyshev}
  \end{enumerate}
\end{lemma}

\begin{proof}
  Let $G$ be a connected component of $E$ which contains an interior point.
  For (\ref{item:boundary-dominance-perturbation}), by Rouch\'e's theorem, the numbers of zeros of $f_n$ and $f_n+f'_n$ in $G$ are exactly the same.
  Also, $\log\left|f_n+f'_n\right|\geq \log\left(\left|f_n\right|-\left|f'_n\right|\right)\geq \log\left(\left|f_n\right|\right)-\log 2$ whenever $\log\left|f_n\right|\geq g_n$.
  For (\ref{item:boundary-dominance-chebyshev}), the number of zeros of $S_{k, M}(f_n)$ in $G$ is $k$ times that of $f_n$.
  By Lemma \ref{lem:che-ind}, for each $\left|z\right|\geq M$, $\left|S_{k,M}(z)\right|\geq M(\left|z\right|/2)^{k-1}\geq M^k2^{1-k}$.
  Hence, $\log\left|S_{k, M}(f_n)\right|\geq (1-k)\log 2+k\log M$ whenever $\left|f_n\right|\geq M$.

  Let $H$ be a connected component of $\{z\in C({\mathbb {R}})|\log\left|f_n(z)\right|\leq g_n(z)\}$ which contains a zero of $f_n$.
  Then for the endpoints $z_1$ and $z_2$ of $H$, we have $\log\left|f_n(z_i)\right|=g_n(z_i)$ for $i=1,2$, and $f_n(z_1)f_n(z_2)<0$.
  For (\ref{item:boundary-dominance-perturbation}), we have $\log\left|f'_n(z_i)\right|\leq g_n(z_i)-\log 2$ for $i=1,2$.
  Therefore, $\log\left|(f_n+f'_n)(z_i)\right|\geq g_n(z_i)-\log 2$ for $i=1,2$ and $(f_n+f'_n)(z_1)\cdot(f_n+f'_n)(z_2)<0$.
  Hence, there is a zero of $f_n+f'_n$ in $H$.
  For (\ref{item:boundary-dominance-chebyshev}), since $S_{k, M}$ is $(1-k)\log 2+k\log M$-oscillating on the interval $[-M, M]$, there are at least $k$ connected components of $\{z\in H|\log\left|S_{k, M}(f_n)(z)\right|<(1-k)\log 2+k\log M\}$ which contain a root of $S_{k, M}(f_n)$.

  Note that the number of zeros of $f_n+f'_n$ is at most that of $f_n$.
  Therefore, the number of zeros of $f_n+f'_n$ in $H$ is exactly $1$, and $f_n+f'_n$ is hence $g_n-\log 2$-oscillating on $H$.
  Also, the number of zeros of $S_{k, M}(f_n)$ is $k$ times that of $f_n$, and hence $S_{k, M}(f_n)$ is $(1-k)\log 2+k\log M$-oscillating on $H$.
\end{proof}

Let $E'$ be a simple set contained in the $\mathbb R$-interior of $E$.
For every $x\in{X}$, take a Green function $g_{x}$ of $x$.
Fix $\varepsilon>0$.
Let $D$ be an effective ${\mathbb {Q}}$-divisor on $C$ supported outside ${X}$ and $E$ of degree $1$.
Let $g_D$ be a subharmonic Green function of $D$ with $\mu=c_1(g_D)$ supported in $E'$.
Let $d$ be an integer.
For each $x\in{X}$, take a rational function $f_x$ on $C$ such that the pole divisor of $f_x$ is $dx$.
Set $b_x=d^{-1}\log\left|f_x\right|_{x, d}$.

For each $x$, let $\ell_x\geq 2g+2$ be an integer.
Set $l=\sum_{x\in{X}}\ell_x$.
For every $x\in{X}$, set
$$
  c_x:=\frac1{d\ell_x}\left(\sum\limits_{y\in X\setminus\{x\}}\log\left|S_{\ell_y, 4}(f_y(x))\right|\right)+b_x.
$$

\begin{proposition}\label{prop:arc-mul-fun}
  For every $B>0$, there exists a multiple $N_0$ of $d$, with the following property.
  If $\max_{x\in{X}}\left|c_x\right|$ is bounded by a given constant, then $N_0$ can be chosen regardless of the choice of $f_x$.
  Let $N$ be a multiple of $N_0$.
  Take an integer $d'$ such that $d'D$ is a ${\mathbb {Z}}$-divisor and take a rational function $f$ with pole divisor $d'D$.
  If $\{(f_x, \log 4)\}_{x\in {X}}$ and $(f, d'g_D)$ satisfy the boundary dominance condition on $E$, then we have the following.

  Set
  $$
    \phi_N^{\mathrm{init}}=f\left(\prod\limits_{x\in {X}}S_{\ell_x, 4}(f_x)\right)^{N/d}.
  $$
  For each band filtration $\{D_N^{(j)}\}$ for $(\ell_x)$ and $N$, there exists a patching map $\psi_N$ with the following property.
  For each patching vector $\boldsymbol{\phi}_N=(\phi_N^{(j)})$ with initial element $\phi_N^{\mathrm{init}}$ bounded by $B$ and $(c_x)$ and satisfying $\phi_N^{(0)}=\phi_N^{\mathrm{init}}$, there is a representative $G_N$ of $\psi_N(\boldsymbol{\phi}_N)$ with pole divisor $d'D+D_N^{(0)}$, such that the number of zeros of $G_N$ in each connected component of $E$ is at least that of $f$, and that $(G_N, d'g_D)$ satisfies the boundary dominance condition on $E$.

\end{proposition}
\begin{proof}

  For $x\in {X}, j>2g\deg(x)$, let $h_{x,j}$ be as follows.
  If $x$ is a $K$-point, let $h_{x,j}$ be a rational function whose pole divisor is $jx$, normalized such that $\left|h_{x,j}\right|_{x, j}=1$.
  If $x$ is not a $K$-point, which means $K={\mathbb {R}}$ and $x$ is a point of degree $2$, let $h_{x,2j-1}$, $h_{x,2j}$ be rational functions whose pole divisor is $jx$, normalized such that $\left|h_{x,2j-1}\right|_{x, j}= \left|h_{x,2j}\right|_{x, j}=1$.
  Also, we require $(h_{x,2j-1}^2h_{x,2j}^{-2})(x)=-1$.
  By the Riemann--Roch theorem, for $j> 2g\deg(x)$, $h_{x,j}$ can be defined.
  Set
  $$
    M_k=\sup_{x\in {X}, 2g\deg(x)<j\leq(kd\ell_x+2g)\deg(x)}\sup_{z\in E}\left|h_{x,j}(z)\right|.
  $$

  In the entire proof, all boundary dominance conditions are assumed to be compatible with $(f, d'g_D)$.
  Set
  $$
    Q(z):=\prod\limits_{x\in {X}}S_{\ell_x, 4}(f_x(z)).
  $$
  Then we have $\left|Q(z)\right|_{x, d\ell_x}=e^{d\ell_xc_x}$.
  For every positive integer $k_1$, set
  $$
    Q_{k_1}(z)=S_{k_1, 4}(Q(z)).
  $$
  So $(Q_{k_1}, (1+k_1)\log 2)$ satisfies the boundary dominance condition.

  Let $Q'$ be a rational function on $C$, with $\left|Q'(z)\right|\leq 2^{k_1}$ on $E$ and pole divisor at most $\sum_{x\in X}(k_1d\ell_x-1)x$. Thus $Q_{k_1}+Q'$ has the same leading term as $Q_{k_1}$ at every $x\in X$.
  For every positive integer $k_2$, set
  $$
    Q_{k_1,k_2}=S_{k_2, 2^{k_1}}(Q_{k_1}+Q').
  $$
  Then $(Q_{k_1, k_2},(1-k_2+k_1k_2)\log 2)$ satisfies the boundary dominance condition.
  Also, we obtain $\left|Q_{k_1}(z)\right|\leq \max\{2^{2+k_1}, 2\left|Q_{k_1, 1}(z)\right|\}$.

  By Lemma \ref{lem:che-low}, for each $k'_2<k_2$, we have
  $$
    \left|Q_{k_1, k'_2}(z)\right|\leq \max\left\{2^{2+k_1k'_2-k'_2}, 2\left|Q_{k_1, k_2}(z)\right|^{k'_2/k_2}\right\}.
  $$
  For the same reason, for each $k'_1<k_1$, we have
  \begin{align*}
    \left|Q_{k'_1}(z)\right|&\leq\max\left\{2^{2+k'_1}, 2\left|Q_{k_1}(z)\right|^{k'_1/k_1}\right\}\\
    &\leq \max\left\{2^{2+k'_1}, 2\max\{2^{2+k_1}, 2\left|Q_{k_1, 1}(z)\right|\}^{k'_1/k_1}\right\}\\
    &\leq \max\left\{2^{4+k'_1}, 8\left|Q_{k_1, k_2}(z)\right|^{k'_1/(k_1k_2)}\right\}.
  \end{align*}
  Take $z'$ such that $\left|Q_{k_1, k_2}(z')\right|\geq 2^{1-k_2+k_1k_2}$ holds.
  Then we have $2^{2-k'_2+k_1k'_2}\leq 4\left|Q_{k_1, k_2}(z')\right|^{k'_2/k_2}$ and $2^{4+k'_1}\leq 32\left|Q_{k_1, k_2}(z')\right|^{k'_1/(k_1k_2)}$.
  Thus, we obtain
  $$
    \left|Q_{k_1, k'_2}(z')Q_{k'_1}(z')\right|\leq 128\left|Q_{k_1, k_2}(z')\right|^{(k'_1+k_1k_2')/(k_1k_2)}.
  $$

  For each $x\in {X}, 2g\deg(x)<j\leq d\ell_xk_1k_2\deg(x)$, set
  $$
    R_{x,j}=e^{j'' c_x}Q_{k_1, k_2'}Q_{k_1'}h_{x,j'}
  $$
  where $j', k_1', k_2'$ are uniquely chosen so that $j'+(k_1'+k_1k_2')d\ell_x\deg(x)=j$, $0\leq k_1'<k_1$, $2g\deg(x)<j'\leq \ell_xd\deg(x)+2g\deg(x)$ and $j''=\lfloor (j'-1)/\deg(x)\rfloor+1$.
  Note that we have $\left|R_{x,j}\right|_{x, \lfloor (j-1)/\deg(x)\rfloor+1}=e^{(\lfloor (j-1)/\deg(x)\rfloor+1)c_x}$.
  Set $M'=128e^{(d\max_{x\in X}\ell_x+2g)\max_{x\in {X}}\left|c_x\right|}M_1$.
  For each $z'$ such that $\left|Q_{k_1, k_2}(z')\right|\geq 2^{1-k_2+k_1k_2}$ holds, for $k_1\geq 2$, we have
  \begin{align*}
    \left|R_{x,j}(z')\right|&=e^{j''c_x}\left|h_{x,j'}(z')\right|\left|Q_{k_1, k'_2}(z')Q_{k'_1}(z') \right|\\
    &\leq 128e^{j''c_x}M_1\left|Q_{k_1, k_2}(z')\right|^{(j-j')/(kd\ell_x\deg(x))}\\
    &\leq M'\sqrt{2}^{(j/\deg(x)-kd\ell_x)/(d\ell_x)}\left|Q_{k_1, k_2}(z')\right|.
  \end{align*}

  Now, we begin to construct the functions recursively.
  Set $k=k_1k_2$ and $N=kd$.
  For each patching filtration $D_N^{(i)}$, set $P_{i}=\{(x, j)|x\in {X}, {\mathrm{ord}}_x(D_N^{(i+1)})\deg(x)<j\leq {\mathrm{ord}}_x(D_N^{(i)})\deg(x)\}$.
  Since ${\mathrm{div}}(f)$ is disjoint from ${X}$, $f^{-1}\phi_N^{(i)}$ can be viewed as an element in $H^0(D_N^{(i)})/H^0(D_N^{(i+1)})$ by the Riemann--Roch theorem.
  Let $\varphi_N^{(i)}$ be a representative of $f^{-1}\phi_N^{(i)}$ for each $i$.
  Let $i_N$ be the integer such that $D_N^{(i_N)}=\sum_{x\in {X}}(2g+(k_1k_2-k_1)d\ell_x)x$.
  For each $1\leq i<i_N$, and $(x, j)\in P_i$, if $\deg(x)=1$, set
  $$
    b_{x,j}=\left(\varphi_N^{(i)}Q^{k_1-k}h_{x,j+(k_1-k)d\ell_x}^{-1}\right)(x)/k_2.
  $$
  Then $\left|b_{x,j}\right|\leq Be^{\left|c_x\right|k_1d\ell_x}/k_2$.
  If $\deg(x)=2$, we assume that $j$ is even, and let $b_{x,j}$ and $b_{x,j-1}$ be the unique real numbers such that
  $$
    k_2\left(\frac{Q^{k-k_1}(b_{x,j}h_{x,j+2(k_1-k)d\ell_x}+b_{x,j-1}h_{x,j+2(k_1-k)d\ell_x-1})}{\varphi_N^{(i)}}\right)(x)=1.
  $$
  Since $(Q^{k-k_1}h_{x,j+2(k_1-k)d\ell_x}/\varphi_N^{(i)})(x)$ and $(Q^{k-k_1}h_{x,j+2(k_1-k)d\ell_x-1}/\varphi_N^{(i)})(x)$ form an orthogonal basis of $K(x)={\mathbb {C}}$ whose elements have norm at least $e^{-\left|c_x\right|k_1d\ell_x}/B$, we have $\left|b_{x,j}\right|\leq Be^{\left|c_x\right|k_1d\ell_x}/k_2$ and the same for $b_{x,j-1}$.

  Set
  $$
    G_N^{(i)}=S_{k_2, 2^{k_1}}(Q_{k_1}+\sum\limits_{i'=1}^{i}\sum\limits_{(x, j)\in P_{i'}} b_{x,j}h_{x,j+(k_1-k)d\ell_x\deg(x)}).
  $$
  For $i=0$, the sum is empty and $fG_N^{(0)}$ represents $\phi_N^{\mathrm{init}}$; the subsequent corrections preserve this leading term.
  Since $S_{k,r}$ is a monic polynomial, we have
  $$
    G_N^{(i)}-G_N^{(i-1)}\in \sum\limits_{(x, j)\in P_i} k_2Q^{k-k_1}b_{x,j}h_{x,j+(k_1-k)d\ell_x\deg(x)}+H^0(D_N^{(i+1)})=\varphi_N^{(i)}+H^0(D_N^{(i+1)}).
  $$

  For each $i\geq i_N$ and  $(x, j)\in P_i$, if $\deg(x)=1$, set
  $$
    b_{x,j}=\left(\varphi_N^{(i)}R_{x,j}^{-1}\right)(x).
  $$
  Then we have $\left|b_{x,j}\right|\leq B$.
  If $\deg(x)=2$, we assume that $j$ is even, and let $b_{x,j}$ and $b_{x,j-1}$ be the unique real numbers such that
  $$
    \left(\frac{b_{x,j}R_{x,j}+b_{x,j-1}R_{x,j-1}}{\varphi_N^{(i)}}\right)(x)=1.
  $$
  Then since $(R_{x,j}/\varphi_N^{(i)})(x)$ and $(R_{x,j-1}/\varphi_N^{(i)})(x)$ form an orthogonal basis of $K(x)={\mathbb {C}}$ whose elements have norm at least $1/B$, we have $\left|b_{x,j}\right|\leq B$ and the same for $b_{x,j-1}$.

  Set
  $$
    G_N^{(i)}=G_N^{(i_N-1)}+\sum\limits_{i'=i_N}^{i}\sum\limits_{(x, j)\in P_{i'}}b_{x,j}R_{x,j}.
  $$
  Then
  $$
    G_N^{(i)}-G_N^{(i-1)}=\sum\limits_{(x, j)\in P_i}b_{x,j}R_{x,j}\in \varphi_N^{(i)}+H^0(D_N^{(i+1)}).
  $$
  Recall that $D_N^{(j_N)}$ is the last term of the patching filtration, and set $G_N=fG_N^{(j_N-1)}$.
  Define $\psi_N$ to map $\boldsymbol{\phi}_N=(\phi_N^{(i)})$ to the class of $G_N$.
  This gives a patching map.

  Finally, we take $k_1$ and $k_2$ large enough to ensure the boundary dominance condition.
  Take $M''$ to be a constant greater than $e^{\left|c_x\right|d\ell_x}$ for each $x\in {X}$.
  Set
  $$
    Q'(z)= \sum\limits_{i=1}^{i_N-1}\sum\limits_{(x, j)\in P_i} b_{x,j}h_{x,j+(k_1-k)d\ell_x\deg(x)}(z)
  $$
  For $(x,j)\in P_i$ with $1\leq i<i_N$, we have $j\leq(kd\ell_x-1)\deg(x)$. Hence
  \[
    \operatorname{div}_{\infty}(Q')\leq\sum_{x\in X}(k_1d\ell_x-1)x.
  \]
  The same bound holds for the partial sums defining $G_N^{(i)}$ for $1\leq i<i_N$.
  For each $z\in E$, we have
  $$
    \left|Q'(z)\right|\leq {M''}^{k_1}dlk_1BM'M_{k_1}/k_2.
  $$
  By taking $k_2>{M''}^{k_1}dlk_1BM'M_{k_1}2^{k_1}$, it is easy to see that $\left|Q'\right|\leq 2^{k_1}$ on $E$.
  Recall that for each $z'\in E$ with $\left|Q_{k_1, k_2}(z')\right|\geq 2^{1-k_2+k_1k_2}$, we have
  $$
    \left|R_{x,j}(z')\right|\leq M'\left|Q_{k_1,k_2}(z')\right|\sqrt{2}^{(j/\deg(x)-kd\ell_x)/(d\ell_x)}.
  $$
  We obtain
  \begin{align*}
    \left|\sum\limits_{i=i_N}^{j_N-1}\sum\limits_{(x,j)\in P_i}b_{x,j}R_{x,j}(z')\right|&\leq \sum\limits_{i=i_N}^{j_N-1}\sum\limits_{(x,j)\in P_{i}}  BM'\sqrt{2}^{(j/\deg(x)-kd\ell_x)/(d\ell_x)}\left|Q_{k_1,k_2}(z')\right|\\
    &=\sum\limits_{x\in{X}}\sum\limits_{j=2g\deg(x)+1} ^{((k-k_1)d\ell_x+2g)\deg(x)}BM'\sqrt{2}^{(j/\deg(x)-kd\ell_x)/(d\ell_x)}\left|Q_{k_1,k_2}(z')\right|\\
    &\leq \sum\limits_{x\in {X}} BM' \sqrt{2}^{2g/(d\ell_x)-k_1} (1-2^{-1/(4d\ell_x)})^{-1} \left|Q_{k_1,k_2}(z')\right|.
  \end{align*}
  By taking $k_1$ large enough, we can make
  $$
    2\# {X} BM' \sqrt{2}^{2g/(d\ell_x)-k_1} (1-2^{-1/(4d\ell_x)})^{-1}\leq 1
  $$
  for each $x\in {X}$.
  Then we have that $(G_N^{(j_N-1)}, 0)$ satisfies the boundary dominance condition.
  Recall that all boundary dominance conditions in the proof are compatible with $(f, d'g_D)$.
  Therefore, $(G_N, d'g_D)$ satisfies the boundary dominance condition.
  Note that both $k_1$ and $k_2$ do not depend on the choice of $f_x$ as long as $\max_{x\in {X}}\left|c_x\right|$ is bounded.
\end{proof}

\subsubsection{Initial patching elements}
Now, we use Proposition \ref{prop:arc-apr-fun} and Proposition \ref{prop:arc-mul-fun} to construct the local section.
\begin{lemma}\label{lem:pow-c}
  Let $z\in{\mathbb {C}}$ be a non-zero number.
  Let $U$ be an open set in ${\mathbb {C}}$ containing $z$.
  There exist a real number $\varepsilon>0$ and an integer $n'$ such that for each $n>n'$, and each $y\in{\mathbb {C}}$ with $\left|\log \left|y\right|-n\log \left|z\right|\right|<n\varepsilon$, there exists $z'\in U$ such that $y={z'}^n$.
\end{lemma}
\begin{proof}
  Take $w=a+bi\in {\mathbb {C}}$ such that $e^{w}=z$.
  Take $V$ to be an open neighborhood of $w$ such that $e^V\subset U$.
  Then $V$ contains a rectangle $\{w'=a'+b'i\in \mathbb C|a-\varepsilon_1\leq a'\leq a+\varepsilon_1, b-\varepsilon_1\leq b'\leq b+\varepsilon_1\}$.
  Take $\varepsilon=\varepsilon_1$ and $n'>\pi/\varepsilon_1$.
  Let $u=c+di\in {\mathbb {C}}$ such that $e^u=y$.
  Then let $d'$ be an integer such that $\left|2\pi d'-d+nb\right|\leq\pi$.
  Let $u'=c/n+(d-2\pi d')i/n$.
  Then $u'$ is contained in the rectangle, and hence in $V$.
  We are done by setting $z'=e^{u'}$.
\end{proof}

\begin{lemma}\label{lem:pow-r}
  Let $z\in{\mathbb {R}}$ be a positive real number.
  Let $U$ be an open interval in ${\mathbb {R}}_{+}$ containing $z$.
  There exists a real number $\varepsilon>0$ such that for each positive integer $n$, and each $y\in{\mathbb {R}}_{+}$ with $\left|\log \left|y\right|-n\log \left|z\right|\right|<n\varepsilon$, there exists $z'\in U$ such that $y={z'}^n$.
\end{lemma}
\begin{proof}
  This is obvious by taking $\varepsilon$ such that $U$ contains the open interval $(e^{-\varepsilon}z, e^{\varepsilon}z)$ and setting $z'=\sqrt[n]{y}$.
\end{proof}
\begin{lemma}\label{lem:che-asy}
  Let $r$ be a positive real number.
  Let $y$ be a complex number with $\left|y\right|>1$, and set $z=r(y+1/y)$.
  Then $\lim_{n\to\infty}\left|S_{n, 2r}(z)\right|^{1/n}=r\left|y\right|$.
\end{lemma}
\begin{proof}
  Note that $S_{n, 2r}(z)=r^n(y^n+y^{-n})$.
  Set $s=\left|y\right|>1$.
  Therefore, $r^n(s^n-1)\leq\left|S_{n, 2r}(z)\right|\leq r^n(s^n+1)$.
  The lemma follows since $\lim_{n\to\infty}(s^n-1)^{1/n}=\lim_{n\to\infty}(s^n+1)^{1/n}=s$.
\end{proof}

\begin{proof}[Proof of Proposition \ref{prop:loc-sec}, archimedean case]
  By Corollary \ref{cor:gre-com-exh}, replacing $g_{E, x}$ by $g_{E', x}$ where $E'$ is a simple set contained in the ${\mathbb {R}}$-interior of $E$ can change the constants $c_{x, y}$ by an arbitrarily small amount, and thus does not change the proposition.
  Then $c_1(g_{E', x})=\mu_{E', x}$ is the equilibrium measure, and is supported in the ${\mathbb {R}}$-interior of $E$.

  Take $0<\varepsilon_1<\varepsilon/6$.
  Note that $g_{E', x}\geq 0$.
  By Proposition \ref{prop:arc-apr-fun} and Corollary \ref{cor:arc-pos-r}, we obtain an integer $d_0$ such that for $\{d_x\}_{x\in {X}}$ and $d'$ multiples of $d_0$, there are rational functions $\{f_x\}_{x\in {X}}$ and $f$, such that $(f_x, \log 8)$ and $(f, d'g_D+\log 8)$ satisfy the boundary dominance condition on $E$ and that $\log\left|f_x\right|/d_x-g_{E', x}$ extends continuously to ${X}$, and its value is between $\varepsilon_1$ and $2\varepsilon_1$ on ${X}$.
  Also, $f$ is positive at all ${\mathbb {R}}$-points in ${X}$.

  Let $d>2g+2\# {X}$ be a multiple of $d_0$, and set $d_x=d\deg(x)$ for each $x$.
  For each $x\neq y$, set $c'_{x,y}=(d\deg(y))^{-1}\log\left|2z\right|$ where $z$ is the root of $z+z^{-1}=f_{y}(x)/2$ with $\left|z\right|\geq 1$.
  Since $\left|f_y(x)\right|\geq e^{d_y(g_{E',y}(x)+\varepsilon_1)}\geq e^{d\varepsilon_1}$,
  we have
  $$
    \left|z-f_{y}(x)/2\right|\leq \left|z^{-1}\right|\leq 1.
  $$
  Since $e^{d_y(g_{E', y}(x)+\varepsilon_1)}\leq \left|f_y(x)\right|\leq e^{d_y(g_{E', y}(x)+2\varepsilon_1)}$, we obtain
  $$
    e^{d_y(g_{E', y}(x)+\varepsilon_1)}/2-1\leq \left|z\right| \leq e^{d_y(g_{E', y}(x)+2\varepsilon_1)}+1.
  $$
  Therefore, as we take $d$ large enough,  $\left|c_{x,y}-c'_{x,y}\right|<3\varepsilon_1$ for each $x\neq y$.
  Set $c''_{x,x}=(d\deg(x))^{-1}\log\left|f_x\right|_{x, d}$.

  For each $x$, let $\{h_{x,j}\}_{j\leq \deg(x)}$ be a set of generators of
  $$
    H^0\left(dx-\sum_{y\neq x}y\right)\Big/H^0\left(dx-\sum_{y\in{X}}y\right).
  $$
  By the Riemann--Roch theorem, the above is a $K$-vector space of dimension $\deg(x)$.
  Note that $h_{x,j}$ is a rational function with a unique pole $x$ of order $d$ and having a zero at each $y\in X\setminus\{x\}$.
  Take numbers $e_{x,j}\in K$ for each $j\leq \deg(x)$.
  Then there exists $M>0$, such that if we have $\left|e_{x,j}\right|<M$, then $\left|\sum_{j=1}^{\deg(x)}e_{x,j}h_{x,j}\right|<4$ on $E$.
  In what follows, we always take $\left|e_{x,j}\right|<M$.
  Set $P_{x}=f_x+\sum_{j=1}^{\deg(x)}e_{x,j}h_{x,j}$.
  The above construction enables us to take $(P_xh_{x,1}^{-1})(x)$ as any element in an open neighborhood of $(f_xh_{x,1}^{-1})(x)$, without changing $P_x(y)$ for each $y\neq x$.
  By Lemma \ref{lem:pow-c} and Lemma \ref{lem:pow-r}, there is an open interval $I_x=(c_{x, min}, c_{x, max})$ in ${\mathbb {R}}$ containing $c''_{x,x}$ and an integer $l_1>0$ with the following property.
  For each $l>l_1$ and $\phi_x\in H^0(dlx)/H^0((dl-1)x)$ with $(dl\deg(x))^{-1}\log\left|\phi_x\right|_{x, dl}\in I_x$, if $\phi_x$ is positive at each ${\mathbb {R}}$-point of ${X}$, then we can choose $\left|e_{x,j}\right|\leq M$ such that $P_x^l$ represents $\phi_x$.
  We may assume that $c_{x, max}, c_{x, min}\in [c_{x, x}-\varepsilon, c_{x, x}+\varepsilon]$.
  Take $\varepsilon'<(c_{x, max}-c_{x, min})/2$ for each $x$, and set $I'_x=(c_{x, min}+\varepsilon', c_{x, max}-\varepsilon')$.

  Let $(a'_{x})_{x\in {X}}$ be a limit point of $({d'}^{-1}\log\left|f(x)\right|)_{x\in {X}}$ as we take $d'\in\{\kappa_{j}\}$.
  Take $\kappa'_j$ to be a subsequence of $\kappa_j$ such that when we take $d'=\kappa'_j$, $\left|{d'}^{-1}\log\left|f(x)\right|-a'_{x} \right|\leq s_x\varepsilon'/2$ and the number of zeros of $f$ in $E'$ is at least $(\mu(E')-\varepsilon)d'$ for each connected component $E'$ of $E$.

  Take $c'_{x, x}\in I'_x$.
  Set $c'_x=\sum_{y\in {X}}s_y\deg(y)c'_{x, y}/s_x$.
  Let $N_2$ be the smallest positive integer such that $N_2s_x$ is an integer at least $2g+2$ for each $x$.
  Let $l$ be a multiple of $N_2$ and set $\ell_x=s_xl$ for each $x\in {X}$.
  Let $d'=N=\kappa'_{j}$ be a multiple of $ld$.
  By Lemma \ref{lem:che-asy}, $\lim_{\ell_y\to\infty}(\deg(y)d\ell_y)^{-1}\log\left|S_{\ell_y, 4}(f_y(x))\right|=c'_{x, y}$.
  Thus, as we take $l$ large enough,
  \[
    \left|(\deg(y)d\ell_y)^{-1}\log\left|S_{\ell_y, 4}(f_y(x))\right|-c'_{x, y}\right|\leq \varepsilon'/(2 \# {X})
  \] for each $x\neq y$.
  For each $\phi_N^{\mathrm{init}}\in V_N^{(0)}$ that is positive at each ${\mathbb {R}}$-point and satisfies
  \[
    \log\left|\phi_N^{\mathrm{init}}\right|_{x, Ns_x}=N(a'_x+s_xc'_x)
  \]
  for each $x$, we have
  $$
    (Ns_x\deg(x))^{-1}\log\left|\phi_N^{\mathrm{init}}f^{-1}\left(\prod\limits_{y\in {X}, y\neq x}S_{\ell_y, 4}(f_y)\right)^{-N/ld}\right|_{x, Ns_x}\in I_x
  $$
  for each $x\in {X}$.

  By enlarging $l$, we may assume that $\ell_x\geq l_1$ for each $x\in {X}$.
  Then we are able to take $\left|e_{x, j}\right|$ such that
  $$
    f\left(\prod\limits_{x\in {X}} S_{\ell_x, 4}(P_x(z))\right)^{N/ld}
  $$
  represents $\phi_N^{\mathrm{init}}$.
  Note that a band filtration for $(s_x)$ and $N$ is a band filtration for $\ell_x$ and $N/l$.
  For each patching vector, the above choice of $P_x$ is made with $\phi_N^{\mathrm{init}}$ replaced by $\phi_N^{(0)}$, and is fixed for all vectors with the same $\phi_N^{(0)}$.
  The proposition follows by applying Proposition \ref{prop:arc-mul-fun} to $P_x$ and $f$.
  Note that the zeros of $G_N$ are controlled by the fact that $G_N$ is a multiple of $f$, and hence every zero of $f$ is also a zero of $G_N$ with at least the same order.

\end{proof}
\subsection{Non-archimedean case}
Let $K$ be a discretely valued non-archimedean field with finite residue field $k$.
Let $C$ be a projective, smooth and geometrically integral curve of genus $g$ over $K$.
\begin{lemma}\label{lem:rat-non-ope}
  Let $U$ be a non-empty open set in $C^{\mathrm{an}}$, $x$ a closed point in $C$.
  There is a rational function $f$ whose only pole is at $x$ and whose zeros are contained in $U$.
\end{lemma}
\begin{proof}
  If $g=0$, the lemma is trivial.
  Assume that $g\geq 1$.
  By replacing $K$ by a finite field extension, we may assume that $x$ is a $K$-point.
  Let $J_C$ be the Jacobian of $C$.
  Let $\alpha: C^g\longrightarrow J_C$ be the map $(x_1, x_2, \ldots, x_g)\longmapsto (x_1+x_2+\cdots+x_g)-gx$.
  Let $U'$ be the smooth locus of $\alpha^{\mathrm{an}}|_{U^g}$, $\mathbf{y}:=(y_1, y_2, \ldots, y_g)$ an algebraic point in $U'$.
  Then $\alpha^{\mathrm{an}}(U')$ contains an open neighborhood $V$ of $z=\alpha^{\mathrm{an}}(\mathbf{y})$.
  By replacing $K$ by a finite field extension, we may assume that $y_1, y_2, \ldots, y_g$ are all $K$-points.
  By replacing $V$ by $V\bigcap (2z-V)$, we may assume that $V=2z-V$.

  Let $V(K):=V\bigcap J_C(K)$.
  By the compactness of $J_C(K)$, there is a finite cover $\{V_i:=V+z_i\}_{1\leq i\leq n}$ where $z_i\in J_C(K)$ for each $i$.
  Consider the sequence $\{V+4iz\}_{0\leq i\leq n}$.
  By the pigeonhole principle, there exist $0\leq i_1<i_2 \leq n$ such that for some $j$, $V_j$ intersects with both $V+4i_1z$ and $V+4i_2z$.
  Let $z_1', z_2'$ be points in $V+4i_1z\bigcap V_j$ and $V+4i_2z\bigcap V_j$ respectively.
  Then the points $z_2'-4i_2z, z_j+2z-z_2', z_1'-z_j$, and $(4i_1+2)z-z_1'$ are contained in $V$.
  Note that $(z_2'-4i_2z)+ (z_j+2z-z_2')+(z_1'-z_j)+((4i_1+2)z-z_1')+4(i_2-i_1-1)z=0$.
  On the other hand, any divisor in $V(K)$ is of the form $D-gx$ where $D$ is an effective divisor supported on $U$.
  Thus, there is an effective divisor supported on $U$ linearly equivalent to $mx$ for some $m>0$.
\end{proof}

For a non-zero rational function $f$ on $C$, a closed point $z$ in $C$ and a Green function $g_{z}$ at the point $z^{\mathrm{an}}$, the \emph{leading coefficient of $f$ at $z$} is defined as $\lim_{z'\to z} \left|f(z')\right|e^{{\mathrm{ord}}_z(f)g_z}$.
For an integer $N$ and a nonnegative real number $M$, we say that \emph{the coefficient of $f$ of order $N$ at $z$ is at most $M$} if $\lim_{z'\to z} \left|f(z')\right|e^{-Ng_z}\leq M$.

Let ${\mathcal {C}}$ be a projective regular model of $C$.
Let $E$ be the compact set ${\mathrm{red}}^{-1}(W)$, locally admissible in $C^{\mathrm{an}}$, where $W$ is a non-empty open subset of the special fiber ${\mathcal {C}}_{k}$.
Let $\xi_1, \xi_2, \ldots, \xi_n$ be distinct divisorial points in $E$ corresponding to vertical prime divisors $D_1, D_2, \ldots, D_n$ of ${\mathcal {C}}$, respectively.
This means that $W\bigcap D_i$ is not empty for each $i$.
Let $U_i$ be ${\mathrm{red}}^{-1}(y_i)$ where $y_i$ is a regular closed point in $W$ contained in $D_i$.
\begin{lemma}\label{lem:non-apr-fun}
  Let $x$ be an algebraic point on $C^{\mathrm{an}}$ outside $E$.
  Let $g$ be a Green function of $x/\deg(x)$ such that $c_1(g)=\mu$ where $\mu$ is a probability measure supported on the set $\{\xi_1, \xi_2, \ldots, \xi_n\}$.
  For every $\varepsilon>0$, there is a rational function $f$ with pole divisor $kx$ and zeros contained in $\bigcup_{i=1}^nU_i$ such that
  $$
    0\leq\frac{\log\left|f\right|}{k\deg(x)}-g\leq \varepsilon
  $$
  on $C^{\mathrm{an}}\setminus(\{x\}\bigcup_{i=1}^nU_i)$.
  Note that $\left|\frac{\log\left|f\right|}{k\deg(x)}-g\right|$ extends continuously to $x$, and the above shows that the value is at most $\varepsilon$. Moreover, for each connected component $E'$ of $E$, the number of zeros of $f$ in $E'$ is at least $k\deg(x)(\mu(E')-\varepsilon)$.
\end{lemma}
\begin{proof}For each $i\leq n$, by Lemma \ref{lem:rat-non-ope}, there is a rational function $f_i$ with pole divisor $n_ix$ and zeros contained in $U_i$.
  Let $g_i$ be the Green function of $x/\deg(x)$ with $c_1(g_i)=\delta_{\xi_i}$.
  Then the function
  $$
    \frac{\log\left|f_i\right|}{n_i\deg(x)}-g_i
  $$
  extends to a harmonic function on $C^{\mathrm{an}}\setminus U_i$.
  Since the boundary of $C^{\mathrm{an}}\setminus U_i$ is a point, the function is constant on $C^{\mathrm{an}}\setminus U_i$ by uniqueness for the Dirichlet problem.
  We may assume that the above constant is $0$ by adding a constant to $g_i$.

  Let $\kappa$ be a non-zero element in $K$ with $\left|\kappa\right|\neq 1$.
  Then $g$ is an ${\mathbb {R}}$-linear combination of $g_i$ and $\log\left|\kappa\right|$: $g=\sum_{i=1}^{n}r_ig_i+ r\log\left|\kappa\right|$ for $r_i=\mu(\xi_i)>0$.
  Note that $g_i-g_j$ is bounded on $C^{\mathrm{an}}\setminus(\{x\}\bigcup U_i\bigcup U_j)$.
  Thus, $g$ can be approximated by ${\mathbb {Q}}$-linear combinations of $g_i$ and $\log\left|\kappa\right|$.
  Assume that $-\varepsilon\leq g-\sum_{i=1}^np_ig_i-p\log\left|\kappa\right| \leq 0$ where $p_1, p_2, \ldots, p_n\geq 0$ and $p$ are rational numbers with $p_1+p_2+\cdots+p_n=1$.
  Moreover, $p_i$ can be taken such that $\left|p_i-r_i\right|\leq \varepsilon/n$.

  Take $d$ to be an integer such that $dp_i$ is a multiple of $n_i\deg(x)$ for all $i$, and that $dp$ is an integer.
  Let
  $$
    f:=\kappa^{dp}\prod\limits_{i=1}^{n}f_i^{\frac{dp_i}{n_i\deg(x)}}.
  $$
  Then $\log\left|f\right|=d\sum_{i=1}^np_ig_i+dp\log\left|\kappa\right|$ on $C^{\mathrm{an}}\setminus(\{x\}\bigcup_{i=1}^nU_i)$.
  Therefore,
  $$
    0\leq \frac{\log\left|f\right|}{d}-g= \sum_{i=1}^np_ig_i+p\log\left|\kappa\right|-g\leq \varepsilon
  $$
  on $C^{\mathrm{an}}\setminus(\{x\}\bigcup_{i=1}^nU_i)$.
  The number of zeros of $f$ in $E'$ is $d\sum_{\xi_i\in E'}p_i\geq d\sum_{\xi_i\in E'}(r_i-\varepsilon/n)\geq d(\mu(E')-\varepsilon)$.
\end{proof}

Let ${X}$ be a finite set of algebraic points on $C^{\mathrm{an}}$ outside $E$.
Let $g_x$ be a Green function of $x$.
Let $g_{E, x}$ be the equilibrium potential of $E$ at $x$.
Let $\mu_{E, x}=c_1(g_{E, x})$ be the equilibrium measure which is a probability measure supported on $\partial E$.
Fix $\varepsilon>0$.
By Lemma \ref{lem:non-apr-fun}, there is a rational function $f_x$ such that the pole divisor of $f_x$ is $d_xx$, that the zeros of $f_x$ are contained in $E$, and that
$0\leq \log\left|f_x\right|- \deg(x)d_xg_{E, x}\leq \deg(x)d_x\varepsilon$
on $C^{\mathrm{an}}\setminus(\{x\}\bigcup E)$.
Let $D$ be an effective ${\mathbb {Q}}$-divisor of degree $1$ supported outside ${X}$ and $E$.
Let $g_D$ be a Green function of $D$ such that $c_1(g_D)$ is a probability measure supported on finitely many divisorial points.
Then there exists a rational function $f$ such that the pole divisor of $f$ is $dD$, that the zeros of $f$ are contained in $E$, and that
$
  0\leq \log\left|f\right|-dg_D\leq d\varepsilon
$
on $C^{\mathrm{an}}\setminus({\mathrm{supp}}(D)\bigcup E)$.
By taking powers of $f_x$ and $f$, we may assume that $d=d_x\geq 2g$ for each $x\in {X}$.

For $x\neq y\in {X}$, set $c_{x,y}=g_{E, x}(y)$ and $c'_{x,y}=(d\deg(y))^{-1}\log\left|f_y(x)\right|$.
Take $c_{x,x}$ to be the real number $\lim_{x'\to x}(g_{E, x}(x')-g_x(x')/\deg(x))$.
Set $c'_{x,x}=(d\deg(x))^{-1}\log\left|f_x\right|_{x, d}$.
Then $\left|c_{x,y}-c'_{x,y}\right|\leq \varepsilon$.
Set $a_x=g_D(x)$ and $a'_x=d^{-1}\log\left|f(x)\right|$.
Let $s_x$ be a positive rational number for each $x$.
Set $c'_x=\frac{1}{s_x}\left(\sum_{y\in {X}} s_y\deg(y)c'_{x,y}\right)$.
Let $N_1$ be the smallest integer such that $N_1s_x$ is an integer at least $2g+2$ for each $x$.
Take $N$ to be any multiple of $dN_1$.
Now, we construct the patching map.

\begin{proof}[Proof of Proposition \ref{prop:loc-sec}, non-archimedean case]
  For each $x\in X$ and $2g<j\leq2g+dN_1s_x$, choose
  representatives
  \[
    \{h_{x,j'}:\deg(x)(j-1)<j'\leq\deg(x)j\}
  \]
  of a basis of $H^0(jx)/H^0((j-1)x)$.
  Note that by the Riemann--Roch theorem, we have
  $$
    \dim_{K}(H^0(jx)/H^0((j-1)x))=\deg(x).
  $$
  Consider the norm $\left|\cdot\right|_{x,j}$ on $H^0(jx)/H^0((j-1)x)$.
  Let $M_1$ be a real number such that for each $x\in {X}$, $2g<j \leq2g+dN_1s_x$, and $\phi\in H^0(jx)/ H^0((j-1)x)$ with $\left|\phi\right|_{x,j}\leq e^{jc'_x}$, $\phi$ is represented by a linear combination of $\{h_{x,j'}\}_{\deg(x)j-\deg(x)+1\leq j'\leq \deg(x)j}$ with coefficients of norm at most $M_1$.
  Set
  $$
    M_2=\sup\limits_{x\in {X}, 2g\deg(x)<j' \leq 2g\deg(x)+dN_1s_x\deg(x)}\sup\limits_{z\in \partial E}\left|h_{x,j'}(z)\right|.
  $$

  Set $f'=\prod_{x\in {X}}f_x^{N_1s_x}$.
  For each $j'>\deg(x)(2g+dN_1s_x)$, let $t$ be the integer such that $2g\deg(x)<j'-dtN_1s_x\deg(x) \leq\deg(x)(2g+dN_1s_x)$.
  Set $h_{x, j'}=h_{x, j'-dtN_1s_x\deg(x)}{f'}^{t}$.
  Then for each $j>2g$ and $\phi\in H^0(jx)/H^0((j-1)x)$ with $\left|\phi\right|_{x,j}\leq e^{jc'_x}$, $\phi$ is represented by a linear combination of $\{h_{x,j'}\}_{\deg(x)j-\deg(x)+1\leq j'\leq \deg(x)j}$ with coefficients of norm at most $M_1$.

  Set $\phi_N^{\mathrm{init}}=f^{(N/d)}\prod_{x\in {X}}f_x^{Ns_x/d}$.
  Then $\phi_N^{\mathrm{init}}$ is an $n$-th power if $N$ is a multiple of $N_1dn$.
  Take $0<B<1/(M_1M_2)$.
  Let $\{D_N^{(j)}\}$ be a band filtration of $(s_x)$ and $N$.
  Let $(\phi_N^{(j)})$ be a patching vector with initial element $\phi_N^{\mathrm{init}}$ and bounded by $B$ and $(c'_x)$.

  Set $\varphi_N^{(j)}=\phi_N^{(j)}f^{-N/d}$ for $j\geq 1$ and $\varphi_N^{(0)}=\left(\phi_N^{(0)}-\phi_N^{\mathrm{init}}\right)f^{-N/d}$.
  Then $\left|\varphi_N^{(j)}\right|_{x, {\mathrm{ord}}_{x}(D_N^{(j)})}\leq Be^{c'_x{\mathrm{ord}}_x(D_N^{(j)})}$.
  Therefore, we may take $\{b_{x, j'}\in K\}_{\deg(x){\mathrm{ord}}_x(D_N^{(j)})-\deg(x)+1\leq j'\leq \deg(x){\mathrm{ord}}_x(D_N^{(j)})}$ such that
  $$
    \sum\limits_{x\in {X}}\sum\limits_{j'=\deg(x){\mathrm{ord}}_x(D_N^{(j+1)})+1}^ {\deg(x){\mathrm{ord}}_x(D_N^{(j)})}b_{x, j'}h_{x, j'}
  $$
  represents $\varphi_N^{(j)}$, and we have $\left|b_{x, j'}\right|<1/M_2$.
  Thus,
  $$
    \left|\sum\limits_{x\in {X}}\sum\limits_{j'=\deg(x){\mathrm{ord}}_x(D_N^{(j+1)})+1}^ {\deg(x){\mathrm{ord}}_x(D_N^{(j)})}b_{x, j'}h_{x, j'}f^{N/d}\right|<\left|f^{N/d}f'^{N/(N_1d)}\right|
  $$
  on $\partial E$ since $\left|f'\right|\geq 1$ on $\partial E$.

  Finally, let
  $$
    G_N=\phi_N^{\mathrm{init}}+\sum\limits_{x\in {X}}\sum\limits_{j'=2g\deg(x)+1}^{Ns_x\deg(x)}f^{N/d}b_{x, j'}h_{x, j'}.
  $$
  Since $G_N$ is a multiple of $f^{N/d}$, the number of zeros of $G_N$ in each connected component is at least that of $f^{N/d}$.
  Also, the zeros of $G_N$ are contained in $E$ by Proposition \ref{prop:rou-thm}.

\end{proof}
Recall that ${\mathcal {C}}$ is a projective regular model of $C$ over $\operatorname{Spec}(O_K)$, $W$ is a non-empty open subset of the special fiber, and $E={\mathrm{red}}^{-1}(W)$.
When the special fiber of ${\mathcal {C}}$ is integral and the measure $\mu$ is supported on the divisorial point corresponding to the special fiber, in Lemma \ref{lem:non-apr-fun}, we can make all zeros of $f$ reduce to a given regular closed point in $W$.
In the proof of Proposition \ref{prop:loc-sec}, we take an initial function and modify it by functions that are strictly smaller than the initial function on $\partial E$, which does not change the reductions of roots.
Therefore, we have the following corollary.
\begin{corollary}\label{cor:red-sam-pt}
  If the special fiber of ${\mathcal {C}}$ is integral and the measure $\mu$ is supported on the divisorial point corresponding to the special fiber, then for any regular closed point $z\in W$, we can make all zeros of $G_N$ in Proposition \ref{prop:loc-sec} reduce to $z$.
\end{corollary}
\section{Global construction}
\label{sec:global-patching}

In this section we prove technical lemmas and propositions concerning global patching of the local functions that we constructed in Section \ref{sec:local-construction}. They will be used in the proof of Theorem \ref{thm:integral-point-construction} in Section \ref{sec:integral-point-construction}.

Let $K$ be a number field.
Let $C$ be a projective, smooth and geometrically integral curve of genus $g$ over $K$.
Let $S$ be a finite set of places of $K$ containing all archimedean places.
First, we recall some basic properties of $O_{K,S}$.
\begin{lemma}\label{lem:apr-int}
  There exists a constant $A$ such that for every collection of positive real numbers $\{A_v\}_{v\in S}$ with $\prod_{v\in S}A_v\geq A$, and every collection of numbers $\{a_v\in K_v\}_{v\in S}$, there is a number $a\in O_{K,S}$ with $\left|a-a_v\right|_v\leq A_v$ for each $v\in S$.
\end{lemma}
\begin{proof}
  By the Dirichlet unit theorem for $S$-units, the image of the map
  \begin{align*}
    O_{K,S}^\times&\longrightarrow {\mathbb {R}}^S\\
    k&\longmapsto (\log\left|k_v\right|)_{v\in S}
  \end{align*}
  is a complete lattice $\Gamma$ in the linear subspace $H:=\{(r_v)_{v\in S}|\sum_{v\in S}r_v=0\}$.
  Endow $H$ with the metric $\lVert(r_v)_{v\in S}\rVert=\max_{v\in S}\left|r_v\right|$.
  Let $A'$ be the last Minkowski minimum of $\Gamma$ under the metric $\lVert\cdot\rVert$.
  Let $m=\#S$.
  Then any ball in $H$ whose radius is at least $mA'$ contains a lattice point.
  For any $(r_v)_{v\in S}$ with $r=\sum_{v\in S}r_v\geq m^2A'$, the ball with center $(r_v-r/m)_{v\in S}$ and radius $r/m$ in $H$ contains a lattice point.
  Thus, there exists $k\in O_{K,S}^\times$ such that $\log\left|k\right|_v\leq r_v$ for every $v\in S$.

  The image of the map
  \begin{align*}
    O_{K}&\longrightarrow \prod\limits_{v\in M_{K,\infty}}K_v\\
    k&\longmapsto (k_v)_{v\in M_{K,\infty}}
  \end{align*}
  is also a complete lattice $\Gamma'$.
  Endow $\prod_{v\in M_{K,\infty}}K_v$ with the metric
  \[
    \lVert(k_v)_{v\in M_{K,\infty}}\rVert'=\max_{v\in M_{K,\infty}}\left|k_v\right|.
  \]
  Let $A''$ be the last Minkowski minimum of $\Gamma'$ under the metric $\lVert\cdot\rVert'$.
  Let $d=[K:{\mathbb {Q}}]$, $A=(dA'')^de^{m^2A'}$.
  Then there exists $k\in O_{K,S}^\times$ such that $\left|k\right|_v\leq A_v$ for each $v\in S$ non-archimedean, and that $\left|k\right|_v\leq A_v/(dA'')$ for each $v\in S$ archimedean.
  By the weak approximation theorem, there exists $a'\in O_{K,S}$ such that $\left|a'-a_v/k\right|_v\leq 1$ for each $v\in S$ non-archimedean.
  Also, there is a number $k'\in O_K$ such that $\left|a'-a_v/k-k'\right|_v\leq dA''$ for each archimedean $v$.
  The lemma follows by taking $a=k(a'-k')$.
\end{proof}
\begin{lemma}\label{lem:apr-uni}
  Let $S_f$, $S_\infty$ be the sets of all non-archimedean places and archimedean places in $S$, respectively. Let $\{A_v\}_{v\in S_f}$ be a collection of positive real numbers. Then there exist positive real numbers $\{A_v\}_{v\in S_\infty}$ and a positive integer $n$ such that for every collection of numbers $\{a_v\in K_v\}_{v\in S}$ with $\prod_{v\in S}\left|a_v\right|_v=1$, there is a number $a\in O_{K,S}^\times$ with the following properties.
  For each $v\in S_f$, $\left|a-a_v^n\right|_v\leq A_v\left|a_v^n\right|_v$.
  Moreover, for each $v\in S$ with $K_v={\mathbb {C}}$, $1/A_v\leq \left|aa_v^{-n}\right|_v\leq A_v$.
  Lastly, for each $v\in S$ with $K_v={\mathbb {R}}$, $1/A_v\leq aa_v^{-n}\leq A_v$.
\end{lemma}
\begin{proof}
  Let $n_1$ be the class number of $O_K$.
  Then there exists a number $k_1\in O_{K,S}^\times$ such that for each $v\in S_f$, $\left|k_1\right|_v=\left|a_v^{n_1}\right|_v$.
  Therefore, we may assume that $\left|a_v\right|_v=1$ for each $v\in S_f$.
  Without loss of generality, we may assume that $A_v\leq 1$ for each $v\in S_f$.
  The condition $\left|a-a_v^n\right|_v\leq A_v\left|a_v^n\right|_v$ is equivalent to $a\equiv a_v^n \mod \mathfrak{a}_v$ where $\mathfrak{a}_v$ is a power of $\mathfrak{p}_v$, the prime ideal corresponding to $v$.
  Let $\mathfrak{a}=\prod_{v\in S_f}\mathfrak{a}_v$.
  Let $n_2$ be the number of elements in $(O_K/\mathfrak{a})^\times$.
  Then $1\equiv a_v^{n_2}\mod \mathfrak{a}_v$ for each $v\in S_f$.

  By the Dirichlet unit theorem, the image of the map
  \begin{align*}
    O_{K}^\times&\longrightarrow {\mathbb {R}}^{S_\infty}\\
    k&\longmapsto (\log\left|k_v\right|)_{v\in S_\infty}
  \end{align*}
  is a complete lattice $\Gamma$ in the linear subspace $H:=\{(r_v)_{v\in S_{\infty}}|\sum_{v\in S_{\infty}}r_v=0\}$.
  Endow $H$ with the metric $\lVert(r_v)_{v\in S_\infty}\rVert=\max_{v\in S_{\infty}}\left|r_v\right|$.
  Let $A'$ be the last Minkowski minimum of $\Gamma$ under the metric $\lVert\cdot\rVert$.
  Let $m=\#S_{\infty}$.
  Then any ball in $H$ whose radius is at least $mA'$ contains a lattice point.
  Thus, there exists $k_2\in O_{K}^\times$ such that $\left|\log\left|k_2/a_v\right|_v\right|<mA'$ for each $v\in S_\infty$.
  Now, for each $v\in S_f$, $k_2^{2n_2}\equiv 1\mod \mathfrak{a}_v$.
  For each $v\in S_\infty$, $e^{-2n_2mA'}\leq \left|k_2^{2n_2}a_v^{-2n_2}\right|_v\leq e^{2n_2mA'}$.
  Since $2n_2$ is an even number, $k_2^{2n_2}a_v^{-2n_2}$ is positive at every real place.
  The lemma follows by taking $n=2n_1n_2$ and $A_v=e^{2n_2mA'}$ for each $v\in S_\infty$.
\end{proof}

Let ${X}$ be a finite set of closed points of $C$.
Let ${\mathcal {C}}$ be a projective regular model of $C$ over $\operatorname{Spec}(O_K)$.
For each non-archimedean place $v$, let $k_v$ be the residue field at $v$.
For a non-archimedean place $v$, a compact set $E_v\subset (C\setminus X)_v^\mathrm{an}$ is said to be \emph{simple with respect to $(\mathcal C,X)$} if $\mathcal C_{k_v}$ is integral and
\[
  E_v=\operatorname{red}_v^{-1}\bigl((\mathcal C\setminus\overline X)_{k_v}\bigr),
\]
where $\overline X$ is the Zariski closure of $X$ in $\mathcal C$.
Let $S$ be a finite set of places containing all archimedean places, with a non-archimedean place $v_0\in S$.
Let ${\mathbb{E}}=(E_v)_{v\in M_K}$ be an admissible adelic set on $C\setminus X$, such that $E_v$ is simple at every archimedean place and simple with respect to $(\mathcal C,X)$ at each non-archimedean place not in $S\setminus \{v_0\}$.
Set ${\mathcal {C}}_S={\mathcal {C}}\times_{\operatorname{Spec}(O_K)}\operatorname{Spec}(O_{K,S})$.
A rational function $f$ on $C$ is said to be \emph{integral} on ${\mathcal {C}}_S$ if its pole divisor on ${\mathcal {C}}_S$ does not contain any vertical component, and is said to be \emph{monic} at a closed point $z\in C$ of order $d$ if the support of ${\mathrm{div}}(f)+d\overline{z}$ does not contain $z$ and $({\mathrm{div}}(f)+d\overline{z})|_{\overline{z}}$ is a trivial divisor where $\overline{z}$ denotes the Zariski closure of $z$ in ${\mathcal {C}}_S$.

Let $D'$ be a non-zero effective divisor on $C$, supported outside ${X}$, and set $d_D=\deg(D')$.
By enlarging $S$, we may assume that $O_{K,S}$ and $O_{K(x), S}$ are principal ideal domains for each $x\in {X}$ and
that the reductions of any two distinct $\overline{{\mathbb {Q}}}$-points in ${X}\bigcup{\mathrm{supp}}(D')$ to ${\mathcal {C}}_{\overline{k_v}}$ are disjoint for each $v\notin S$.
For each point $x\in{X}$, and $2g<j\leq 4g+1$, consider the $1$-dimensional $K(x)$-vector space $H^0(jx)/H^0((j-1)x)$.
Since $O_{K(x),S}$ is a principal ideal domain, monic functions exist in $H^0(jx)$.
Consider the $O_{K(x), S}$-submodule of $H^0(jx)/H^0((j-1)x)$ generated by a monic function.
Find representatives of a set of $O_{K, S}$-generators of this module. By enlarging $S$, we can make these representatives integral on ${\mathcal {C}}_S$.
Now, for each $m> 2g$, an integral function $h_{x, m}$ monic at $x$ of order $m$ exists in $H^0({\mathcal {C}}_S, {\mathcal {O}}(m\overline{x}))$ as a generator of the rank $1$ $O_{K(x),S}$-module $H^0({\mathcal {C}}_S, {\mathcal {O}}(m\overline{x}))/H^0({\mathcal {C}}_S, {\mathcal {O}}((m-1)\overline{x}))$.
Note that we can take $h_{x, (2g+2)m}=h_{x, 2g+2}^{m}$.
For each $v\notin S$ and $x'\in x_v$, we take a Green function $g_{x'}$ of $x'$ such that $(2g+2)g_{x'}=\log\left|h_{x, 2g+2}\right|_v$ in a neighborhood of $x'$.

\begin{proposition}\label{prop:pat-con}
  Let $(s_x)_{x\in {X}}$ be a vector consisting of positive rational numbers.
  Let $\{D_N^{(j)}\}$ be a band filtration.
  Then $\{D_{N, v}^{(j)}\}$ is a band filtration for each $v$.
  Assume that $D_{N}^{(j)}-D_N^{(j+1)}$ is a prime divisor for each $j\geq 1$.
  Let $\{B_v\}_{v\in S_f}$ be positive real numbers.
  Let $\{c_{v, x, y}\}_{v\in S, x, y\in {X}_v}$ and $\{a_{v, x}\}_{v\in S, x\in {X}_v}$ be real numbers.
  Set
  $$
    c_{x, y}=\sum_{v\in S, x'\in x_v, y'\in y_v}\deg(x')\deg(y')c_{v, x', y'}/(\deg(x)\deg(y))
  $$
  and $a_{x}=\sum_{v\in S, x'\in x_v}\deg(x')a_{v, x'}/\deg(x)$.
  Take $\varepsilon>0$.
  Assume that for each $x\in {X}$, $a_x>0$ and $a_x+\sum_{y\in {X}}s_y\deg(y)c_{x, y}=0$.
  Then there exist real numbers $\{B_v>0\}_{v\in S_\infty}$, an even integer $n>0$ and an integer $N'>0$ with the following property.

  Set $s_{v, x'}=s_x$ for each $v\in S, x'\in x_v$.
  For each $N$ that is a multiple of $N'$ and each band filtration $\{D_N^{(j)}\}$ for $(s_x)$ and $N$, we notice that $\{D_{N, v}^{(j)}\}$ is a band filtration for $(s_{v, x})$ and $N$ for each $v$.
  Assume that $D_{N}^{(j)}-D_N^{(j+1)}$ is a prime divisor for each $j\geq 1$.
  Consider initial elements $\{\phi_{N,v}^{\mathrm{init}}\}_{v\in S_f}$ which are represented by $n$-th powers such that $\log\left|\phi_{N,v}^{\mathrm{init}}\right|_{x, Ns_{v, x}}=Na_{v, x}+\sum_{y\in {X}_v}Ns_{v, y}\deg(y)c_{v, x, y}$.
  Then there exist $c'_{v, x, y}$ for each $v\in S_{\infty}, x, y\in {X}_v$ such that $c'_{v, x, y}=c_{v, x, y}$ for $x\neq y$ and $\left|c'_{v, x, y}-c_{v, x, y}\right|<\varepsilon$ for $x=y$.
  For $v\in S_f$, set $c'_{v, x, y}=c_{v, x, y}$.
  Write $c'_{v, x}=\sum_{y\in {X}_v}s_{v, y}\deg(y)c'_{v, x, y}/s_{v, x}$.
  Moreover, there exist initial elements $\{\phi_{N,v}^{\mathrm{init}}\}_{v\in S_\infty}$ which are represented by $n$-th powers such that $\log\left|\phi_{N,v}^{\mathrm{init}}\right|_{x, Ns_{v, x}}=Na_{v, x}+Ns_{v, x}c'_{v, x}$.
  Assume that we have patching maps $\psi_{N, v}$ for each $v\in S$.
  Finally, there exist patching vectors $(\phi_{N, v}^{(j)})$ with initial element $\phi_{N,v}^{\mathrm{init}}$ bounded by $B_v$ and $(c'_{v, x})_{x\in {X}_v}$, such that each $\psi_{N, v}(\boldsymbol{\phi}_{N, v})$ is represented by a common function $G_N$ in $H^0(D_N^{(0)})$, and $G_N$ is monic at each point in ${X}$ and integral on ${\mathcal {C}}_S$.
\end{proposition}
\begin{proof}
  The proof is in two steps.
  Firstly, we patch the top coefficients.
  We apply Lemma \ref{lem:apr-uni} to $O_{K(x), S}^\times$ for each $x\in {X}$.
  A place $x'$ of $K(x)$ is the same as a place $v$ of $K$ and a point $x'\in x_v$.
  We write $\left|\cdot\right|_{x'}$ for the norm given by the point $x'$.
  Note that the norm of $x'$ viewed as a place of $K(x)$ is $\left|\cdot\right|_{x'}^{\deg(x')}$.
  Set $B_{x'}=B_v^{\deg(x')}$ where $x'\in x_v$.
  Given $\{B'_v\}_{v\in S_f}$, we get $\{B'_v\}_{v\in S_{\infty}}$ and $n$.
  Take $N'$ large enough that $e^{Ns_x\varepsilon}>B'_v$ for each $x\in {X}$ and $v\in S_{\infty}$.
  Set $\alpha_{v, x'}=(\phi_{N,v}^{\mathrm{init}}h_{x, Ns_x}^{-1})(x')$ for each $x\in {X}, v\in S_f, x'\in x_v$.
  Note that by taking $N'$ appropriately, we can make $Ns_x$ a multiple of $2g+2$ for each $x$.
  Then $\alpha_{v, x'}$ is an $n$-th power.
  Also, $\log\left|\alpha_{v, x'}\right|_{x'}=Na_{v, x'}+Ns_{v, x'}c'_{v, x'}$.
  For archimedean $v$, let $\alpha_{v, x'}$ be a number with $\log\left|\alpha_{v, x'}\right|_{x'}=Na_{v, x'}+Ns_{v, x'}c'_{v, x'}$.
  If $K(x')={\mathbb {R}}$, we take $\alpha_{v, x'}$ to be positive.

  Now we have $\prod_{v\in S, x'\in x_v}\left|\alpha_{v, x'}\right|_{x'}^{\deg(x')}=1$.
  By Lemma \ref{lem:apr-uni}, there exists $\beta_x\in K(x)$ such that for $v$ non-archimedean, $\left|\alpha_{v, x'}-\beta_{x}\right|_{x'}\leq B_v\left|\alpha_{v, x'}\right|_{x'}$.
  Also, for $v$ archimedean, $1/B_v\leq \left|\alpha_{v, x'}/\beta_{x}\right|_{x'}\leq B_v$.
  Lastly, $\beta_{x}$ is positive at each real place.

  For each $v\in S_{\infty}$ and $x'\in x_v$, set $c'_{v, x', x'}=c_{v, x', x'}-(N\deg(x')s_{x})^{-1}\log\left|\alpha_{v, x'}/\beta_{x}\right|_{x'}$.
  Then $\left|c'_{v, x', x'}-c_{v, x', x'}\right|<\varepsilon$.
  Also, $\log\left|\beta_x h_{x, Ns_x}\right|_{x', Ns_x}=\log\left|\beta_x\right|=Na_{v,x'}+Ns_xc'_{v, x'}$.
  Now, each $\beta_x h_{x, Ns_x}$ can be viewed as an element in $H^0(Ns_xx)/H^0(Ns_xx-x)$, and is represented by a function $G_x\in H^0(Ns_xx)$ monic at $x$ and integral on ${\mathcal {C}}_S$.
  We take $G_N^{(0)}=\sum_{x\in {X}}G_x$.
  Then $G_N^{(0)}$ is monic at each point in ${X}$.
  Let $\phi_{N, v}^{(0)}$ be the class containing $G_{N, v}^{(0)}$.

  Now we patch the lower coefficients.
  We take $\{B_v\}_{v\in S_\infty}$ large enough that $\prod_{v\in S}B_v^{\deg(x)}\geq A_x$ for each $x\in {X}$ where $A_x$ is the constant $A$ in Lemma \ref{lem:apr-int} for $O_{K(x), S}$.
  Given $\phi_{N, v}^{(j)}$ for $j\leq j'$, $\psi_{N, v}$ gives us a class in $H^0(D_N^{(0)})/H^0(D_N^{(j'+1)})$.
  We construct $\phi_{N, v}^{(j)}$ for $j\leq j'$ together with a monic function $G_N^{(j')}$ integral on ${\mathcal {C}}_S$ whose reduction at each $v\in S$ is in the class.
  We prove by induction.
  Assume we have completed the construction for $j'$.
  For $j'+1$, let $x$ be the prime divisor $D_N^{(j'+1)}-D_N^{(j'+2)}$.

  Set $H_{N, v}^{(j')}=\psi_{N, v}(\phi_{N, v}^{(0)}, \ldots, \phi_{N, v}^{(j')}, 0, \ldots, 0)$.
  Set $d={\mathrm{ord}}_{x}(D_N^{(j'+1)})$.
  Set $\alpha'_{v, x'}=((G_{N, v}^{(j')}-H_{N, v}^{(j')})h_{x, d}^{-1})(x')$ for each $v\in S, x'\in x_v$.
  Set $b_{v, x'}=B_ve^{Na_{v,x'}+dc'_{v,x'}}/\left|h_{x, d}\right|_{x', d}$.
  Set $c'_x=\sum_{v\in S, x'\in x_v}\deg(x')c'_{v, x'}/\deg(x)$.
  Since $G_{N}^{(0)}$ is monic at $x$, $a_x+s_xc'_x=0$.
  Then
  $$
    \prod_{v\in S, x'\in x_v}b_{v, x'}^{\deg(x')}=e^{\deg(x)(Na_x+dc'_x)}\prod_{v\in S}B_{v}^{\deg(x)}\geq A_x.
  $$
  Therefore, there exists $\beta'\in O_{K(x), S}$ such that $\left|\alpha'_{v, x'}-\beta'\right|_{x'}\leq b_{v, x'}$.
  Note that $\beta'h_{x, d}$, viewed as an element of $H^0(dx)/H^0(dx-x)$, is represented by an integral function $f$ on ${\mathcal {C}}_S$.
  Take $G_{N}^{(j'+1)}=G_{N}^{(j')}-f$.
  Take $\phi_{N, v}^{(j'+1)}$ to be the class of $G_{N, v}^{(j'+1)}-H_{N, v}^{(j')}$.
  Then $\left|\phi_{N, v}^{(j'+1)}\right|_{x', d}\leq B_ve^{Na_{v,x'}+dc'_{v,x'}}$.
  The proposition follows.

\end{proof}

Choose $1<c_0<\operatorname{cap}_X({\mathbb{E}})$, and assume that $E_v$ is disjoint from ${\mathrm{supp}}(D'_v)$ for each $v\in S$.
Let $\{\mu_v\}_{v\in S}$ be a collection of measures such that for each archimedean $v$, ${\mathrm{supp}}(\mu_v)$ lies in the ${\mathbb {R}}$-interior of $E_v$ and that for each non-archimedean $v$, $\mu_v$ is a finite linear combination of Dirac measures of divisorial points in $E_v$.
Also, let $\mu_{v_0}$ be the Dirac measure at the divisorial point corresponding to the special fiber ${\mathcal {C}}_{k_{v_0}}$.
Let $g_{D',v}$ be a Green function of $D'_v$ with $c_1(g_{D',v})=d_D\mu_v$ for each $v\in S$.
For $v\in S$ and $x'\in {X}_v$, set $a_{v,x'}=d_D^{-1}g_{D',v}(x')$.
For each $x\in {X}$, set $a_x=\sum_{v\in S}\sum_{x'\in x_v}\deg(x')a_{v,x'}/\deg(x)$.

\begin{proposition}\label{prop:glo-sec}
  Let $\delta'$ be a positive real number.
  Assume that $0<a_{x}<\delta'$.
  Fix a closed point $p_0$ in the regular locus of the special fiber ${\mathcal {C}}_{k_{v_0}}$ which is not contained in the reductions of ${X}$ and ${\mathrm{supp}}(D')$.
  For each $\varepsilon>0$, there exist a positive integer $k$ and a rational function $f_k$ on $C$, with the following properties:
  \begin{enumerate}
    \item the pole divisor of $f_k$ is
      $$
        \operatorname{div}_\infty(f_k)=kD'+\sum\limits_{x\in {X}} \ell_xx
      $$
      where $\{\ell_x\}_{x\in {X}}$ are positive integers such that
      $$
        \ell_x\leq \frac{3\delta'}{\log c_0}\deg(D')k
      $$
      for each $x\in {X}$;
    \item there is a rational function $f'_k$, such that $f_k-f'_k\in H^0(kD'+2g\sum_{x\in {X}}x)$ and that on ${\mathcal {C}}_{S}$, $f'_k$ is integral and monic at every point in ${X}$;
    \item every zero of $f_k$ reduces to $p_0$ at $v_0$;
    \item \label{item:global-zero-distribution} for each $v \in S$ and connected component $A$ of $E_v$, the number of zeros of $f_{k,v}$ in $A$, counted with multiplicity, is at least $\deg(D')k(\mu_v(A)-\varepsilon)$;
    \item \label{item:global-boundary-dominance} for each $v\in S_\infty$, $(f_{k,v}, kg_{D',v}-1)$ satisfies the boundary dominance condition on $E_v$, and for each $v\in S_f$, $(f_{k,v}, kg_{D',v})$ satisfies the boundary dominance condition on $E_v$.
  \end{enumerate}
  Also, $k$ can be taken arbitrarily large.
\end{proposition}
\begin{proof}

  Set $m=\# S$.
  Let $(c_{x, y})_{x, y\in {X}}=\Gamma=\Gamma_X(\mathbb E)$ be the Cantor matrix.
  Set $a_{min}=\min_{x\in {X}}\{a_x\}>0$.
  For each vector $\mathbf{a}$ with each coordinate between $a_{min}/(2\delta')$ and $2$, by Lemma \ref{lem:can-cap-sol} and Lemma \ref{lem:mat-inv}, each coordinate of $\Gamma^{-1}\mathbf{a}$ is a negative number at least $-2/\log c_0$.
  By the compactness of the set of all such vectors $\mathbf{a}$, we see that there exist $\varepsilon_1>0, \varepsilon_2>0$, such that for each $\Gamma'=(c'_{x, y})_{x, y\in {X}}$ with $\left|c'_{x, y}-c_{x, y}\right|\leq \varepsilon_1$ for each $x, y\in {X}$, each coordinate of $\Gamma'^{-1}\mathbf{a}$ is between $-3/\log c_0$ and $-\varepsilon_2$.
  We may take $\varepsilon_1<a_{min}/2$.

  For $v\in M_K$ and $x'\neq y'\in {X}_v$, set $c_{v,x',y'}=g_{E_v, x'}(y')$ and $c_{v,x',x'}=\lim_{x''\to x'}(g_{E_v, x'}(x'')-g_{x'}(x'')/
    \deg(x'))$.
  Note that for $v\notin S$, $c_{v,x',y'}=0$.
  Thus,
  $$
    c_{x, y}=\frac{1}{\deg(x)\deg(y)}\sum\limits_{v\in M_K}\sum\limits_{x'\in x_v, y'\in y_v}\deg(x')\deg(y')c_{v,x',y'}.
  $$
  Set $\kappa_i=i!=\prod_{j=1}^ij$.
  By Proposition \ref{prop:loc-sec}, applied to $D'/d_D$ and $g_{D',v}/d_D$ at each archimedean place recursively and each non-archimedean place in $S$, we get a subsequence $\kappa'_i$, real numbers $c'_{v,x',y'}$ with $\left|c'_{v,x',y'}-c_{v,x',y'}\right|\leq \varepsilon_1/m$ and $a'_{v,x'}$ with $\left|a'_{v,x'}-a_{v,x'}\right|\leq \varepsilon_1/m$ for each $v\in S$ and $x', y'\in {X}_v$.
  Moreover, for archimedean $v$, $c'_{v, x, x}$ can be taken as any real number in an open interval.
  Set $c'_{x, y}=\sum_{v\in S, x'\in x_v, y'\in y_v}\deg(x')\deg(y')c'_{v,x',y'}/(\deg(x)\deg(y))$ and $\Gamma'=(c'_{x, y})_{x, y\in {X}}$.
  Set $a'_x=\sum_{v\in S, x'\in x_v}a'_{v, x'}\deg(x')/\deg(x)$ and ${\mathbf{a}}'=(a_x')_{x\in {X}}$.
  Since $c'_{v, x, x}$ can be taken as any real number in an open interval, we can choose these numbers so that $(s_x\deg(x))_{x\in {X}}=-{\Gamma'}^{-1}\mathbf{a'}$ has rational coefficients.
  Note that $a_{min}/2\leq a'_x\leq2\delta'$, $\left|c'_{x, y}-c_{x, y}\right|\leq \varepsilon_1$.
  Thus, $0<s_{x}\leq 3\delta'/\log c_0$.
  Set $s_{v, x'}=s_{x}$ for each $x'\in x_v$ and ${\mathbf{s}}_v=(s_{v, x'})_{x'\in {X}_v}$.
  Set
  $$
    c'_{v, x}=\sum_{y\in {X}_v}c'_{v, x, y}s_{v,y}\deg(y)/s_{v,x}
  $$
  and ${\mathbf{c}}'_v=(c'_{v, x})_{x\in {X}_v}$.

  For each $N=\kappa'_i$ large enough and divisible by $d_D$, set $k=N/d_D$. For each patching divisor $D_{N, v}^{(j)}$ for ${\mathbf{s}}_v$ and $N$, we have a patching map $\psi_{N, v}$.
  Also, we get initial elements $\phi_{N,v}^{\mathrm{init}}$ for non-archimedean $v$ and $B_v$ for non-archimedean $v$.
  By Proposition \ref{prop:pat-con} and possibly changing the choice of $c'_{v, x, x}$ for archimedean $v$, we get initial elements $\phi_{N,v}^{\mathrm{init}}$ for archimedean $v$, $B_v$ for archimedean $v$, and patching vectors $\boldsymbol{\phi}_{N, v}$ for each $v\in S$ with initial elements $\phi_{N,v}^{\mathrm{init}}$ bounded by $B_v$ and ${\mathbf{c}}'_v$, such that $\psi_{N, v}(\boldsymbol{\phi}_{N, v})$ is represented by the same function $G_N\in H^0(D_N^{(0)})$ for each $v$.
  Set $f'_k=G_N$.

  Also, we have $G_{N, v}\in H^0(kD'_v+D_{N, v}^{(0)})$ such that $(G_{N, v}, kg_{D',v})$ satisfies the boundary dominance condition.
  By Corollary \ref{cor:red-sam-pt}, we can make all zeros of $G_{N, v_0}$ reduce to $p_0$.
  By the weak approximation theorem, there is a rational function $f_k\in H^0(kD'+D_{N}^{(0)})$, such that $f_{k,v}-G_{N, v}\in H^0(kD'_v+2g\sum_{x\in {X}_v}x)$ and that $f_k-G_{N, v}$ can be arbitrarily small on $\partial E_v$.
  Thus, we can ensure that $(f_{k,v}, kg_{D',v}-1)$ satisfies the boundary dominance condition for archimedean $v$, and $(f_{k,v}, kg_{D',v})$ satisfies the boundary dominance condition for non-archimedean $v$.
  Also, the number of zeros of $f_{k,v}$ in each connected component of $E_v$ is exactly the same as that of $G_{N, v}$ and every zero of $f_{k,v_0}$ reduces to $p_0$.
  Now, $\ell_x=Ns_x=d_Dks_x$.
  Since $s_x\leq 3\delta'/\log c_0$, $\ell_x\leq 3d_Dk\delta'/\log c_0$.
  The proposition follows.

\end{proof}
\section{Construction of algebraic integral points}
\label{sec:integral-point-construction}

In this section we prove the following result.
\begin{theorem} \label{thm:integral-point-construction}
  Let $K,C,X,U$ be as in Theorem~\ref{thm:adelic-sos}.
  Let $\mathbb{E}=(E_v)_{v\in M_K}$ be an admissible adelic set on
  $U$ with $\operatorname{cap}_X(\mathbb E)>1$. For every
  $v\in M_K$, let $\mu_v$ be a probability measure supported on $E_v$.
  Let $D$ be an $\mathbb R$-divisor on $C$ of degree $1$. Suppose that, for every flat adelic divisor
  $\overline F=(F,(g_{F,v})_{v\in M_K})$ for which $F|_U$ is
  effective, we have
  \begin{equation} \label{eq:integral-point-construction-hypothesis}
    \sum_{v\in M_K}\int g_{F,v}\,\mathrm d\mu_v
    \leq D\cdot\overline F.
  \end{equation}
  Then there exists a sequence of distinct algebraic points
  $(\alpha_n)$ in $C(\overline K)$ such that $\alpha'_{n, v} \subseteq E_v$ and $\delta_{\alpha_n, v} \longrightarrow \mu_v$ at every place $v \in M_K$, and $\alpha_n'/\deg(\alpha_n') \longrightarrow D$ in $\operatorname{Pic}^1(C)_\mathbb{R}$.
\end{theorem}

By Lemma~\ref{lem:gvf-on-curves} and Lemma~\ref{lem:gvf-point-convergence}, Theorem \ref{thm:int-existential-closedness-t-0} is a direct consequence.

\subsection{First reductions - need only to construct equidistribution at a finite set of places $T \subseteq M_K$}
Note that
\begin{itemize}
  \item since changing $D$ up to linear equivalence does not affect the inequality, we may assume that $D$ is effective and supported outside $X$.
  \item it suffices to establish equidistribution at any finite set of places $T \subseteq M_K$. Theorem \ref{thm:integral-point-construction} then follows by a standard diagonal argument.
\end{itemize}

For each $v\in M_K$, fix a compatible metric $\rho_v$ on $C_v^{\mathrm{an}}$. All diameters and Wasserstein distances at $v$ below are taken with respect to $\rho_v$.

\subsection{Archimedean regularization - choose $\delta,\eta>0$, and $c_0>1$}
Choose $\delta,\eta>0$ and $\operatorname{cap}_X(\mathbb{E})>c_0>1$. By Proposition \ref{prop:com-exh-nor} and Corollary \ref{cor:arc-mea}, we may assume at every archimedean place $v \in M_{K,\infty}$
\begin{itemize}
  \item $\mu_v$ has a continuous potential,
  \item $E_v \cap \operatorname{supp}(D)_v^\mathrm{an}=\emptyset$,
  \item $E_v$ is a disjoint union of small domains, each of diameter $<\eta$.
\end{itemize}
This archimedean regularization process introduces errors:
\begin{itemize}
  \item we can keep the margin $\operatorname{cap}_X(\mathbb{E})>c_0>1$,
  \item we have, for each degree zero divisor $F=F_+-F_-$ on $C$ with $F_-$ supported on $X$,
    \begin{equation}
      \sum \int g_{F,v} \,\mathrm{d}\mu_v \leq \overline{F} \cdot D + \delta \deg(F_-). \label{eq:delta-condition}
    \end{equation}
\end{itemize}

\subsection{Non-archimedean regularization and an adelic $\mathbb{R}$-divisor - make $S \supseteq T$ and construct $\overline{D}$}
Let $\mathcal{C}/O_K$ be an integral model of $C/K$ and $\mathcal{U}$ be an open subscheme of $\mathcal{C}$ with $\mathcal{U}_K=U$, such that $E_v=\operatorname{red}^{-1}(\mathcal{U}_{k_v})$ at every non-archimedean place $v \in M_{K,f}$.

Choose a bigger set of places $S_0 \supseteq T$ such that $S_0$ contains all archimedean places and the bad places where $E_v$ is not simple with respect to $(\mathcal C,X)$. Enlarge $S_0$, if necessary, so that it satisfies all the conditions for global patching as in Proposition \ref{prop:glo-sec}.

By Proposition \ref{prop:non-sub-avg}, we may also assume $\mu_v$ is the Dirac measure at the (unique) Shilov boundary of $E_v$ for each $v\notin S_0$.

By Proposition \ref{prop:non-mea} and Proposition \ref{prop:non-exh}, we may assume at every place $v \in S_0 \cap M_{K,f}$
\begin{itemize}
  \item $\mu_v$ is a finite convex combination of Dirac measures at divisorial points,
  \item $E_v \cap \operatorname{supp}(D)_v^\mathrm{an} = \emptyset$,
  \item $E_v$ is a disjoint union of locally admissible compact sets, each of diameter $<\eta$.
\end{itemize}

We carry out this replacement inside the original $E_v$ while retaining its entire Shilov boundary. Since the equilibrium measures are supported on that boundary, the local Green matrix (and hence $\operatorname{cap}_X(\mathbb E)$) remains unchanged.

Choose $v_0 \notin S_0$. We may shrink $E_{v_0}=\operatorname{red}^{-1} \big( \mathcal{C}_{k_{v_0}}-\overline{X}_{k_{v_0}} \big)$ to a smaller set $\operatorname{red}^{-1}\big( \mathcal{C}_{k_{v_0}}-\overline{X \cup \operatorname{supp}(D)}_{k_{v_0}} \big)$. Make $S=S_0 \cup \{v_0\}$.

Let $\overline{D}=(D,(g_{D,v})_{v\in M_K})$ be an adelic divisor such that $c_1(\overline{D})_v=\mu_v$ at every place $v \in M_K$. Then the local Green functions $(g_{D,v})_{v\in M_K}$ are determined up to constants, and we may choose these constants so that
\begin{itemize}
  \item $g_{D,v}(\xi_v)=0$ at $v \notin S_0$, where $\xi_v$ is the Berkovich point corresponding to the divisor $\mathcal{C}_{k_v}$. (Let $\mathcal{D}_{S_0}$ be the Zariski closure of $D$ in $\mathcal{C}_{S_0}$. Then $g_{D,v}$ is exactly the model Green function induced by $\mathcal{D}_{S_0}$.)
  \item $h_{\overline{D}}(x_0)=\delta$ for some $x_0 \in X$.
\end{itemize}

For any $x \in C$, applying \eqref{eq:delta-condition} to $F=x/\deg(x)-x_0/\deg(x_0)$ gives
\[
  h_{\overline{D}}(x) - h_{\overline{D}}(x_0) = \overline{F} \cdot D - \sum \int g_{F,v} \,\mathrm{d}\mu_v \geq -\delta.
\] It follows that $h_{\overline{D}}(x) \geq 0$ for all $x \in C$.
Similarly, for any $x \in X$, applying \eqref{eq:delta-condition} to $F=x_0/\deg(x_0)-x/\deg(x)$ gives $h_{\overline{D}}(x) \leq 2\delta$ for all $x \in X$.

\subsection{Adelic $\mathbb{Q}$-divisor - construct $\overline{D}^\prime$, choose $\delta^\prime>2\delta$}
Say $\operatorname{supp}(D)=\big\{ P_1,\dots,P_n \big\}$. Let $r_1,\dots,r_n \in \mathbb{R}_{>0}$ be such that $D+\sum r_i P_i$ is a $\mathbb{Q}$-divisor. Take an integer $m$ that clears the denominator of $D+\sum r_i P_i$. Then $D^\prime=m(D+\sum r_iP_i)$ is an integral divisor.

For a height function $h$ on $C(\overline K)$, write $\operatorname{abs}(h)=\inf_{\alpha\in C(\overline K)}h(\alpha)$ for its absolute minimum.
For any $i=1,\dots,n$, let $\overline{P}_i$ be an adelic divisor such that
\begin{itemize}
  \item $c_1(\overline{P}_i)_v=\deg(P_i) \mu_v$ at every place $v \in M_K$,
  \item it is induced outside $S_0$ by the Zariski closure of $P_i$ in $\mathcal{C}_{S_0}$,
  \item $\operatorname{abs}(h_{\overline{P}_i})>0$.
\end{itemize}
Let $\overline{D}^\prime=m(\overline{D}+\sum r_i\overline{P}_i)=(D^\prime,(g_{D^\prime,v})_{v\in M_K})$. Then clearly $\operatorname{abs}(h_{\overline{D}^\prime})>0$.
Choose $\delta^\prime>2\delta$. The coefficients $r_1,\dots,r_n$ can be made small enough so that $h_{\overline{D}^\prime}(x) \leq \delta^\prime \deg(D^\prime)$ for all $x \in X$.

\subsection{Global construction - choose $0<\rho<\operatorname{abs}(h_{\overline{D}^\prime})$, $\varepsilon>0$, and construct $f_k,f_k^\prime$}
Let $D_0=2g\sum_{x \in X}x$ and $\overline{D}_0=(D_0,(g_{D_0,v})_{v\in M_K})$ be an adelic divisor induced outside $S$ by the Zariski closure of $D_0$ in $\mathcal C_S$, with $g_{D_0,v}<-1$ on $E_v$ at every $v \in S$.

Choose $0<\rho<\operatorname{abs}(h_{\overline{D}^\prime})$. By the arithmetic Nakai--Moishezon theorem \cite[Theorem 1.5]{Zha92}, for any $k \gg 0$, there exists a $K$-basis $\big\{ \varphi_i \big\}$ of $H^0(kD^\prime+D_0)$ such that
\begin{itemize}
  \item at every $v \in S$, $\log|\varphi_i|_v \leq k g_{D^\prime,v} + g_{D_0,v} - k\rho/\#S$ for every $i$,
  \item at every $v \notin S$, $\log|\varphi_i|_v \leq k g_{D^\prime,v} + g_{D_0,v}$ for every $i$.
\end{itemize}

Choose a regular closed point $p_0$ in the special fiber of $\mathcal{C}$ over $v_0$ that is away from the reduction of $\operatorname{supp}(D) \cup X$. Choose $\varepsilon>0$. For suitable $k \gg 0$, Proposition \ref{prop:glo-sec} yields rational functions $f_k,f_k^\prime$ on $C$ such that
\begin{itemize}
  \item $f^\prime_k$ is integral and monic at every $x \in X$ on $\mathcal{C}_S$, and $f_k-f_k^\prime \in H^0(kD^\prime+D_0)$,
  \item the pole divisor of $f_k$ is $\operatorname{div}_\infty(f_k)=k D^\prime+\sum_{x \in X} \ell_x x$ with $\ell_x \leq \frac{3 \delta^\prime}{\log c_0} \deg(D^\prime) k$,
  \item at every $v \in S$, in any connected component $A$ of $E_v$, there are at least
    \[
      (\mu_v(A)-\nobreak\varepsilon) \deg(D^\prime) k
    \]
    zeros of $f_k$ (counted with multiplicity),
  \item at $v_0$, $|f_k|_{\mathcal{C}_{k_{v_0}}}=1$ and all zeros of $f_k$ reduce to one point $p_0$,
  \item at every non-archimedean place $v \in S$, $(f_{k,v},kg_{D^\prime,v})$ satisfies the boundary dominance condition on $E_v$,
  \item at every archimedean place $v \in S$, $(f_{k,v},kg_{D^\prime,v}-1)$ satisfies the boundary dominance condition on $E_v$.
\end{itemize}

\subsection{Eisenstein perturbation - construct $h_k$}
Let $\varpi_0$ be a uniformizer in $O_{K,v_0}$. Let $E_{0}$ be the exceptional divisor of the blow-up of $\mathcal{C}$ at $p_0$. Clearly we have $\operatorname{ord}_{E_{0}}(\varpi_0)=\operatorname{ord}_{v_0}(\varpi_0)=1$.

Say $f_k-f_k^\prime=\sum a_i \varphi_i$ and $\varpi_0=\sum b_i \varphi_i$. By Lemma \ref{lem:apr-int}, there exist a constant $M>0$ (depending only on $S$) and elements $A_i,B_i \in O_{K,S}$ such that
\begin{itemize}
  \item $|A_i-a_i|_v \leq M$ at every $v \in S$,
  \item $|B_i-a_i|_v \leq M$ at every $v \in S_0=S \backslash \{v_0\}$,
  \item $|B_i-a_i-b_i|_{v_0} \leq M$ at $v_0$.
\end{itemize}
We define
\begin{itemize}
  \item $u_1=\sum (A_i-a_i) \varphi_i$ and $h_1=f_k+u_1=f_k^\prime+\sum A_i \varphi_i$,
  \item $u_2=\sum (B_i-a_i) \varphi_i$ and $h_2=f_k+u_2=f_k^\prime+\sum B_i \varphi_i$.
\end{itemize}
Let $h \in \{h_{1},h_{2}\}$. Then when $k$ is large,
\begin{itemize}
  \item all zeros of $h$ are inside $E_v$ at every $v \in M_K$. Moreover, $h$ has exactly the same number of zeros as $f_k$ in every component of $E_v$ at every $v \in S$,
  \item all zeros of $h$ reduce to $p_0$ at $v_0$,
  \item At least one of $\{h_1,h_2\}$ satisfies $\operatorname{ord}_{E_0}(h_i)=1$. Denote it by $h_k$. It is Eisenstein at $p_0$; thus $\operatorname{div}_0(h_k)$ is a single closed point $\alpha_k'$.
\end{itemize}
Choose a geometric point $\alpha_k\in C(\overline K)$ above $\alpha_k'$.

\subsubsection{All zeros of $h$ are inside $E_v$ at every $v \in M_K$. Moreover, $h$ has exactly the same number of zeros as $f_k$ in every component of $E_v$ at every $v \in S$}

Set
$
  r_k=\dim_K H^0(kD^\prime+D_0)$ and
$
  c=\frac{\rho}{\#S}
$.
By Riemann--Roch, $r_k=O(k)$. Recall that
\[
  \log|\varphi_i|_v
  \leq
  k g_{D^\prime,v}+g_{D_0,v}-kc \quad \text{and} \quad g_{D_0,v}<-1 \text{ on $E_v$}, \quad \text{at every $v\in S$}.
\]

For $u_1$, we have at every archimedean $v\in S$
\begin{align*}
  |u_1|_v
  &\leq
  \sum_i |A_i-a_i|_v|\varphi_i|_v\\
  &\leq
  Mr_k
  \exp\left(
  k g_{D^\prime,v}+g_{D_0,v}-kc
  \right).
\end{align*}
Since $r_k=O(k)$, when $k$ is large we have
\[
  \log|u_1|_v<k g_{D^\prime,v}-1 \text{ on $E_v$}.
\]
At a non-archimedean $v\in S$, we have a stronger estimate
\[
  |u_1|_v \leq M\max_i|\varphi_i|_v,
\]
and therefore
\[
  \log|u_1|_v<k g_{D^\prime,v} \text{ on $E_v$, for $k$ sufficiently large}.
\]

The same estimates hold for $u_2$ at every $v\in S_0$, since
\[
  |B_i-a_i|_v\leq M \quad \text{ at every $v \in S_0$}.
\]

At $v_0$, set
\[
  e_2=u_2-\varpi_0
  =\sum_i(B_i-a_i-b_i)\varphi_i.
\]
Since $g_{D^\prime,v_0}=0$ on $E_{v_0}$, we have
\[
  |e_2|_{v_0} \leq M\max_i|\varphi_i|_{v_0} < |\varpi_0|_{v_0}, \text{ for $k$ sufficiently large}.
\]
Consequently,
\[
  |u_2|_{v_0} = |\varpi_0+e_2|_{v_0} = |\varpi_0|_{v_0} <1 = \exp\bigl(k g_{D^\prime,v_0}\bigr) \; \text{ on $E_{v_0}$.}
\]
It follows that, for $j=1,2$,
\[
  \log|h_j-f_k|_v < k g_{D^\prime,v}-1, \; \text{ at every archimedean $v\in S$,}
\]
and
\[
  \log|h_j-f_k|_v < k g_{D^\prime,v}, \; \text{at every non-archimedean $v\in S$.}
\]
By the boundary dominance condition for $f_k$, $h_j$ and $f_k$ have the same number of zeros in every connected component of $E_v$.

The global construction gives $\ell_x>2g$ for $k$ sufficiently
large. Since
\[
  u_j\in H^0(kD^\prime+D_0), \quad D_0=2g\sum_{x\in X}x,
\]
we have
$
  \operatorname{div}_\infty(h_j)\leq \operatorname{div}_\infty(f_k).
$
All zeros of $f_k$ lie in $E_v$, and the preceding argument
produces in $E_v$ at least the same number of zeros of $h_j$ as $f_k$. Since
\[
  \deg\operatorname{div}_0(h_j)
  =
  \deg\operatorname{div}_\infty(h_j)
  \leq \deg \operatorname{div}_\infty(f_k)
  =
  \deg\operatorname{div}_0(f_k),
\]
equality must hold. Thus $h_j$ has exactly the same
number of zeros as $f_k$ in every connected component of $E_v$,
and it has no zeros outside $E_v$, for every $v\in S$.

For $v \notin S$, the estimate
\[
  \log|\varphi_i|_v
  \leq k g_{D^\prime,v}+g_{D_0,v}
\]
means exactly that each $\varphi_i$ is integral as a section of
$kD^\prime+D_0$ at $v$. Since $A_i,B_i\in O_{K,S}$, the
functions
\[
  h_1=f_k^\prime+\sum_iA_i\varphi_i,
  \quad
  h_2=f_k^\prime+\sum_iB_i\varphi_i \; \text{ are integral at $v$.}
\]
Moreover, the added terms have pole order at
most $2g<\ell_x$ at every $x\in X$, so they do not change the
leading term of $f_k^\prime$. Hence $h_1$ and $h_2$ remain
monic at every $x\in X$. Since $E_v$ is simple with respect to $(\mathcal C,X)$, it follows
that all their zeros lie in $E_v$.

\subsubsection{All zeros of $h_j$ reduce to $p_0$ at $v_0$.}
Let $\xi_{v_0}$ be the divisorial point corresponding to
the special fiber of $\mathcal C$ at $v_0$. Note that it is the boundary of the open residue disk $\operatorname{red}^{-1}(p_0)$. The preceding estimates give $|u_1|_{\xi_{v_0}}<1$ and $|u_2|_{\xi_{v_0}}<1$. Since $|f_k|_{\xi_{v_0}}=1$ and all zeros of $f_k$ are in $\operatorname{red}^{-1}(p_0)$, we may apply Rouch\'e's theorem to the perturbation $h_j=f_k+u_j$ and conclude that all zeros of $h_j$ lie in $\operatorname{red}^{-1}(p_0)$.

\subsubsection{At least one of $\{h_1,h_2\}$ satisfies $\operatorname{ord}_{E_0}(h_i)=1$}
In fact, for $k$ sufficiently large, $|u_1|_{E_0}<|\varpi_0|_{E_0}$ and $|u_2-\varpi_0|_{E_0}<|\varpi_0|_{E_0}$.
Hence
\[
  \big|(h_2-h_1)-\varpi_0\big|_{E_0}=\big|(u_2-\varpi_0)-u_1\big|_{E_0}<|\varpi_0|_{E_0}\quad\text{and}\quad \operatorname{ord}_{E_0}(h_2-h_1)=\operatorname{ord}_{E_0}(\varpi_0)=1.
\]
Since all zeros of $h_1,h_2$ reduce to $p_0$, while no pole does, we have
$\operatorname{ord}_{E_0}(h_i)\geq1$ for $i=1,2$. Therefore
\[
  1=\operatorname{ord}_{E_0}(h_2-h_1)
  \geq
  \min\bigl\{\operatorname{ord}_{E_0}(h_1),
  \operatorname{ord}_{E_0}(h_2)\bigr\},
\]
and consequently at least one of $h_1,h_2$ has order exactly $1$ along $E_0$. Denote it by $h_k$.
It is Eisenstein at $p_0$: all zeros of $h_k$ reduce to $p_0$ at $v_0$, no pole of $h_k$ reduces to $p_0$, $\operatorname{ord}_{\mathcal{C}_{k_{v_0}}}(h_k)=0$, and $\operatorname{ord}_{E_0}(h_k)=1$. As a result, $\operatorname{div}_0(h_k)$ is a single closed point $\alpha_k'$.

\subsection{Take the limit and conclude}
Note that
\[
  \deg(\alpha_k')=k \deg(D^\prime) + \sum_{x \in X} \ell_x \deg(x) \leq k \deg(D^\prime) \left( 1+ C_X \delta^\prime \right), \text{ where } C_X \coloneqq \frac{3 \sum_x \deg(x)}{\log c_0}.
\] By construction, we have in every component $A$ of $E_v$ at every $v \in S_0$,
\[
  \delta_{\alpha_k,v}(A)=\frac{\sum_{\beta\in\alpha'_{k,v}\cap A}\deg(\beta)}{\deg(\alpha_k')} \geq \frac{\mu_v(A)-\varepsilon}{1+C_X\delta^\prime}.
\]

Thus, if $N_v$ denotes the (finite) number of components of $E_v$, then
\[
  \sum_A \Big| \delta_{\alpha_k,v}(A)-\mu_v(A) \Big|
  = 2 \sum_{\mu_v(A)>\delta_{\alpha_k,v}(A)}
  \Big( \mu_v(A)-\delta_{\alpha_k,v}(A) \Big)
  \leq 2\frac{C_X\delta^\prime}{1+C_X\delta^\prime}
  + 2N_v \varepsilon.
\]

Denote by
$\mathbb E^{(0)}=(E_v^{(0)})_{v\in M_K}$ and $\mu_v^{(0)}$ the adelic set
and the measures before the regularizations above. First fix, once and for all,
\[
  1<c_0<\operatorname{cap}_X(\mathbb E^{(0)}),
\]
and at every stage choose the regularizations so that the resulting adelic
set still has capacity greater than $c_0$. Consequently
$C_X=3\sum_{x\in X}\deg(x)/\log c_0$ is independent of $n$.

Choose $\delta_n,\eta_n \longrightarrow 0$ and
$\delta_n^\prime>2\delta_n$ with $\delta_n^\prime \longrightarrow 0$. At the
$n$-th stage, carry out the regularizations above, and denote by
$E_{n,v}\subseteq E_v^{(0)}$ and $\widetilde\mu_{n,v}$ the set and measure there.
\begin{itemize}
  \item At each
    archimedean $v\in T$, Proposition \ref{prop:com-exh-nor} and Corollary \ref{cor:arc-mea}
    allow us to arrange $W_1\bigl(\widetilde\mu_{n,v},\mu_v^{(0)}\bigr)\leq\delta_n$
    and to arrange for every component of $E_{n,v}$ to have diameter
    less than $\eta_n$.
  \item At each non-archimedean $v\in T$, take divisorial
    approximations as in Proposition~\ref{prop:non-mea}, and choose model neighborhoods supplied by Proposition~\ref{prop:non-exh} with component diameter less than $\eta_n$.
\end{itemize}

Let $P_{n,v}$ be the family of these
components. Then for every $v\in T$,
$
  \max_{A\in P_{n,v}}\operatorname{diam}_{\rho_v}(A)<\eta_n.
$
Since $T$ is finite, the integer
$
  N_n=\max_{v\in T}\# P_{n,v}
$
is finite.
After the regularized sets have been fixed, choose
$\varepsilon_n>0$ so that $N_n\varepsilon_n\longrightarrow0$.

For every $v \in T$, the preceding estimate gives
\[
  \sum_{A\in P_{n,v}}
  \bigl|\delta_{\alpha_n,v}(A)-\widetilde\mu_{n,v}(A)\bigr|
  \leq
  2\frac{C_X\delta_n^\prime}{1+C_X\delta_n^\prime}
  +2N_n\varepsilon_n
  \longrightarrow0.
\]

By construction, $\widetilde\mu_{n,v}\longrightarrow\mu_v^{(0)}$ weakly for every
$v\in T$. Applying Lemma \ref{lem:shrinking-partition}
below (with $Y=C_v^{\mathrm{an}}$,
$\widetilde\mu_n=\widetilde\mu_{n,v}$, $\nu_n=\delta_{\alpha_n,v}$, and
$P_n=P_{n,v}$) yields $\delta_{\alpha_n,v}\longrightarrow\mu_v^{(0)}$
weakly at every $v\in T$.
Since $\alpha'_{n,v}\subseteq E_{n,v}\subseteq E_v^{(0)}$ and $\alpha_n'/\deg(\alpha_n') \longrightarrow D$ in $\operatorname{Pic}^1(C)_\mathbb{R}$ by construction, Theorem \ref{thm:integral-point-construction} follows.

\begin{lemma} \label{lem:shrinking-partition}
  Let $(Y,d)$ be a compact metric space, and let
  $\widetilde\mu_n,\nu_n,\mu$ be probability measures on $Y$. For each $n$, let $P_n$ be a finite family of pairwise disjoint Borel sets whose union has full $\widetilde\mu_n$-measure and $\nu_n$-measure.
  Assume $\widetilde\mu_n\longrightarrow\mu$ weakly, $d_n=\max_{A\in P_n}\operatorname{diam}_d(A) \longrightarrow0$, and
  $\Delta_n=\sum_{A\in P_n}
    \bigl|\nu_n(A)-\widetilde\mu_n(A)\bigr| \longrightarrow0$. Then $\nu_n\longrightarrow\mu$ weakly.
\end{lemma}
\begin{proof}
  Since Lipschitz continuous functions are dense, we may test only Lipschitz continuous functions.
  Let $f$ be a continuous function such that $\left|f(y_1)-f(y_2)\right|\leq m d(y_1, y_2)$.
  Set $M=\sup_{Y}\left|f\right|$.
  Then,
  \[
    \left|\int f\,\mathrm d \nu_n-\int f\,\mathrm d \widetilde\mu_n\right|\leq md_n+2M\Delta_n.
  \]
  Therefore,
  \[
    \lim_{n\to \infty}\int f\,\mathrm d \nu_n=\lim_{n\to \infty}\int f\,\mathrm d \widetilde\mu_n=\int f\,\mathrm d \mu
  \]
\end{proof}
\section{Proof of main theorems}
\label{sec:main-theorem-proofs}

\subsection{Proof of Theorem \ref{thm:prescribed-places}}
In this section, we prove the main theorem.

\begin{theorem} \label{thm:adelic-sos-main-text}
  Let $K,C,X,U$ be as in Theorem~\ref{thm:adelic-sos}.
  Let $\mathbb{E}=(E_v)_{v\in M_K}$ be an admissible adelic set on $U$ with $\operatorname{cap}_X(\mathbb{E})>1$.
  Let $S$ be a subset of places, and for each $v\in S$, let $\mu_v$ be a
  probability measure supported on $E_v$.
  The following are equivalent:
  \begin{enumerate}
    \item there exists a sequence of distinct algebraic points $(\alpha_n)$ in $U(\overline{K})$ such that $\alpha'_{n,v}\subseteq E_v$ for every $n$ and every $v\in M_K$, and $\delta_{\alpha_n,v}\longrightarrow\mu_v$ weakly at every place $v\in S$,
    \item $\displaystyle \sum_{v \in S} \int \log|Q|_v \,\mathrm{d}\mu_v + \sum_{v \notin S} \log\lVert Q\rVert_{E_v, \mathrm{sup}}  \geq 0$ for all $0 \neq Q \in H^0(U, \mathcal O_U)$.
  \end{enumerate}
\end{theorem}
If one cares not only about the local measures, but also about the asymptotic behavior of the geometric divisor class, a variant of the theorem applies.

\begin{theorem}\label{thm:adelic-sos-fixed-class-main-text}
  Let $K,C,X,U$ be as in Theorem~\ref{thm:adelic-sos}.
  Let $\mathbb E=(E_v)_{v\in M_K}$ be an admissible adelic set on $U$
  with $\operatorname{cap}_X(\mathbb E)>1$.
  Let $S$ be a subset of places, and for each $v\in S$, let $\mu_v$ be a
  probability measure supported on $E_v$.
  Fix $D\in\operatorname{Pic}^1(C)_{\mathbb R}$.
  The following are equivalent:
  \begin{enumerate}
    \item there exists a sequence of distinct algebraic points
      $(\alpha_n)$ in $U(\overline K)$ such that
      \[
        \alpha'_{n,v}\subset E_v \qquad\text{for every }n\text{ and every }v\in M_K,
      \]
      \[
        \delta_{\alpha_n,v}\longrightarrow\mu_v \quad\text{weakly for every }v\in S,
      \]

      \[
        \text{and } \frac{\alpha_n'}{\deg(\alpha_n')} \longrightarrow D \quad\text{in }\operatorname{Pic}^1(C)_{\mathbb R},
      \]

    \item for every flat adelic divisor
      $\overline F=(F,(g_{F,v})_{v\in M_K})$ such that $F|_U$ is effective, we have
      \[
        \sum_{v\in S}\int_{E_v}g_{F,v}\,\mathrm d\mu_v + \sum_{v\notin S} \inf_{z\in E_v}g_{F,v}(z) \leq D\cdot\overline F.
      \]
  \end{enumerate}
\end{theorem}

The key is to use Sion's minimax theorem \cite[Theorem 3.4]{Sio58} to determine the measure at each place.
In particular, we prove the following result.
\begin{proposition}\label{prop:complete-local-measures}
  Let $S$ be a subset of places, and for each $v\in S$, let $\mu_v$ be a
  probability measure supported on $E_v$.
  Fix $D\in\operatorname{Pic}^1(C)_{\mathbb R}$.
  Assume that for every flat adelic divisor
  $\overline F=(F,(g_{F,v})_{v\in M_K})$ such that $F|_U$ is effective, we have
  \[
    \sum_{v\in S}\int_{E_v}g_{F,v}\,\mathrm d\mu_v + \sum_{v\notin S} \inf_{z\in E_v}g_{F,v}(z) \leq D\cdot\overline F.
  \]
  Then, there exist probability measures $\mu_v$ supported on $E_v$ for each $v\notin S$, such that for every flat adelic divisor
  $\overline F=(F,(g_{F,v})_{v\in M_K})$ such that $F|_U$ is effective, we have
  \[
    \sum_{v\in M_K}\int_{E_v}g_{F,v}\,\mathrm d\mu_v\leq D\cdot\overline F.
  \]
\end{proposition}
\begin{proof}
  First, for each non-archimedean $v\notin S$ with the boundary $\partial E_v$ a point, let $\mu_v$ be the Dirac measure at $\partial E_v$.
  Then, $\inf_{z\in E_v}g_{F,v}(z)=\int_{E_v}g_{F,v}\,\mathrm d\mu_v$ by the maximum principle.
  Since for all but finitely many places $v$, the boundary $\partial E_v$ is a point, enlarging $S$ by these places, we may assume that $M_K\setminus S$ is a finite set.

  We now treat the case of finitely many test divisors.
  Let $(\overline F_i=(F_i,(g_{F_i,v})_{v\in M_K}))_{1\leq i\leq n}$ be finitely many flat adelic divisors such that each $F_i|_U$ is effective.
  For a real number $M$ and each $v\notin S$, set
  $g^{(M)}_{F_i,v}=\min\{g_{F_i,v},M\}$.
  Set
  \[
    \Delta_n=\left\{\mathbf r=(r_i)_{1\leq i\leq n}\in\mathbb R^n
    \mid r_i\geq0,\ \sum_{i=1}^n r_i=1\right\}.
  \]
  Let
  \[
    \mathscr P=\prod_{v\notin S}\mathscr P(E_v),
  \]
  where $\mathscr P(E_v)$ is the space of probability measures on $E_v$ with the weak topology.
  For $\mathbf r\in\Delta_n$ and
  $\boldsymbol\mu=(\mu_v)_{v\notin S}\in\mathscr P$, set
  \[
    f_{n,M}(\mathbf r,\boldsymbol\mu)
    \coloneqq\sum_{i=1}^n r_i\bigg(
    \sum_{v\in S}\int_{E_v}g_{F_i,v}\,\mathrm d\mu_v
    +\sum_{v\notin S}\int_{E_v}g^{(M)}_{F_i,v}\,\mathrm d\mu_v
    -D\cdot\overline F_i\bigg).
  \]
  For each $\mathbf r\in\Delta_n\cap\mathbb Q^n$ and
  $\overline F_{\mathbf r}=\sum_i r_i\overline F_i$, by the hypothesis,
  \[
    \begin{aligned}
      \inf_{\boldsymbol\mu\in\mathscr P}f_{n,M}(\mathbf r,\boldsymbol\mu)
      &=\sum_{v\in S}\int_{E_v}g_{F_{\mathbf r},v}\,\mathrm d\mu_v+\sum_{v\notin S}\inf_{z\in E_v}\sum_{i=1}^n r_i g^{(M)}_{F_i,v}(z)
      -D\cdot\overline F_{\mathbf r}\\
      &\leq\sum_{v\in S}\int_{E_v}g_{F_{\mathbf r},v}\,\mathrm d\mu_v+\sum_{v\notin S}\inf_{z\in E_v}g_{F_{\mathbf r},v}(z)-D\cdot\overline F_{\mathbf r}\leq0.
    \end{aligned}
  \]
  Since the left-hand side is continuous in $\mathbf r$, the same inequality holds for every $\mathbf r\in\Delta_n$.
  Since $\Delta_n$ and $\mathscr P$ are compact and convex, and $f_{n,M}$ is continuous and affine in each variable separately, by Sion's minimax theorem \cite[Theorem 3.4]{Sio58}, we have
  \[
    \inf_{\boldsymbol\mu\in\mathscr P} \sup_{\mathbf r\in\Delta_n}f_{n,M}(\mathbf r,\boldsymbol\mu) = \sup_{\mathbf r\in\Delta_n}\inf_{\boldsymbol\mu\in\mathscr P}f_{n,M}(\mathbf r,\boldsymbol\mu)\leq0.
  \]
  Therefore, there exists $\boldsymbol\mu_{n,M} = (\mu_{v,n,M})_{v\notin S}\in\mathscr P$ such that, for each $1\leq i\leq n$,
  \[
    \sum_{v\in S}\int_{E_v}g_{F_i,v}\,\mathrm d\mu_v +\sum_{v\notin S} \int_{E_v} g^{(M)}_{F_i,v}\,\mathrm d\mu_{v,n,M} \leq D\cdot\overline F_i.
  \]

  Since $\mathscr P$ is compact and first countable, take $\boldsymbol\mu_n=(\mu_{v,n})_{v\notin S}$ to be a limit of $\boldsymbol \mu_{n,M_j}$ for some $M_j\to+\infty$.
  By the monotone convergence theorem and the lower semicontinuity of $g_{F_i, w}$ on $E_w$, for each $1\leq i\leq n$, we have
  \[
    \begin{aligned}
      \sum_{v\notin S}\int_{E_v}g_{F_i,v}\,\mathrm d\mu_{v,n}
      &\leq \liminf_{j\to\infty} \sum_{v\notin S} \int_{E_v}g^{(M_j)}_{F_i,v}\,\mathrm d\mu_{v,n}\\
      &\leq \liminf_{j\to\infty}\liminf_{l\to\infty} \sum_{v\notin S} \int_{E_v}g^{(M_j)}_{F_i,v}\,\mathrm d\mu_{v,n,M_l}\\
      &\leq \liminf_{l\to\infty} \sum_{v\notin S} \int_{E_v}g^{(M_l)}_{F_i,v} \,\mathrm d\mu_{v,n,M_l}.
    \end{aligned}
  \]
  Thus, for each $1\leq i\leq n$,
  \[
    \sum_{v\in S}\int_{E_v}g_{F_i,v}\,\mathrm d\mu_v
    +\sum_{v\notin S}\int_{E_v}g_{F_i,v}\,\mathrm d\mu_{v,n}
    \leq D\cdot\overline F_i.
  \]

  Since $C$ has countably many closed points, the set of degree zero divisors on $C$ that are effective on $U$ is countable.
  Take an enumeration $(F_i)_{i\geq1}$.
  Choose a flat lift $\overline F_i=(F_i,(g_{F_i,v})_{v\in M_K})$ for each $i\geq1$.
  Applying the preceding construction to the first $n$ lifts gives $\boldsymbol\mu_n\in\mathscr P$.
  Take $\boldsymbol\mu=(\mu_v)_{v\notin S}$ to be a limit of
  $\boldsymbol\mu_{n_j}$ for some $n_j\to\infty$.
  For every fixed $i$, lower semicontinuity gives
  \[
    \sum_{v\notin S}\int_{E_v}g_{F_i,v}\,\mathrm d\mu_v \leq \liminf_{j\to\infty} \sum_{v\notin S}\int_{E_v}g_{F_i,v}\,\mathrm d\mu_{v,n_j}.
  \]
  Hence,
  \[
    \sum_{v\in M_K}\int_{E_v}g_{F_i,v}\,\mathrm d\mu_v\leq D\cdot\overline F_i.
  \]
  Every flat adelic divisor $\overline F$ with $F|_U$ effective has geometric part $F_i$ for some $i$, and differs from $\overline F_i$ by local constants, which change both sides of the inequality by the same amount.
  The proposition follows.
\end{proof}

We next remove the divisor class from the hypothesis.
\begin{proposition}\label{prop:choose-divisor-class}
  Let $\mathbb E=(E_v)_{v\in M_K}$ be an admissible adelic set on $U$.
  Let $S\subseteq M_K$, and for each $v\in S$, let $\mu_v$ be a probability measure supported on $E_v$.
  Assume that
  \[
    \sum_{v\in S}\int_{E_v}\log|Q|_v\,\mathrm d\mu_v
    +\sum_{v\notin S}\log\lVert Q\rVert_{E_v,\mathrm{sup}}\geq0
    \qquad\text{for every }0\neq Q\in H^0(U,\mathcal O_U).
  \]
  Then there exists $D\in\operatorname{Pic}^1(C)_{\mathbb R}$ such that, for every flat adelic divisor $\overline F=(F,(g_{F,v})_{v\in M_K})$ with $F|_U$ effective, we have
  \[
    \sum_{v\in S}\int_{E_v}g_{F,v}\,\mathrm d\mu_v
    +\sum_{v\notin S}\inf_{z\in E_v}g_{F,v}(z)
    \leq D\cdot\overline F.
  \]
\end{proposition}

\begin{proof}
  Fix $D_0\in\operatorname{Pic}^1(C)_{\mathbb R}$.
  Set $V=\operatorname{Pic}^0(C)_{\mathbb R}$.
  Consider the function $f: V\to \mathbb R\bigcup\{+\infty\}$ such that
  \[
    f(L)=\sup_{\overline F}\left\{\sum_{v\in S}\int_{E_v}g_{F,v}\,\mathrm d\mu_v+\sum_{v\notin S} \inf_{z\in E_v}g_{F,v}(z) -D_0\cdot\overline F\right\}
  \]
  where $F$ is taken among all $\mathbb R$-divisors that are $\mathbb R$-linearly equivalent to $L$ and are effective on $U$, and $\overline F$ is a flat lift of $F$.
  By the assumption, $f(0)=0$.
  It is easy to see that for each $r>0$, $f(rL)=rf(L)$.
  Also, note that the function $\overline F\mapsto \sum_{v\in S}\int_{E_v}g_{F,v}\,\mathrm d\mu_v+\sum_{v\notin S} \inf_{z\in E_v}g_{F,v}(z) -D_0\cdot\overline F$ is superadditive.
  If $f(L)=+\infty$, $f(0)\geq f(-L)+f(L)=\infty$ gives a contradiction.
  Thus, $f$ takes finite values and is a superlinear function.
  By the Hahn--Banach theorem, there is a linear function $\lambda:V\to \mathbb R$ such that $\lambda(L)\geq f(L)$.

  By the Mordell-Weil theorem, $V$ is finite-dimensional.
  Moreover, the N\'eron--Tate height pairing is nondegenerate.
  Therefore, there exists $D_1\in \operatorname{Div}^0(C)_{\mathbb R}$ such that $\lambda(L)=-2[K:\mathbb Q]\langle D_1, L\rangle_{NT}$.
  Take $D=D_0+D_1$. The proposition follows.
\end{proof}
Now we are ready to prove the theorems.
\begin{proof}[Proof of Theorem~\ref{thm:adelic-sos-fixed-class-main-text}]
  $(1)\Longrightarrow(2)$.
  Fix a flat adelic divisor $\overline F$ with $F|_U$ effective.
  We have
  \[
    \begin{aligned}
      \lim_{n\to\infty}h_{\overline F}(\alpha_n)&=\lim_{n\to\infty}\overline F\cdot \alpha_n'/\deg(\alpha_n')\\
      &=\lim_{n\to\infty}\overline F\cdot (\alpha_n'/\deg(\alpha_n')-D) + \overline F\cdot D=\overline F\cdot D.
    \end{aligned}
  \]
  Also, we have
  \[
    \begin{aligned}
      \overline F\cdot D=\lim_{n\to\infty}h_{\overline F}(\alpha_n)&=\lim_{n\to\infty}\sum_{v\in M_K}\int_{E_v}g_{F, v}\,\mathrm d\delta_{\alpha_n, v}\\
      &\geq \liminf_{n\to\infty}\sum_{v\in S}\int_{E_v}g_{F, v}\,\mathrm d\delta_{\alpha_n, v} + \sum_{v\notin S}\inf_{z\in E_v}g_{F,v}(z)\\
      &\geq \sum_{v\in S}\int_{E_v}g_{F, v}\,\mathrm d\mu_v + \sum_{v\notin S}\inf_{z\in E_v}g_{F,v}(z).
    \end{aligned}
  \]
  Note that the last inequality holds since each $g_{F, v}$ is lower semicontinuous on $E_v$ and
  \[
    \inf_{z\in E_v}g_{F,v}(z)=0
  \]
  for all but finitely many $v$.
  This gives the inequality in $(2)$.

  $(2)\Longrightarrow(1)$.
  By Proposition~\ref{prop:complete-local-measures}, we can choose probability measures $\mu_v$ on $E_v$ for $v\notin S$ such that
  \[
    \sum_{v\in M_K}\int_{E_v}g_{F,v}\,\mathrm d\mu_v \leq D\cdot\overline F
  \]
  for every flat adelic divisor $\overline F$ with $F|_U$ effective.
  Theorem~\ref{thm:integral-point-construction} now gives the required sequence, with weak convergence at every place
  and normalized divisor classes converging to $D$.
\end{proof}

\begin{proof}[Proof of Theorem~\ref{thm:adelic-sos-main-text}]
  $(1)\Longrightarrow(2)$.
  Fix $0\neq Q\in H^0(U,\mathcal O_U)$.
  Since the points $\alpha_n$ are distinct, we have $Q(\alpha_n)\neq0$ for all sufficiently large $n$.
  For all such $n$, the product formula gives
  \[
    \sum_{v\in M_K}\int_{E_v}\log|Q|_v\,\mathrm d\delta_{\alpha_n,v}=0.
  \]
  Since each $\log\left|Q\right|_v$ is upper semicontinuous on $E_v$ and $\log\lVert Q\rVert_{E_v,\mathrm{sup}}=0$ for all but finitely many $v$, we obtain
  \[
    \begin{aligned}
      0=\lim_{n\to\infty}\sum_{v\in M_K}\int_{E_v}\log|Q|_v\,\mathrm d\delta_{\alpha_n,v}\\
      &\leq \limsup_{n\to\infty}\sum_{v\in S}\int_{E_v}\log|Q|_v\,\mathrm d\delta_{\alpha_n, v} + \sum_{v\notin S}\log\lVert Q\rVert_{E_v,\mathrm{sup}}\\
      &\leq \sum_{v\in S}\int_{E_v}\log|Q|_v\,\mathrm d\mu_v + \sum_{v\notin S}\log\lVert Q\rVert_{E_v,\mathrm{sup}}.
    \end{aligned}
  \]
  We get the inequality in $(2)$.

  $(2)\Longrightarrow(1)$.
  Proposition~\ref{prop:choose-divisor-class} gives a class $D\in\operatorname{Pic}^1(C)_{\mathbb R}$ satisfying condition $(2)$ of
  Theorem~\ref{thm:adelic-sos-fixed-class-main-text}.
  The theorem follows.
\end{proof}

\subsection{Proof of Theorem \ref{thm:archimedean-main}}\label{sec:archimedean-main-proof}
We prove an $\mathbb R$-neighborhood version of Theorem \ref{thm:archimedean-main}.
Let $\mathcal U/\mathbb Z$ be a quasi-projective arithmetic surface whose generic fiber is normal but not projective.
For a compact set $E$ in $\mathcal U(\mathbb C)$ an $\mathbb R$-neighborhood of $E$ is the union of a neighborhood of the closure of $E\setminus \mathcal U(\mathbb R)$  and a neighborhood of $E\bigcap \mathcal U(\mathbb R)$.
We say that an $\mathbb R$-neighborhood $E'$ of $E$ is \emph{strictly bigger} than $E$ if either $E'\setminus E$ contains an interior point or $(E'\setminus E)\bigcap \mathcal U(\mathbb R)$ contains an interior point in $\mathcal U(\mathbb R)$.
If $E$ is not a union of some connected components of $\mathcal U(\mathbb R)$, the strictly bigger assumption holds automatically.
Otherwise the assumption is necessary, to ensure the capacity is strictly greater than $1$.
For example, on $\overline{\mathbb C}\setminus \{\pm i\}$, the real line has capacity exactly $1$.
It is unknown whether the Fekete--Szeg\H{o} theorem holds in the capacity $1$ case.
\begin{theorem}\label{thm:adelic-sos-archimedean-main-text}
  Let $\mathcal U/\mathbb Z$ be a normal quasi-projective arithmetic surface whose generic fiber is not projective.
  Let $\mu$ be a probability measure compactly supported in $\mathcal U(\mathbb C)$.
  Assume that $\mu$ is invariant under complex conjugation.
  Let $E$ be a strictly bigger compact $\mathbb R$-neighborhood of the support $\operatorname{supp}(\mu)$ of $\mu$ which is stable under complex conjugation.
  Then, the following are equivalent:
  \begin{enumerate}
    \item there exists a sequence $\{\alpha_n\}$ of distinct algebraic integral points on $\mathcal U$ whose Galois orbits lie in $E$, such that the Dirac measures $\delta_{\alpha_n, \mathbb C}$ of the Galois orbits $\alpha_{n, \mathbb C}'$ of $\alpha_n$ in $\mathcal U(\mathbb C)$ converge weakly to $\mu$,
    \item for each $0\neq Q\in H^0(\mathcal{U}, \mathcal{O}_{\mathcal{U}})$, we have
      \[
        \int_{\mathcal U(\mathbb{C})}\log\left|Q\right| \,\mathrm{d}\mu\geq 0.
      \]
  \end{enumerate}
\end{theorem}
\begin{proof}
  First, we reduce to the adelic setting.
  By taking finitely many blow-ups on special fibers, we may assume that $\mathcal U$ is an open subscheme of a regular arithmetic surface $\mathcal C$.
  This does not change the theorem since these blow-ups do not affect algebraic integral points by the valuative criterion of properness, and do not affect the ring of regular functions since $\mathcal U$ is normal.
  Let $C$ be the generic fiber of $\mathcal C$.
  By Stein factorization, $C$ is geometrically integral over a number field $K$.
  We also normalize $E$.
  Recall that $E$ contains an open neighborhood $E_{\mathbb C}$ of the closure of $\operatorname{supp}(\mu)\setminus \mathcal U(\mathbb R)$.
  The open neighborhood $E_{\mathbb C}$ can be assumed to be a finite union of connected open sets by compactness.
  For the same reason, $E$ contains a neighborhood $E_{\mathbb R}$ of $\operatorname{supp}(\mu)\bigcap \mathcal U(\mathbb R)$ which is a finite union of open intervals.
  Since $E$ is strictly bigger than $\operatorname{supp}(\mu)$, we may assume that $E_{\mathbb C}\bigcup E_{\mathbb R}$ is as well.
  Replace $E$ by the closure of $E_{\mathbb C}\bigcup\sigma(E_{\mathbb C})\bigcup E_{\mathbb R}$ where $\sigma$ is complex conjugation.

  Now, $\mathcal U(\mathbb C)$ is a finite union of connected complex manifolds.
  Moreover, $\mathcal U(\mathbb C)$ modulo complex conjugation is the union $\bigcup_{v\in M_{K, \infty}}U_v^{\mathrm{an}}$.
  Let $\lambda$ be the map $\mathcal U(\mathbb C)\to \bigcup_{v\in M_{K, \infty}}U_v^{\mathrm{an}}$.
  By the Dirichlet unit theorem, testing $(2)$ with units in $O_K^\times$ gives that $\mu$ has the same mass on each connected complex manifold.
  Set $\mu_v=[K:\mathbb Q]\lambda_*{\mu}|_{U_v^{\mathrm{an}}}/[K_v:\mathbb R]$.
  Set $E_v=\lambda(E)|_{U_v^{\mathrm{an}}}$ for $v$ archimedean.
  Set $E_v=\operatorname{red}_v^{-1}(\mathcal U_{k_v})$ for $v$ non-archimedean.
  Then, $(E_v)_{v\in M_K}$ is admissible.

  Since the class group of $O_K$ is finite, for each non-zero regular function $Q$ on $U$, there is a positive integer $n$ and a non-zero element $\beta\in K^\times$ such that $\lVert Q^n/\beta\rVert_{E_v,\mathrm{sup}}=1$ for each non-archimedean $v$.
  Therefore, $(2)$ is equivalent to
  \[
    \sum_{v \in S} \int \log|Q|_v \,\mathrm{d}\mu_v + \sum_{v \notin S} \log\lVert Q\rVert_{E_v, \mathrm{sup}}  \geq 0\text{ for all } 0 \neq Q \in H^0(U, \mathcal O_U).
  \]

  Now, we translate to the case $S=M_{K, \infty}$, $\{\mu_v\}_{v\in M_{K, \infty}}$ and $(E_v)_{v\in M_K}$ in Theorem~\ref{thm:adelic-sos-main-text}.
  The part $(1)\Longrightarrow(2)$ just follows from $(1)\Longrightarrow(2)$ in Theorem~\ref{thm:adelic-sos-main-text}.
  For the other half, by Theorem~\ref{thm:adelic-sos-main-text}, we need to prove that the capacity of $(E_v)_{v\in M_K}$ is strictly greater than $1$.

  Take $v_0$ such that $E_{v_0}\neq \operatorname{supp}(\mu_{v_0})$.
  Recall that $U_{v_0}^\mathrm{an}\setminus U_{v_0}(\mathbb R)$ is connected.
  If $E_{v_0}\subset U_{v_0}(\mathbb R)$, $U_{v_0}^\mathrm{an}\setminus E_{v_0}$ is connected.
  Thus, the requirement in Lemma~\ref{lem:gre-def} holds automatically.
  Otherwise each connected component of $U_{v_0}^\mathrm{an}\setminus E_{v_0}$ has a boundary point not in $U_{v_0}(\mathbb R)$, and the requirement in Lemma~\ref{lem:gre-def} holds.
  Therefore, the capacity is strictly greater than $1$ by Lemma~\ref{lem:gre-def}.
  The theorem follows.
\end{proof}

\subsection{Proof of GVF theorems}
First, we prove Theorem \ref{thm:int-existential-closedness-t-0}.
\begin{theorem}\label{thm:int-existential-closedness-t-0-main-text}
  Let $K,C,X,U$ be as in Theorem~\ref{thm:adelic-sos}.
  Let $\mathbb{E}=(E_v)_{v\in M_K}$ be an admissible adelic set on $U$ with $\operatorname{cap}_X(\mathbb{E})>1$. Let $\ell$ be a normalized GVF functional on $K(C)$ that is integral with respect to $\mathbb E$. Assume $m_x=0$ for all $x \in X$.
  Then there exists a sequence of distinct algebraic points $(\alpha_n)$ in $C(\overline{K})$ such that the Galois orbit $\alpha'_{n, v} \subseteq E_v$ for each place $v \in M_K$ and $\alpha_n \longrightarrow \ell$ in the GVF topology.
\end{theorem}
\begin{proof}
  By Lemma~\ref{lem:gvf-on-curves} and Lemma~\ref{lem:gvf-point-convergence}, Theorem \ref{thm:integral-point-construction} applies.
\end{proof}
Then, we prove Theorem \ref{thm:int-existential-closedness}.
\begin{lemma} \label{lem:loose-place-truncation}
  Let $\widetilde{\mathbb{E}}=(\widetilde{E}_v)_{v\in M_K}$ be a loose admissible adelic set on $U$. Let $\ell$ be a normalized GVF functional on $K(C)$ that is integral with respect to $\widetilde{\mathbb{E}}$. There exist admissible adelic sets $\mathbb{E}_n$ and normalized GVF functionals $\ell_n$ such that
  \begin{itemize}
    \item $\mathbb{E}_n \subseteq \widetilde{\mathbb{E}}$ and $\operatorname{cap}_X(\mathbb{E}_n) \longrightarrow \infty$,
    \item $\ell_n$ is integral with respect to $\mathbb{E}_n$ and $m_x(\ell_n)=0$ for all $x \in X$,
    \item $\ell_n \longrightarrow \ell$ in the GVF topology.
  \end{itemize}
\end{lemma}

\begin{proof}
  Let $S=\{v\in M_K:\widetilde E_v\cap X_v^{\mathrm{an}}\neq\varnothing\}$ be the finite set of loose places, and put $F_v=\widetilde E_v$ for $v\in S$.
  For archimedean $v\in S$, let $V_v$ be the union of the connected open sets and the relative interiors of the real intervals defining $F_v$, with $X_v^{\mathrm{an}}$ removed.
  Choose an increasing admissible compact exhaustion
  \[
    H_{v,1}\subseteq H_{v,2}\subseteq\cdots\subseteq V_v, \qquad \bigcup_{n\geq1}H_{v,n}=V_v,\qquad \overline{V_v}=F_v.
  \]
  For non-archimedean $v\in S$, choose an increasing exhaustion of $F_v\setminus X_v^{\mathrm{an}}$ by model compact sets $H_{v,n}$.

  Consider
  \[
    E_{n,v}=
    \begin{cases}
      H_{v,n}, & v \in S,\\
      \widetilde{E}_v, & v \notin S,
    \end{cases}
    \qquad \mathbb{E}_n=(E_{n,v})_{v\in M_K}.
  \]
  Then $\mathbb{E}_n$ is admissible, $\mathbb{E}_n \subseteq \widetilde{\mathbb{E}}$, and
  \begin{equation}\label{eq:loose-capacity}
    \operatorname{cap}_X(\mathbb{E}_n)\longrightarrow+\infty.
  \end{equation}

  To justify \eqref{eq:loose-capacity}, for each $x\in X$ choose $v\in S$ and a point $a\in x_v$ lying in $F_v$.
  At an archimedean place, local admissibility provides compact arcs $K_j\subseteq V_v$ joining a fixed $b\neq a$ to points $z_j\to a$.
  Since $\eta_a$ is bounded below, Lemma \ref{lem:cap-low-bd} gives, for a constant $c>0$ independent of $j$,
  \[
    \operatorname{cap}_a(K_j)
    \geq c\operatorname{diam}_{\rho_a}(K_j)
    \geq c\rho_a(b,z_j)
    \longrightarrow+\infty.
  \]
  Each $K_j$ is eventually contained in $H_{v,n}$, so $V_a(H_{v,n})\to-\infty$.
  At a non-archimedean place, $F_v$ contains a coordinate disc about $a$; the same conclusion follows by testing the energy on its type II points $\zeta_r$, since $I_a(\delta_{\zeta_r})\to+\infty$ as $r\to0$.

  Writing $\Gamma_n=\Gamma_X(\mathbb E_n)=(c_{x,y}^{(n)})$, Galois averaging and monotonicity therefore give
  \[
    c_{x,x}^{(n)}\longrightarrow-\infty
    \quad(x\in X),\qquad
    0\leq c_{x,y}^{(n)}\leq c_{x,y}^{(1)}\quad(x\neq y).
  \]
  Using the uniform probability vector in the definition of $\operatorname{val}$, we obtain
  \[
    \operatorname{val}(\Gamma_n)
    \leq\frac1{|X|}\max_{x\in X}\sum_{y\in X}c_{x,y}^{(n)}
    \longrightarrow-\infty.
  \]
  Discarding finitely many terms, we assume that $\operatorname{cap}_X(\mathbb E_n)>1$.

  Write $\ell=(m,\{\mu_v\})$.
  Fix $\varepsilon_n\downarrow0$.
  At each archimedean $v\in S$, the measure $\mu_v$ has no atoms by Lemma \ref{lem:arch-mea-no-atoms}.
  Decompose $\mu_v$ among the finitely many closures defining $F_v$ and apply Proposition \ref{prop:mea-com-exh} to their compact exhaustions.

  After passing to a subsequence of the exhaustions, we obtain probability measures $\mu_{v,n}$ supported on $H_{v,n}$, with continuous potentials, and normalized potential differences $u_{v,n}$ satisfying
  \[
    \mathrm{dd}^cu_{v,n}=\mu_{v,n}-\mu_v,\qquad u_{v,n}\geq-\varepsilon_n,\qquad \mu_{v,n}\longrightarrow\mu_v.
  \]
  At each non-archimedean $v\in S$, take $\mu_{v,n}=\mu_v^{H_{v,n}}$ and $u_{v,n}=u_{\mu_v}^{H_{v,n}}$, the balayage of $\mu_v$ onto $H_{v,n}$ and its defect, as in Section \ref{sec:balayage}.
  Set $\mu_{v,n}=\mu_v$ for $v\notin S$.
  If $r$ is the number of archimedean places in $S$, set
  \[
    m^{(n)}(D)=m(D)+\sum_{v\in S}u_{v,n}(D_v)+r\varepsilon_n\deg(D)
    \qquad (D\in\operatorname{Div}(C)_{\mathbb R}),
  \]
  with $u_{v,n}(D_v)$ as in Section~\ref{sec:balayage}.
  Then $m^{(n)}$ is nonnegative on effective divisors.
  Green reciprocity gives the product formula; the extra constant contributes $r\varepsilon_n\deg(\operatorname{div}(f))=0$.
  The local logarithmic integrability and normalization are preserved, so these data define a normalized GVF functional $\ell^{(n)}$ integral with respect to $\mathbb E_n$.

  For a closed point $y$, let $\delta_{y,v}$ be its normalized local orbit measure.
  The lower bound for $u_{v,n}$, weak convergence and the GVF logarithmic integrability imply $\int u_{v,n}\,\mathrm d\delta_{y,v}\to0$ for every closed point $y$. Indeed, at an archimedean place the upper bound follows from upper semicontinuity of the logarithmic kernel; at a non-archimedean place use the decreasing balayage defects. For a fixed model adelic divisor $\overline A$, Green reciprocity also gives
  \[
    \ell^{(n)}(\overline A)-\ell(\overline A)=r\varepsilon_n\deg(A)+\sum_{v\in S}\int u_{v,n}\,\mathrm dc_1(\overline A)_v\longrightarrow0.
  \]
  Here the curvature measures have continuous local potentials, so the integral convergence follows again from the construction and Green reciprocity. As in Lemma \ref{lem:gvf-point-convergence}, positivity and approximation in the boundary norm give

  \begin{equation}\label{eq:balayage-convergence}
    \ell^{(n)}\longrightarrow\ell.
  \end{equation}

  For $x\in X$, put $m_x^{(n)}=m^{(n)}(x)/\deg(x)$, let $m_X^{(n)}=(m_x^{(n)})_{x\in X}$, and set
  \[
    p^{(n)}=-\Gamma_X(\mathbb{E}_n)^{-1}m_X^{(n)}\in\mathbb{R}_{\geq0}^{X}.
  \]
  If $a_n=\max_{x\in X}m_x^{(n)}$, then $(a_n)$ is bounded, since $m_x^{(n)}\to m_x$ for every $x\in X$ and $X$ is finite.
  Lemma \ref{lem:can-cap-sol}, together with \eqref{eq:loose-capacity}, gives
  \[
    0\leq p_x^{(n)}\leq
    \frac{a_n}{\log\operatorname{cap}_X(\mathbb{E}_n)}
    \longrightarrow0
  \]
  for every $x \in X$.

  For $x \in X$, let
  \[
    \overline{D}_x^{\mathbb{E}_n}
    =\left(\frac{x}{\deg(x)},\left(g_x^{E_{n,v}}\right)_{v\in M_K}\right),
    \qquad
    g_x^{E_{n,v}}(z)=\frac1{\deg(x)}\sum_{t \in x_v}\deg(t)g_{E_{n,v},t}(z),
  \]
  and let $\ell_{x,n}$ be the functional
  \(
  \overline{A} \longmapsto \ell_{x,n}(\overline{A})= \overline{D}_x^{\mathbb{E}_n}\cdot\overline{A}.
  \)
  The induction formula and the monotonicity of the equilibrium Green functions along the chosen exhaustions give
  \begin{equation}\label{eq:loose-equilibrium-bound}
    \sup_n\left|\ell_{x,n}(\overline{A})\right|<+\infty, \quad \text{for every $x \in X$ and every fixed model adelic divisor $\overline{A}$.}
  \end{equation}

  Finally, consider
  \[
    \ell_n=
    \frac{\ell^{(n)}+\sum_{x \in X}p_x^{(n)}\ell_{x,n}}
    {1+\sum_{x \in X}p_x^{(n)}}.
  \]
  It is a normalized GVF functional integral with respect to $\mathbb{E}_n$, and $m_x(\ell_n)=0$ for every $x \in X$. Equations \eqref{eq:balayage-convergence} and \eqref{eq:loose-equilibrium-bound}, together with $p_x^{(n)}\longrightarrow 0$, show that $\ell_n(\overline A)\longrightarrow\ell(\overline A)$ for every model adelic divisor $\overline A$. Since $\ell_n$ and $\ell$ are positive, approximation in the boundary norm extends this convergence to every adelic divisor, as in the proof of Lemma \ref{lem:gvf-point-convergence}. Thus $\ell_n\longrightarrow\ell$ in the GVF topology.
\end{proof}

\begin{theorem} \label{thm:int-existential-closedness-main-text}
  Let $K,C,X,U$ be as in Theorem~\ref{thm:adelic-sos}.
  Let $\widetilde{\mathbb{E}}=(\widetilde{E}_v)_{v\in M_K}$ be a loose admissible adelic set on $U$. Let $\ell$ be a normalized GVF functional on $K(C)$ that is integral with respect to $\widetilde{\mathbb{E}}$. Then there exists a sequence of distinct algebraic points $(\alpha_n)$ in $U(\overline{K})$ such that $\alpha'_{n,v} \subseteq \widetilde{E}_v$ at every place $v \in M_K$ and $\alpha_n \longrightarrow \ell$ in the GVF topology.
\end{theorem}

\begin{proof}
  Lemma \ref{lem:loose-place-truncation} and Theorem \ref{thm:int-existential-closedness-t-0-main-text}.
\end{proof}
\section{The essential minimum of height functions}
\label{sec:essential-minimum}

\begin{lemma}\label{lem:ess-min-gvf-inequality}
  Let $\overline D\in\operatorname{\widehat{Div}}(K(C)/\mathbb Z)$.
  Assume that $\deg(\widetilde{D})>0$.
  Then, $\operatorname{ess}(\overline{D})\geq \inf \Big\{ \ell(\overline{D}): \ell \text{ is a normalized GVF} \Big\}$.
\end{lemma}
\begin{proof}
  Let $(x_n)$ be a sequence of distinct algebraic points such that $h_{\overline D}(x_n)\longrightarrow\operatorname{ess}(\overline D)$.
  Then, the sequence $(x_n)$ has bounded Weil height with respect to any very ample line bundle on $C$.

  We take a countable dense subset $\mathscr F_v$ of the real-valued continuous functions on $C_v^{\mathrm{an}}$ in the $\lVert\cdot\rVert_\infty$ topology for each $v\in M_{K,\infty}$.
  Let $\mathscr D$ be the countable collection of divisors on all regular projective models of $C$ over $O_K$, and choose one model adelic lift of each member.
  Let $\mathscr S$ be the $\mathbb Q$-linear span of these lifts and the metric divisors $(0,f_v)$, where $f_v\in\mathscr F_v$ and all other Green functions are zero.

  For every model adelic divisor $\overline E$, since the sequence $(x_n)$ has bounded Weil height, $h_{\overline E}(x_n)$ is bounded.
  By a diagonal argument, after passing to a subsequence, $h_{\overline E}(x_n)$ converges for every model adelic divisor $\overline E\in\mathscr S$.

  The same approximation from above and below as in the proof of Lemma~\ref{lem:gvf-point-convergence} extends this convergence to every $\overline E\in\operatorname{\widehat{Div}}(K(C)/\mathbb Z)$.
  Thus we get a normalized GVF $\ell$ which is the limit of $x_n$.
  Then, $\ell(\overline D)=\operatorname{ess}(\overline D)$, proving the inequality.
\end{proof}

\begin{theorem} \label{thm:ess-attained-by-integral-gvf-main-text}
  Let $K,C,X,U$ be as in Theorem~\ref{thm:adelic-sos}.
  Let $\overline L$ be an adelic line bundle on $\operatorname{Spec} K(C)$ whose geometric part $\widetilde L$ has positive degree.
  Let $\mathbb{E}=(E_v)_{v\in M_K}$ be a loose admissible adelic set on $U$. Assume $c_1(\overline{L})_v \leq 0$ outside $E_v$ at every place $v \in M_K$. Then
  \[
    \operatorname{ess}(\overline{L})=\inf \Big\{ \ell(\overline{L}): \ell \text{ is a normalized GVF, integral with respect to $\mathbb{E}$} \Big\}.
  \]
\end{theorem}

\begin{proof}
  Let $\ell$ be a normalized GVF. By Lemma~\ref{lem:balayage-gvf}, its balayage $\ell^{\mathbb E}$ is integral with respect to $\mathbb E$. Lemma~\ref{lem:balayage-inequality} gives
  \[
    \ell^{\mathbb E}(\overline L)\leq\ell(\overline L).
  \]
  Taking infima and using Theorem~\ref{thm:int-existential-closedness-main-text} and Lemma~\ref{lem:ess-min-gvf-inequality}, we obtain
  \[
    \begin{aligned}
      \operatorname{ess}(\overline L)
      &\leq\inf_{\substack{\ell\text{ normalized GVF}\\\ell\text{ integral with respect to }\mathbb E}}\ell(\overline L)\\
      &\leq\inf_{\ell\text{ normalized GVF}}\ell(\overline L)
      \leq\operatorname{ess}(\overline L).
    \end{aligned}
  \]
\end{proof}

\begin{theorem}[Theorem {\ref{thm:ess-min-attained-by-integers}}] \label{thm:ess-min-attained-by-integers-main-text}
  Let $K,C,X,U$ be as in Theorem~\ref{thm:adelic-sos}.
  Let $\overline L$ be an adelic line bundle on $\operatorname{Spec} K(C)$ whose geometric part $\widetilde L$ has positive degree.
  Let $\mathbb{E}=(E_v)_{v\in M_K}$ be a loose admissible adelic set on $U$. Assume $c_1(\overline{L})_v \leq 0$ outside $E_v$ at every place $v \in M_K$. Then there exists a sequence of distinct algebraic points $(\alpha_n)$ in $U(\overline{K})$ such that $\alpha'_{n,v} \subseteq E_v$ at every place $v \in M_K$ and $h_{\overline{L}}(\alpha_n) \longrightarrow \operatorname{ess}(\overline{L})$.
\end{theorem}

\begin{proof}
  Theorem \ref{thm:ess-attained-by-integral-gvf-main-text} and Theorem \ref{thm:int-existential-closedness-main-text}.
\end{proof}
\section{The integral Chebyshev problem}
\label{sec:integer-chebyshev-problem}

\begin{theorem} \label{thm:montgomery-conj-is-true-main-text}
  Montgomery's conjecture (Conjecture~\ref{conj:montgomery}) is true: $t_{\mathbb Z}([0,1])=s^{-1}$.
\end{theorem}

\begin{proof}
  Pick your favorite prime number $p_0$. Apply Theorem \ref{thm:ess-min-attained-by-integers-main-text} to the height $h_{[0,1]}$ and the loose admissible adelic set $\mathbb{E}=(E_p)_{p\leq\infty}$, where $E_{p_0}=\mathbb{P}^{1,\mathrm{an}}_{p_0}$ is loose at $p_0$, $E_p \subseteq \mathbb{A}^{1,\mathrm{an}}_p$ is the Berkovich unit disk for all $p_0 \neq p < \infty$, and $E_\infty=[0,1]$ at $p=\infty$.
\end{proof}

\begin{remark}
  In fact, what we have proved above is stronger than Montgomery's conjecture: the essential minimum of $h_{[0,1]}$ can be attained by algebraic numbers totally contained in $[0,1]$ and integral outside $p_0$.
\end{remark}
\section{The Schur--Siegel--Smyth trace problem}
\label{sec:sss-trace-problem}

In this section, we connect our result with the Schur--Siegel--Smyth trace problem.
In particular, we answer \cite[Question 3.7(1)]{LL25}.
We also prove a result related to \cite[Question 3.7(2)]{LL25}, which rules out an earlier version of the question.

\begin{lemma}
  Fix $x\in\mathbb{R}_{\geq 0}$.
  For each $\varepsilon>0$, there exists a compactly supported probability measure
  $\mu_{x,\varepsilon}$ on $\mathbb{R}_{\geq 0}$ such that for
  each $y\in\mathbb C\setminus\{x\}$,
  \[
    \int\log\left|z-y\right|\,\mathrm{d}\mu_{x,\varepsilon}(z)
    \geq \log\left|y-x\right|-\varepsilon,
  \]

  \[
    \int\log\left|z-x\right|\,\mathrm{d}\mu_{x,\varepsilon}(z)>-\infty,
  \]
  and
  \[
    \int z\,\mathrm{d}\mu_{x,\varepsilon}(z)\leq x+\varepsilon.
  \]
\end{lemma}

\begin{proof}
  Let $0<\varepsilon_1<\varepsilon_2$.
  Let $\nu$ be the equilibrium measure on the closed interval
  $[x+\varepsilon_1,x+\varepsilon_2]$ with respect to the point $x$.
  Under the isomorphism
  $\phi:\overline{\mathbb{C}}\longrightarrow\overline{\mathbb{C}}$,
  $z\longmapsto 1/(z-x)$, the push-forward $\phi_*\nu$ is the equilibrium measure
  on $[1/\varepsilon_2,1/\varepsilon_1]$.
  Also,
  $\phi_*(U_{\nu}(z)-\log\left|z-x\right|)
    =U_{[1/\varepsilon_2,1/\varepsilon_1]}(z)
    -U_{[1/\varepsilon_2,1/\varepsilon_1]}(0)$, where $U_{[a,b]}$ denotes the potential
  of the equilibrium measure on $[a,b]$.
  It is well known that $U_{[a,b]}$ attains its infimum
  $\log(b-a)-\log 4$ and
  $U_{[a,b]}(0)=\log(a+b+2\sqrt{ab})-\log 4$ as long as $a>0$.
  Therefore,
  \begin{align*}
    \int\log\left|z-y\right|\,\mathrm{d}\nu(z)-\log\left|y-x\right|
    &=U_{\nu}(y)-\log\left|y-x\right|\\
    &\geq -\log\frac{\varepsilon_1+\varepsilon_2+2\sqrt{\varepsilon_1\varepsilon_2}}
    {\varepsilon_2-\varepsilon_1}.
  \end{align*}
  As we take $\varepsilon_2=\varepsilon$ and $\varepsilon_1$ sufficiently small, we get
  \[
    \log\frac{\varepsilon_1+\varepsilon_2+2\sqrt{\varepsilon_1\varepsilon_2}}
    {\varepsilon_2-\varepsilon_1}<\varepsilon.
  \]
\end{proof}

Let $\mu$ be a compactly supported probability measure on
$\mathbb{R}_{\geq 0}$ such that
$\displaystyle \int\log\left|Q\right|\,\mathrm{d}\mu\geq 0$ for each non-zero polynomial
$Q$ with integer coefficients.

Let $X_{\mathrm{SSS}}$ be the exceptional set for the Schur--Siegel--Smyth trace problem, i.e.
\[
  X_{\mathrm{SSS}} := \left\{ x \in \overline{\mathbb{Z}}_{\mathrm{tp}}: \frac{\operatorname{tr}(x)}{\deg(x)} < \lambda_{\mathrm{SSS}} \right\}.
\]

\begin{proposition}\label{prop:mea-min-poly}
  Let $P_\beta$ be the minimal polynomial of $\beta$, where
  $\beta \in X_{\mathrm{SSS}}$.
  Then there is a constant $c_{P_\beta}>0$, depending only on $P_\beta$, such that
  \[
    \int z\,\mathrm{d}\mu(z)
    \geq c_{P_\beta}\int\log\left|P_\beta\right|\,\mathrm{d}\mu+\lambda_{\mathrm{SSS}}.
  \]
\end{proposition}

\begin{proof}
  If $\int\log\left|P_\beta\right|\,\mathrm{d}\mu=0$, the inequality holds automatically.
  Assume $\int\log\left|P_\beta\right|\,\mathrm{d}\mu>0$.
  Let $\beta_1,\beta_2,\ldots,\beta_m$ be all conjugates of $\beta$.
  Let $\nu_\varepsilon$ be the measure
  $\sum_{i=1}^m\mu_{\beta_i,\varepsilon}/m$.
  Then for each non-zero polynomial $Q$ with integer coefficients such that
  $Q(\beta)\neq 0$, we have
  \[
    \int\log\left|Q\right|\,\mathrm{d}\nu_\varepsilon
    \geq \frac{\sum_{i=1}^m\log\left|Q(\beta_i)\right|}{m}-\deg(Q)\varepsilon
    \geq -\deg(Q)\varepsilon.
  \]
  Let $\nu'$ be the equilibrium measure on $[0,8]$.
  Then for each non-zero polynomial $Q$ with integer coefficients, we obtain
  \[
    \int\log\left|Q\right|\,\mathrm{d}\nu'\geq\deg(Q)\log 2.
  \]
  Let $\mu_\varepsilon$ be the measure
  $(\log 2\,\nu_\varepsilon+\varepsilon\nu')/(\log 2+\varepsilon)$.
  Then we have
  \[
    \int\log\left|Q\right|\,\mathrm{d}\mu_\varepsilon\geq 0 \text{ for each } Q\in \mathbb Z[x] \text{ with } Q(\beta)\neq 0
  \]
  and
  \[
    \int z\,\mathrm{d}\mu_\varepsilon(z)
    \leq
    \frac{\log 2\left(\frac{\operatorname{tr}(\beta)}{\deg(\beta)}
      +\varepsilon\right)+4\varepsilon}{\log 2+\varepsilon}.
  \]
  For sufficiently small $\varepsilon$,
  $\int z\,\mathrm{d}\mu_\varepsilon<\lambda_{\mathrm{SSS}}$.
  This also forces
  \[
    \int\log\left|P_\beta\right|\,\mathrm{d}\mu_\varepsilon< 0.
  \]

  Let $\varepsilon'>0$.
  Let $\mu'=(\varepsilon'\mu_\varepsilon+\mu)/(1+\varepsilon')$.
  Then for each non-zero polynomial $Q$ with integer coefficients such that
  $Q(\beta)\neq 0$, we have
  \[
    \int\log\left|Q\right|\,\mathrm{d}\mu'\geq 0,
  \]
  \[
    \int\log\left|P_\beta\right|\,\mathrm{d}\mu'
    =\frac{\int\log\left|P_\beta\right|\,\mathrm{d}\mu
      +\varepsilon'\int\log\left|P_\beta\right|\,\mathrm{d}\mu_\varepsilon}
    {1+\varepsilon'},
  \]
  and
  \[
    \int z\,\mathrm{d}\mu'(z)
    =\frac{\int z\,\mathrm{d}\mu(z)
      +\varepsilon'\int z\,\mathrm{d}\mu_\varepsilon(z)}{1+\varepsilon'}.
  \]
  Set
  $\varepsilon'=-(\int\log\left|P_\beta\right|\,\mathrm{d}\mu)
    /(\int\log\left|P_\beta\right|\,\mathrm{d}\mu_\varepsilon)$. Then
  $\int z\,\mathrm{d}\mu'(z)\geq\lambda_{\mathrm{SSS}}$.
  Therefore,
  \[
    \int z\,\mathrm{d}\mu(z)-\lambda_{\mathrm{SSS}}
    \geq-\varepsilon'
    \left(\int z\,\mathrm{d}\mu_\varepsilon(z)-\lambda_{\mathrm{SSS}}\right).
  \]
  The proposition follows by taking
  \[
    c_{P_\beta}=
    \frac{\int z\,\mathrm{d}\mu_\varepsilon(z)-\lambda_{\mathrm{SSS}}}
    {\int\log\left|P_\beta\right|\,\mathrm{d}\mu_\varepsilon}.
  \]
\end{proof}

Take a probability measure $\mu_{\mathrm{SSS}}$ compactly supported on $\mathbb{R}_{\geq 0}$ satisfying the polynomial test such that
\[
  \int z\,\mathrm{d}\mu_{\mathrm{SSS}}(z)=\lambda_{\mathrm{SSS}}.
\]
By \cite[Lemma 5.10]{Smi24}, such a measure exists.

\begin{proposition}\label{prop:mea-avo-pt}
  For each $\beta\in X_{\mathrm{SSS}}$, $\beta$ is not contained in the support of
  $\mu_{\mathrm{SSS}}$.
\end{proposition}

\begin{proof}
  Let $t$ be a real number such that $[0,t]$ contains the support of $\mu_{\mathrm{SSS}}$.
  Let $P_\beta$ be the minimal polynomial of $\beta$.
  Let $\varepsilon>0$.
  Let
  $\nu=\mu_{\mathrm{SSS}}|_{\mathbb{R}\setminus[\beta-\varepsilon,\beta+\varepsilon]}
    +b\mu_{[0,t]}$, where $b$ is the mass of $\mu_{\mathrm{SSS}}$ on the closed interval
  $[\beta-\varepsilon,\beta+\varepsilon]$ and $\mu_{[0,t]}$ is the equilibrium
  measure on $[0,t]$.
  Then for each $y\in\mathbb{C}$,
  \[
    \int\log\left|z-y\right|\,\mathrm{d}\nu(z) -\int\log\left|z-y\right| \,\mathrm{d} \mu_{\mathrm{SSS}}(z) =bU_{[0,t]}(y)-U_{\mu_{\mathrm{SSS}}|_{[\beta-\varepsilon,\beta+\varepsilon]}}(y)
    \geq-b\log 4.
  \]
  Also, we have
  \[
    \int\log\left|z-\beta\right|\,\mathrm{d}\nu(z)
    -\int\log\left|z-\beta\right|\,\mathrm{d}\mu_{\mathrm{SSS}}(z)
    \geq b(\log t-\log 4\varepsilon).
  \]

  Let $\nu'=(\nu+2b\mu_{[0,8]})/(1+2b)$.
  Then $\int\log\left|Q\right| \,\mathrm{d}\nu'\geq 0$ for every non-zero polynomial $Q$ with
  integer coefficients.
  Also,
  $\int\log\left|P_\beta\right| \,\mathrm{d}\nu'\geq
    b(\log t-\log\varepsilon)/(\deg(\beta)(1+2b))$.
  By Proposition~\ref{prop:mea-min-poly}, we get
  \begin{align*}
    \int z\,\mathrm{d}\nu'(z)
    &\geq c_{P_\beta}\int\log\left|P_\beta\right|\,\mathrm{d}\nu'+\lambda_{\mathrm{SSS}}\\
    &\geq c_{P_\beta}b\frac{\log t-\log\varepsilon}
    {(1+2b)\deg(\beta)}+\lambda_{\mathrm{SSS}}.
  \end{align*}
  On the other hand,
  \[
    \int z\,\mathrm{d}\nu'(z)
    \leq
    \frac{\int z\,\mathrm{d}\mu_{\mathrm{SSS}}(z)
      +b\int z\,\mathrm{d}\mu_{[0,t]}
      +2b\int z\,\mathrm{d}\mu_{[0,8]}(z)}{1+2b}
    =\frac{\lambda_{\mathrm{SSS}}+\frac{tb}{2}+8b}{1+2b}.
  \]
  Thus, if $b\neq 0$,
  \[
    \frac{c_{P_\beta}(\log t-\log\varepsilon)}{\deg(\beta)}
    \leq\frac{t}{2}+8-2\lambda_{\mathrm{SSS}}.
  \]
  When $\varepsilon$ is sufficiently small, the above inequality cannot hold.
  Therefore, $\mu_{\mathrm{SSS}}$ is zero on $(\beta-\varepsilon,\beta+\varepsilon)$ for
  $\varepsilon$ small enough.
\end{proof}

\begin{theorem} \label{thm:levenberg-question-main-text}
  Let $X \subseteq X_{\mathrm{SSS}}$ be a finite subset.
  Let $\mathcal D_X$ be the Zariski closure in $\mathbb P^1_{\mathbb Z}$ of $\infty$ and the closed points of $\mathbb P^1_{\mathbb Q}$ corresponding to the elements of $X$, and put $\mathcal U_X=\mathbb P^1_{\mathbb Z}\setminus\mathcal D_X$.
  Then there exists a sequence of distinct points $(\alpha_n)$ in $\mathcal{U}_X(\overline{\mathbb{Z}})\cap\overline{\mathbb{Z}}_{\mathrm{tp}}$ such that $ \frac{\operatorname{tr}(\alpha_n)}{\deg(\alpha_n)} \longrightarrow \lambda_{\mathrm{SSS}}$.
\end{theorem}

\begin{proof}
  We would like to apply Theorem \ref{thm:adelic-sos-archimedean-main-text} to the measure $\mu_{\mathrm{SSS}}$ compactly supported on $\mathcal{U}_X(\mathbb R)=\mathbb{R}-\cup_{\beta \in X} \beta'_{\infty}$.
  We take $E$ to be a compact neighborhood of $\operatorname{supp}(\mu_{\mathrm{SSS}})$ in $\mathcal{U}_X(\mathbb R)$ which is contained in $\mathbb{R}_+$.
  We choose $E\ne\operatorname{supp}(\mu_{\mathrm{SSS}})$.
  All we need to do is to check the condition
  \[
    \int \log|Q| \,\mathrm{d}\mu_{\mathrm{SSS}} \geq 0 \quad \text{ for all $0\neq Q \in H^0(\mathcal{U}_X,\mathcal O_{\mathcal{U}_X})$},
  \] which boils down to
  \[
    \int \log|Q| \,\mathrm{d}\mu_{\mathrm{SSS}} \geq 0 \quad \text{ for all $0\neq Q \in \mathbb{Z}[x]$}
  \] and
  \[
    \int \log|P_\beta| \,\mathrm{d}\mu_{\mathrm{SSS}} = 0 \quad \text{ for every $\beta \in X$}.
  \]

  Now, $\mu_{\mathrm{SSS}}$ naturally satisfies the first condition, and the extra tests for the minimal polynomials $P_\beta$ with $\beta \in X$ are guaranteed by Proposition \ref{prop:mea-min-poly}.
\end{proof}
By taking $X=\{0\}$ in Theorem \ref{thm:levenberg-question-main-text}, we answer \cite[Question 3.7(1)]{LL25}, asserting that $\lambda_{\mathrm{SSS}}$ is achieved by a sequence of algebraic units.
However, an earlier formulation of \cite[Question 3.7(2)]{LL25}, asserting that $\lambda_{\mathrm{SSS}}$ can be reached by equilibrium measures with respect to $0$ and $\infty$, is not true.
Let $X$ be a finite Galois-stable set of totally nonnegative algebraic integers such that for each $\alpha\in X$, $\deg(\alpha)^{-1}\operatorname{tr}(\alpha)<\lambda_{\mathrm{SSS}}$.
Let $Y$ be the set of closed points on $\mathbb{P}^1_{\mathbb{Q}}$ corresponding to algebraic integers in $X$ or $\infty$.
Let $\mathcal U_X=\mathbb P^1_{\mathbb Z}\setminus\overline Y$, where $\overline Y$ is the Zariski closure of $Y$ in $\mathbb P^1_{\mathbb Z}$.
For a compact set $\Sigma\subset [0,\infty)\setminus X$, put $\mathbb E_\Sigma=(E_{\Sigma,v})_{v\in M_{\mathbb Q}}$, where
\[
  E_{\Sigma,\infty}=\Sigma,\qquad
  E_{\Sigma,p}=\operatorname{red}_p^{-1}((\mathcal U_X)_{\mathbb F_p})\quad\text{for every prime }p,
\]
and write $\operatorname{cap}_Y(\Sigma):=\operatorname{cap}_Y(\mathbb E_\Sigma)$.
Let $\Sigma$ be a compact subset of $[0,\infty)\setminus X$ with capacity $\operatorname{cap}_{Y}(\Sigma)=1$.
Note that in this case, the Cantor matrix $\Gamma=\Gamma_Y(\mathbb E_\Sigma)=(c_{y_1, y_2})_{y_1, y_2\in Y}$ has positive off-diagonal elements.
Therefore, $\Gamma$ is negative semidefinite, and there is a unique probability vector $\mathbf{p}=(p_y)_{y\in Y}$ such that $\Gamma\mathbf{p}=0$.
This enables us to talk about the equilibrium measure $\mu_{\Sigma}$, defined as follows.
Set
\[
  \mu_{\Sigma}=\sum\limits_{y\in Y}p_y\mu_{\Sigma,y_\infty}
\]
where $\mu_{\Sigma,y_\infty}$ is the equilibrium measure of $\Sigma$ with respect to $y$.
\begin{proposition}\label{prop:mea-not-equi}
  Assume that $X_{\mathrm{SSS}}\nsubseteq X$. Then
  \[
    \int z\,\mathrm{d}\mu_{\Sigma}(z)>\lambda_{\mathrm{SSS}}.
  \]
\end{proposition}
\begin{proof}
  We prove by contradiction.
  Assume that
  \[
    \int z\,\mathrm{d}\mu_{\Sigma}(z)=\lambda_{\mathrm{SSS}}.
  \]
  Take $\beta\in X_{\mathrm{SSS}}\setminus X$.
  Let $P_\beta$ be the minimal polynomial of $\beta$.
  Then, by Proposition \ref{prop:mea-min-poly},
  \[
    \int \log\left|P_\beta(z)\right|\,\mathrm{d}\mu_{\Sigma}(z)=0.
  \]
  By Proposition \ref{prop:mea-avo-pt}, either $\beta\notin\Sigma$, or $\beta$ is contained in the polar component of $\Sigma$.
  If the latter happens, we remove the polar component; this does not affect the equilibrium measure.

  Let $P_y$ be the minimal polynomial of $y$ for each $y\in Y\setminus\{\infty\}$.
  Set $P_{\infty}=1$.
  In what follows, all potentials $U$ are normalized such that $U-\log\left|z\right|$ extends continuously to $\infty$ with value $0$.
  Then, $U_{\mu_{\Sigma,y_\infty}}\geq \log\left|P_y\right|/\deg(y)-c_{\infty, y}$ for each $y\in Y$, and the inequality is strict outside $\Sigma$.
  Write $\delta_{\beta,\infty}$ for the average of Dirac measures at every conjugate of $\beta$.
  Then,
  \begin{align*}
    \frac1{\deg(\beta)}\int \log\left|P_\beta(z)\right|\,\mathrm{d}\mu_{\Sigma}(z)
    &=\int U_{\mu_{\Sigma}}\,\mathrm{d}\delta_{\beta,\infty}\\
    &> \sum_{y\in Y}p_y\int\left(\log\left|P_y\right|/\deg(y)-c_{\infty, y}\right)\,\mathrm{d}\delta_{\beta,\infty}\\
    &=\sum_{y\in Y}p_y\int\log\left|P_y\right|/\deg(y)\,\mathrm{d}\delta_{\beta,\infty}\\
    &\geq 0.
  \end{align*}
  This leads to a contradiction.
\end{proof}
In an earlier version of \cite[Question 3.7(2)]{LL25}, they asked whether there is a set $E\subset \mathbb R_{\geq 0}$ of capacity exactly $1$ with respect to $0, \infty$ whose equilibrium measure $\mu_E$ (with respect to $\{0, \infty\}$) satisfies
\[
  \int z\,\mathrm d\mu_E(z)=\lambda_{\mathrm{SSS}}.
\]
Taking $X=\{0\}$, our result shows that this never happens.

\end{document}